\pdfoutput=1
\documentclass{article}
\usepackage[english]{babel}
\usepackage[T1]{fontenc}
\usepackage[utf8]{inputenc}
\usepackage{amsthm}
\usepackage{amsmath}
\usepackage{amssymb}
\usepackage{palatino}
\usepackage{graphicx}
\usepackage[colorlinks=true, citecolor =blue]{hyperref}
\usepackage{dsfont}
\usepackage{float}
\usepackage[all]{xy}
\usepackage[top=4cm,bottom=4cm,left=3cm,right=3cm]{geometry}
\usepackage{enumitem}
\usepackage{mathtools}
\usepackage{subcaption}
\usepackage{authblk}
\usepackage{bbm}
\usepackage{ulem}
\usepackage{mathrsfs}
\usepackage{booktabs}

\newcommand{\red}{\textcolor{red}}
\newcommand{\blue}{\textcolor{blue}}

\newcommand{\p}{\partial}
\newcommand{\eps}{\varepsilon}
\newcommand{\id}{\textup{Id}}

\newtheorem{theorem}{Theorem}
\newtheorem{lemma}[theorem]{Lemma}
\newtheorem{definition}[theorem]{Definition}

\newtheorem{proposition}[theorem]{Proposition}

\newtheorem{corollary}[theorem]{Corollary}
\usepackage{tikz}
\numberwithin{equation}{section}
\numberwithin{theorem}{section}

\newcommand{\rB}{B^{+}}
\newcommand{\lB}{B^{-}}
\newcommand{\rP}{\Gamma^{+}}
\newcommand{\lP}{\Gamma^{-}}
\newcommand{\dl}{\langle\!\langle}
\newcommand{\Bdl}{\Big\langle\!\!\Big\langle}
\newcommand{\dr}{\rangle\!\rangle}
\newcommand{\Bdr}{\Big\rangle\!\!\Big\rangle}
\newcommand{\reverse}{\mathfrak{r}}
\newcommand{\myuwave}[1]{%
  \setbox0=\hbox{#1}
  \raisebox{-2pt}{\uwave{\phantom{#1}}}
  \llap{\usebox0}
}

\author{Chenjiayue Qi\thanks{chenjiayue@ihes.fr}}
\affil{Institut des Hautes Études Scientifiques (IHES), Bures-sur-Yvette, France}

\date{}

\begin{document}

\title{Global Solution of a Functional Hamilton-Jacobi Equation associated with a Hard Sphere Gas}
\maketitle
\vspace{-25pt}
\begin{abstract}
In recent years it has been shown for hard sphere gas that, by retaining the correlation information, dynamical fluctuation and large deviation of empirical measure around Boltzmann equation could be proved, in addition to the classical kinetic limit result by Lanford. After taking low-density limit, the correlation information can be encoded into a functional Hamilton-Jacobi equation. The results above are restricted to short time. This paper establishes global-in-time construction of a solution of the Hamilton-Jacobi equation, by analyzing a system of coupled Boltzmann equations. The global solution converges to a non-trivial stationary solution of the Hamilton-Jacobi equation in the long-time limit under proper assumptions.
\end{abstract}

\paragraph{Acknowledgements}
The author is very grateful to Laure Saint-Raymond and Thierry Bodineau for their many inspiring discussions on the topic of this paper, as well as their valuable suggestions on the overall understanding of its main results.

\tableofcontents

\section{Introduction}
In the seminal work of Lanford \cite{Lanford_1975}, it is shown that the average dynamics of a hard sphere gas in the low-density limit is governed by the Boltzmann equation. The proof establishes the propagation of chaos for a hard sphere gas: dynamical correlations between different hard spheres are negligible in a certain sense. 

Since the result above could be seen as a law of large numbers, one can also look at the corresponding central limit theorem and large deviation theory. In \cite{BGSS_2023}, by retaining the correlation information between different particles, the dynamical fluctuations and large deviations of the empirical measure around Boltzmann equation are derived. 

In particular, the correlation information is encoded into the so-called cumulant generating functional $\mathcal{I}_{\eps}(t,g)$, and it is shown that after taking the low-density limit the limiting functional $\mathcal{I}(t,g)$ satisfies a functional Hamilton-Jacobi equation. The functional Hamilton-Jacobi equation could provide a direct new proof for the convergence of the empirical measure towards the solution of Boltzmann equation. It also plays an important role in establishing the dynamical fluctuations and large deviations of a hard sphere gas. All the results in \cite{Lanford_1975} and \cite{BGSS_2023} mentioned above are however restricted to short times. Here we mention the recent breakthrough \cite{Deng_Hani_Ma_2024} on extending Lanford's argument \cite{Lanford_1975} into long times.

The current paper is devoted to the construction of global-in-time solution $\mathcal{I}(t,g)$ of the limiting functional Hamilton-Jacobi equation. The construction is based on the study of the Euler-Lagrange system (coupled Boltzmann equations) associated with the Hamilton-Jacobi equation. 

In subsection \ref{sec;hard_sphere_gas} we recall the basic setting of the hard sphere gas, while introducing the formulation of the functional Hamilton-Jacobi equation. Then in subsection \ref{sec:coupled_intro}, we introduce the associated Euler-Lagrange system, i.e. the coupled Boltzmann equations. In subsection \ref{sec:main_result_subsection} we claim the main results of our paper: the global well-posedness of the coupled Boltzmann equations, and further the existence of global-in-time bounded solutions of the functional Hamilton-Jacobi equation. In subsection \ref{sec:implications}, we discuss future directions based on the current results.

\subsection{A Hamilton-Jacobi Equation for Hard Sphere Gas}\label{sec;hard_sphere_gas}
One approach to describe hard sphere gas at microscopic level is to fix the total number $N$, as well as the diameter $\eps>0$ of these identical hard spheres. The evolution for the positions $(\mathbf{x}_1^\eps,...,\mathbf{x}_N^\eps)\in \mathds{T}^{dN}$ and velocities $(\mathbf{v}_1^\eps,...,\mathbf{v}_N^\eps)\in \mathds{R}^{dN}$ of the $N$ particles, satisfies a system of ordinary differential equations (Newton's laws)
\begin{equation}
\frac{d\mathbf{x}_i^\eps}{dt}=\mathbf{v}_i^{\eps},\quad \frac{d\mathbf{v}_i^\eps}{dt}=0\quad \textup{as long as $|\mathbf{x}_i^\eps-\mathbf{x}_j^\eps|>\eps$ for arbitrary $j$ with $1\leq i\not=j\leq N$},
\label{eq:hard_sphere_law_1}
\end{equation}
with specular reflection at collsions if $|\mathbf{x}_i^{\eps}-\mathbf{x}_j^{\eps}|=\eps$
\begin{equation}\label{eq:hard_sphere_law_2}
(\mathbf{v}_i^{\eps})':=\mathbf{v}_i^{\eps}-(\mathbf{v}_i^{\eps}-\mathbf{v}_j^{\eps})\cdot\omega\omega,\quad (\mathbf{v}_j^{\eps})':=\mathbf{v}_j^{\eps}+(\mathbf{v}_i^{\eps}-\mathbf{v}_j^{\eps})\cdot\omega\omega,\quad \omega:=\frac{\mathbf{x}_i^{\eps}-\mathbf{x}_j^{\eps}}{\eps}
\end{equation}
This means after a collision with $|\mathbf{x}_i^{\eps}-\mathbf{x}_j^{\eps}|=\eps$, the velocities $(\mathbf{v}_i^{\eps},\mathbf{v}_j^{\eps})$ of the two particles will be changed into $\big((\mathbf{v}_i^{\eps})',(\mathbf{v}_j^{\eps})'\big)$. This will induce a well-defined trajectory for initial conditions of full Lebesgue measure in the canonical phase space $\mathcal{D}_N^{\eps}$
\begin{equation*}
\mathcal{D}_{N}^{\eps}:=\Big\{(\mathbf{x}_1^\eps,...,\mathbf{x}_N^\eps,\mathbf{v}_1^\eps,...,\mathbf{v}_N^\eps)\in\mathds{T}^{dN}\times \mathds{R}^{dN}:\ \forall i\not=j,\ |\mathbf{x}_i^\eps-\mathbf{x}_j^\eps|>\eps\Big\},
\end{equation*}
excluding multiple collisions and accumulation of collision times. 

This microscopic dynamics induces a Liouville equation for the probability density $W_N^{\eps}$ of the $N$ particles, where $W_{N}^{\eps}(t,X_N,V_N)$ refers to the probability density of finding $N$ hard spheres with configuration $(X_N,V_N)$ at time $t$. Here the variable $X_N$ refers to the positions of the $N$ particles $X_N:=(x_1,...,x_N)$, and the variable $V_N$ refers to the velocities of the $N$ particles $V_N:=(v_1,...,v_N)$. The Liouville equation for $W_{N}^{\eps}$ is
\begin{equation*}
\p_t W_{N}^\eps+V_N\cdot\nabla_{X_N}W_{N}^\eps=0,
\end{equation*}
with boundary condition corresponding to specular reflection. 

One can further consider the grand canonical formulation of a hard sphere gas: instead of fixing the total number of particles, we assume the total number $\mathcal{N}$ of particles to be random with a modified Poisson distribution law. For each diameter $\eps$, we fix 
a constant $\mu_\eps$ as the parameter for the modified Poisson distribution of the total number of particles. We assume that at initial time $t=0$ the probability density of having $N$ particles and configuration $(X_N,V_N)$ as the following with $z_i:=(x_i,v_i)$
\begin{equation}\label{eq:grand_canonical_measure}
\frac{1}{\mathcal{Z}^{\eps}}\frac{\mu_{\eps}^N}{N!}\prod_{i=1}^{N}f^0(z_i)\mathbbm{1}_{\mathcal{D}_{N}^{\eps}}.
\end{equation}
Here $\mathcal{Z}^{\eps}$ is the normalizing factor defined as
\begin{equation*}
\mathcal{Z}^{\eps}:=1+\sum_{N\geq 1}\frac{\mu_{\eps}^N}{N!}\int_{\mathds{T}^{dN}\times \mathds{R}^{dN}}\prod_{i=1}^{N}f^0(z_i)\mathbbm{1}_{\mathcal{D}_{N}^{\eps}}dV_NdX_N.
\end{equation*}
It is clear that we have two sources of randomness: the number $\mathcal{N}$ of particles is random, while given the total number $N$, the configuration $(X_{N},V_{N})$ is also random. For each sample $(X_{N},V_{N})$, it will follow the evolution law given by equations \eqref{eq:hard_sphere_law_1} and \eqref{eq:hard_sphere_law_2}. Thus at each time $t\geq 0$ we would have the distribution law for $(X_{N},V_{N})$.

In the Boltzmann-Grad scaling, we impose $\mu_{\eps}=\eps^{-(d-1)}$ to ensure the number of collisions per particle is of order $1$ per unit time \cite{Grad_1949}. This scaling implies the asymptotic below for $\mathds{E}_{\eps}(\mathcal{N})$, where $\mathds{E}_\eps$ is the expectation upon the probability measure given by \eqref{eq:grand_canonical_measure}
\begin{equation*}
\lim_{\eps\rightarrow 0}\mathds{E}_{\eps}(\mathcal{N})\eps^{d-1}=1.
\end{equation*}
A central object in the study of hard sphere gas is the empirical measure $\pi_t^{\eps}$, defined as
\begin{equation*}
\pi_t^{\eps}:=\frac{1}{\mu_{\eps}}\sum_{i=1}^{\mathcal{N}}\delta_{\mathbf{z}_i^{\eps}(t)},
\end{equation*}
where $\delta_{\mathbf{z}_i^\eps(t)}$ means the Dirac mass at $\mathbf{z}_i^\eps(t)\in \mathds{T}^d\times \mathds{R}^d$. To encode the correlation information, the cumulant generating functional for hard sphere gas with diameter $\eps$ is introduced
\begin{equation}\label{eq:cumulant_intro}
\Lambda^{\eps}(t,h):=\frac{1}{\mu_{\eps}}\log\mathds{E}_{\eps}\Big[\exp\Big(\mu_{\eps}\pi_t^{\eps}(h)\Big)\Big]=\frac{1}{\mu_{\eps}}\log\mathds{E}_{\eps}\Big[\exp\Big(\sum_{i=1}^{\mathcal{N}}h(\mathbf{z}_i^{\eps}(t))\Big)\Big],
\end{equation}
where $h$ is a test function with variables $(x,v)$. By taking the Boltzmann-Grad limit $\eps\rightarrow 0$ with $\mu_\eps\eps^{d-1}=1$, the functional $\Lambda^{\eps}(t,h)$ should converge to a limiting cumulant generating functional $\Lambda(t,h)$. This functional satisfies the following functional Hamilton-Jacobi equation, with the functional derivative $\frac{\p \Lambda(t,h)}{\p h(t)}$ taken as a measure in $x$ and $v$ for each $t$
\begin{equation}\label{eq:functional_HJ_transport}
\p_t\Lambda(t,h)=\mathcal{H}\Big(\frac{\p \Lambda(t,h)}{\p h(t)},h(t)\Big)+\int v\cdot\nabla_x h \frac{\p \Lambda(t,h)}{\p h(t)}dvdx,\quad \Lambda(0,g)=\int (e^{g(0)}-1)f^0 dvdx.
\end{equation}
In this equation, the Hamiltonian $\mathcal{H}(\varphi,p)$ is defined as
\begin{equation}\label{eq:functional_hamiltonian}
\mathcal{H}(\varphi,p)=\frac{1}{2}\int\varphi(x,v)\varphi(x,v_*)\Big(e^{\Delta p (x,v,v_*)}-1\Big)\big((v_*-v)\cdot\omega\big)_+d\omega dv_* dvdx,
\end{equation}
with $\omega\in \mathds{S}^{d-1}$ being the collision direction. The function $\Delta p$ is defined as
\begin{equation}\label{eq:Delta_p}
\Delta p(x,v,v_*):=p(x,v')+p(x,v_*')-p(x,v)-p(x,v_*).
\end{equation}
The variables $(v',v_*')$ is the pre-collisional configuration, defined in a way similar to \eqref{eq:hard_sphere_law_2}
\begin{equation*}
v'=v-\big((v-v_*)\cdot\omega \big)\omega,\quad v_*'=v_*+\big((v-v_*)\cdot\omega \big)\omega.
\end{equation*}
This Hamilton-Jacobi equation contains a collision term represented by $\mathcal{H}$, and a tranport term, which resembles the Boltzmann equation. In the Hamiltonian $\mathcal{H}$, the term $e^{\Delta p}-1$ represents the effect of collision: the $p(x,v')+p(x,v_*')$ in $\Delta p$ has a similar role as the gain term in Boltzmann equation, and the $-p(x,v)-p(x,v_*)$ in $\Delta p$ has the same role as the loss term in Boltzmann equation. By taking the derivative $\frac{\p \Lambda(t,h)}{\p h(t)}$ at $h=0$, the Boltzmann equation is recovered formally in a weak sense.

In fact, the Hamilton-Jacobi equation encodes much more information about the hard sphere dynamics than the usual Boltzmann equation, in particular it encodes all the dynamical correlations. For a complete justification of the contents above, the readers may read \cite{BGSS_2023}. In \cite{Bouchet_2020}, more formal discussion with physical motivation about the meaning of this Hamiltonian is given.

One can also use test functions $h$ on the entire trajectory during the time interval $[0,t]$, as $h\big(z([0,t])\big)$, which is the case in \cite{BGSS_cluster}. Particularly in this paper, we may choose the test function of the form
\begin{equation}\label{eq:def_of_g_by_h}
h\big(z([0,t])\big)=g\big(t,z(t)\big)-\int_0^tD_sg\big(s,z(s)\big)ds,
\end{equation}
where $D_s$ refers to $\p_s+v\cdot \nabla_x$, and $g$ is a function depending on variables $(t,x,v)$. This choice of test functions enables us to integrate the transport term in equation \eqref{eq:functional_HJ_transport}. It then gives the Hamilton-Jacobi equation for $\mathcal{I}(t,g):=\Lambda(t,h)$, where $h$ is defined through $g$ by \eqref{eq:def_of_g_by_h}
\begin{equation}\label{eq:functional_HJ}
\p_t\mathcal{I}(t,g)=\mathcal{H}\Big(\frac{\p \mathcal{I}(t,g)}{\p g(t)},g(t)\Big),\quad \mathcal{I}(0,g)=\int (e^{g(0)}-1)f^0 dvdx.
\end{equation}
This functional equation has been introduced in Theorem 7 of \cite{BGSS_2023}.

\subsection{Coupled Boltzmann Equations as an Euler-Lagrange System}\label{sec:coupled_intro}
To find the solution $\mathcal{I}(t,g)$ of the Hamiltonian-Jacobi equation \eqref{eq:functional_HJ}, it is shown in the subsection 7.1.1. of \cite{BGSS_2023} that we can look at the associated Hamiltonian system. There are two interesting equivalent formulations of the Hamiltonian system. The first formulation is for $s\in[0,t]$
\begin{equation}\label{eq:Hamiltonian_physical_intro}
\begin{split}
&D_s\varphi_t(s)=\frac{\p \mathcal{H}}{\p p}\big(\varphi_t(s),p_t(s)\big),\quad \textup{with $\varphi_t(0)=f^0e^{p_t(0)}$},\\
&D_s(p_t-g)(s)=-\frac{\p \mathcal{H}}{\p \varphi}\big(\varphi_t(s),p_t(s)\big),\quad \textup{with $p_t(t)=g(t)$}.
\end{split}
\end{equation}
The subscript $t$ means we are studying the coupled system in the time interval $s\in[0,t]$, with terminal data given at time $t$. Given a mild solution $(\varphi_t,p_t)$ of equation \eqref{eq:Hamiltonian_physical_intro} on $[0,t]$ with initial data $\varphi_t(0)=f^0e^{p_t(0)}$ and terminal data $p_t(t)=g(t)$, we define the functional $\widehat{\mathcal{I}}(t,g)$ as
\begin{equation}\label{eq:characteristic_functional_intro}
    \widehat{\mathcal{I}}(t,g):=\int (e^{p_t(0)}-1)f^0 dvdx+\int_0^t\int \varphi_t(s) D_s\big(p_t(s)-g(s)\big)dvdxds+\int_0^t\mathcal{H}\Big(\varphi_t(s),p_t(s)\Big)ds.
\end{equation}
It will be proved in Theorem \ref{th:justification_mild} that an equivalent form of the functional $\widehat{\mathcal{I}}(t,g)$ constructed in equation \eqref{eq:characteristic_functional_intro} is a mild solution of the Hamilton-Jacobi equation \eqref{eq:functional_HJ}. The notion of mild solution will be specified in Section \ref{sec:main_result}.

However in this paper we do not directly deal with the coupled system given above. In \cite{BGSS_2023} (Section 7), by performing the change of variables
\begin{equation}\label{eq:change_of_variable_BGSS}
\Big(\psi_t(s),\eta_t(s)\Big)=\Big(\varphi_t(s)e^{-p_t(s)},e^{p_t(s)}\Big),
\end{equation}
an alternative equivalent formulation with better symmetry is introduced 
\begin{equation}\label{eq:coupled_Boltzmann_intro_alpha=0}
\begin{split}
D_s\psi_t&=-\psi_t D_sg+\int \big((v_*-v)\cdot\omega\big)_+\eta_t(v_*)\Big[\psi_t(v')\psi_t(v_*')-\psi_t(v)\psi_t(v_*)\Big]d\omega dv_*,\quad \psi_t(0)=f^0,\\
D_s \eta_t&=\eta_t D_sg-\int \big((v_*-v)\cdot\omega\big)_+\psi_t(v_*)\Big[\eta_t(v')\eta_t(v_*')-\eta_t(v)\eta_t(v_*)\Big]d\omega dv_*,\quad \eta_t(t)=e^{g(t)}.
\end{split}
\end{equation}
In the paper, we generalize the change of variables \eqref{eq:change_of_variable_BGSS} into
\begin{equation}\label{eq:change_of_variable_alpha}
\Big(\psi_t(s),\eta_t(s)\Big)=\Big(\varphi_t(s)e^{-p_t(s)+\alpha'|v|^2},e^{p_t(s)-\alpha'|v|^2}\Big),
\end{equation}
for arbitrary $\alpha'\in\mathds{R}$. The change of variables \eqref{eq:change_of_variable_BGSS} in \cite{BGSS_2023} corresponds to the particular case $\alpha'=0$. It will be proved in Lemma \ref{lem:symmetrization} that after this generalized change of variables, the $(\psi_t,\eta_t)$ satisfies the following coupled Boltzmann equations during the time interval $[0,t]$
\begin{equation}\label{eq:coupled_Boltzmann_intro_pre}
\begin{split}
&D_s\psi_t=-\psi_t D_sg+\int \big((v_*-v)\cdot\omega\big)_+\eta_t(v_*)\Big[\psi_t(v')\psi_t(v_*')-\psi_t(v)\psi_t(v_*)\Big]d\omega dv_*,\\
&D_s \eta_t=\eta_t D_sg-\int \big((v_*-v)\cdot\omega\big)_+\psi_t(v_*)\Big[\eta_t(v')\eta_t(v_*')-\eta_t(v)\eta_t(v_*)\Big]d\omega dv_*,\\
&\psi_t(0)=f^0e^{\alpha'|v|^2},\quad \eta_t(t)=e^{g(t)}e^{-\alpha'|v|^2}.
\end{split}
\end{equation}
Under the generalized change of variables, the form of the coupled Boltzmann equation is the same as \eqref{eq:coupled_Boltzmann_intro_alpha=0}, but with different $\psi_t(0)$ and $\eta_t(t)$. It will be explained in Section \ref{sec:perturbation} that a proper choice of $\alpha'$ enables us to solve the equation in a convenient functional setting. 
\begin{definition}\label{def:biased_collision_operator}
\textup{\textbf{[Biased Collision Operator] }}We define the biased collision operator $\mathcal{Q}_{\eta}(\psi_1,\psi_2)$ as follows
\begin{equation}\label{eq:biased_collision_operator}
\begin{split}
&\mathcal{Q}_\eta(\psi_1,\psi_2):=\frac{1}{2}\int \big((v_*-v)\cdot\omega\big)_+ \eta(v_*)\Big[\psi_1(v')\psi_2(v_*')+\psi_2(v')\psi_1(v_*')\\
&\qquad\qquad\qquad\qquad\qquad\qquad\qquad\qquad\ \ \ -\psi_1(v)\psi_2(v_*)-\psi_2(v)\psi_1(v_*)\Big]d\omega dv_*.
\end{split}
\end{equation}
\end{definition}
The function $\phi$ is defined as follows, related to the spatial transport
\begin{equation}\label{eq:transport}
\phi(s,x,v):=D_sg(s,x,v)=(\p_s+v\cdot\nabla_x)g(s,x,v)
\end{equation}
Based on the definitions of the biased collision operator and the function $\phi$, we rewrite equation \eqref{eq:coupled_Boltzmann_intro_pre} in a more compact form 
\begin{equation}\label{eq:coupled_Boltzmann_sym_intro}
\begin{split}
D_s \psi_t&=-\psi_t\phi+\mathcal{Q}_{\eta_t}(\psi_t,\psi_t),\quad \psi_t(0)=f^0e^{\alpha'|v|^2},\\
D_s \eta_t&=\eta_t\phi-\mathcal{Q}_{\psi_t}(\eta_t,\eta_t),\quad \eta_t(t)=e^{g(t)}e^{-\alpha' |v|^2}.
\end{split}
\end{equation}
Since equation \eqref{eq:coupled_Boltzmann_sym_intro} is equivalent to \eqref{eq:Hamiltonian_physical_intro} through the change of variables \eqref{eq:change_of_variable_alpha}, we can as well construct the functional $\widehat{\mathcal{I}}(t,g)$ given in \eqref{eq:characteristic_functional_intro}, by solving equation \eqref{eq:coupled_Boltzmann_sym_intro}. 

We call $\psi_t$ the 'forward component', due to its given initial data and the positive sign of collision operator. The other component $\eta_t$ is called the 'backward component', due to its given terminal data at time $t$ and also the negative sign of collision operator. Each component provides a bias for the nonlinear collision of the other component, which is transparent in the equation \eqref{eq:coupled_Boltzmann_sym_intro}. If we take $g-\alpha'|v|^2\equiv 0$, the coupled Boltzmann equations would degenerate to the usual Boltzmann equation, with $\eta_t\equiv 1$ and $\phi\equiv 0$.

A pivotal tool we use in the present paper to solve equation \eqref{eq:coupled_Boltzmann_sym_intro} is the theory of global-in-time solution for the Boltzmann equation with given initial data. There have been many works dealing with global solutions of Boltzmann equation, with different notions of solutions. Early works include \cite{ukai_1977} for classical solutions, \cite{Illner_Shinbrot_1984} for mild solutions, and \cite{DiPerna_Lions_1989} for renormalized solutions. Specifically in the present paper, we will adapt the perturbation regime for global mild solutions in a certain weighted $L^{\infty}$ space \cite{Ukai_2006}. 

A subtle issue in the present paper is that we want to solve equation \eqref{eq:coupled_Boltzmann_sym_intro} for those functions $e^{g(t)}$ with quadratic exponential growth in the velocity variable, for example when $g(t)=\frac{1}{4}|v|^2$. If we take $\alpha'=0$ in \eqref{eq:change_of_variable_alpha}, which corresponds to the original change of variables \eqref{eq:change_of_variable_BGSS} in \cite{BGSS_2023}, it naturally requires the forward component $\psi_t$ to have quadratic exponential decay in the velocity variable. The initial data $f^0$ could be assumed to have quadratic exponential decay, but it is hard to prove the propagation of this quadratic exponential decay. To overcome this difficulty, we will carry out a symmetrization procedure in Section \ref{sec:perturbation} by choosing a proper $\alpha'$.

\section{Main Results}\label{sec:main_result}
\subsection{Global-in-time Solution}\label{sec:main_result_subsection}
As explained in subsection \ref{sec:coupled_intro}, to solve the Hamilton-Jacobi equation \eqref{eq:functional_HJ} we will look at the associated Euler-Lagrange system, which is the coupled Boltzmann equations \eqref{eq:coupled_Boltzmann_sym_intro}. The goal is to find a certain class of functions $g$ such that, the mild solution $\mathcal{I}(t,g)$ could be constructed for arbitrary time $t\geq 0$. 

We say a pair of functions $(\psi_t,\eta_t)$ is a mild solution of the coupled Boltzmann equations \eqref{eq:coupled_Boltzmann_sym_intro}, if for arbitrary $s\in[0,t]$
\begin{equation}\label{eq:mild_solution_Boltzmann}
\begin{split}
&\psi_t(s)=S_sf^0e^{\alpha'|v|^2}-\int_0^s S_{s-\tau}\psi_t(\tau)\phi(\tau)d\tau+\int_0^s S_{s-\tau}\mathcal{Q}_{\eta_t(\tau)}\big(\psi_t(\tau),\psi_t(\tau)\big)d\tau,\\
&\eta_t(s)=S_{-(t-s)}e^{g(t)}e^{-\alpha'|v|^2}-\int_s^t S_{-(\tau-s)}\eta_t(\tau)\phi(\tau)d\tau+\int_s^t S_{-(\tau-s)}\mathcal{Q}_{\psi_t(\tau)}\big(\eta_t(\tau),\eta_t(\tau)\big)d\tau.
\end{split}
\end{equation}
Here the operators $\{S_{\tau}\}_{\tau\in\mathds{R}}$ are the transport semigroup defined as $S_{\tau}f(x,v)=f(x-\tau v,v)$. For the Hamilton-Jacobi equation \eqref{eq:functional_HJ}, we say a functional $\mathcal{I}(t,g)$ is a mild solution of the equation if for arbitrary $t\geq 0$
\begin{equation}\label{eq:mild_solution_HJ}
\mathcal{I}(t,g)=\mathcal{I}(0,g)+\int_0^t \mathcal{H}\Big(\frac{\p \mathcal{I}(s,g)}{\p g(s)},g(s)\Big)ds,\quad \mathcal{I}(0,g)=\int (e^{g(0)}-1)f^0dvdx.
\end{equation}
The main result of the paper is to construct a mild solution of the Hamilton-Jacobi equation with $f^0$ close to the spatially homogeneous standard Maxwellian $M$
\begin{equation*}
M(x,v)=(2\pi)^{-\frac{d}{2}}\exp\big(-\frac{1}{2}|v|^2\big),
\end{equation*}
and the function $e^{g}$ close to a certain reference function $\mathcal{E}$. In this paper, we consider those $\mathcal{E}$ of the form
\begin{equation}\label{assump:eta_equi}
\mathcal{E}(x,v)=(2\pi)^{-\frac{d}{2}}\exp(\alpha|v|^2),\quad \alpha<\frac{1}{2}.
\end{equation}
To simplify notations, we define the normalization function $\mathcal{B}$ as $e^{-\alpha'|v|^2}$. We want to choose a proper $\alpha'$ that symmetrizes the forward initial data and the backward terminal data in \eqref{eq:coupled_Boltzmann_sym_intro}, with $f^0e^{\alpha'|v|^2}$ and $e^{g(t)}e^{-\alpha'|v|^2}$ being close to the same coupled equilibrium $\mathcal{G}$
\begin{equation}\label{eq:coupled_equilibrium}
\mathcal{G}:=M^{\frac{1}{2}}\mathcal{E}^{\frac{1}{2}}=(2\pi)^{-\frac{d}{2}}\exp\Big(-\frac{1}{2}(\frac{1}{2}-\alpha)|v|^2\Big).
\end{equation}
This requires us to choose $\mathcal{B}$ with $M\mathcal{B}^{-1}=\mathcal{EB}=\mathcal{G}$, which yields
\begin{equation}\label{eq:normalizing_function}
\mathcal{B}=M^{\frac{1}{2}}\mathcal{E}^{-\frac{1}{2}}=\exp\Big(-\frac{1}{2}(\frac{1}{2}+\alpha)|v|^2\Big),
\end{equation}
and thus
\begin{equation}
\alpha'=\frac{1}{2}\big(\frac{1}{2}+\alpha\big).
\end{equation}

Since we only consider those reference functions $\mathcal{E}$ with $\alpha<\frac{1}{2}$, the coupled equilibrium $\mathcal{G}$ has quadratic exponential decay in $v$.

We will solve the coupled Boltzmann equations \eqref{eq:coupled_Boltzmann_sym_intro} with $\psi_t$ and $\eta_t$ being perturbations around $\mathcal{G}$. The perturbations should be in a $L^{\infty}$ space with a polynomial weight on $v$, denoted by $L_{\beta}^{\infty}$
\begin{equation}\label{eq:weighted_L_infty}
\lVert f\rVert_{L_{\beta}^{\infty}}=\sup_{(x,v)\in \mathds{T}^d\times \mathds{R}^d}\Big|f(x,v)(1+|v|)^{\beta}\Big|.
\end{equation}

For Theorem \ref{th:main_theorem_1}, we assume the forward initial data $f^0\mathcal{B}^{-1}$ is close to the coupled equilibrium $\mathcal{G}$ with the perturbation $L_{x,v}^2$-orthogonal to a kernel $\mathcal{K}$, representing the conserved quantities, to be defined in equation \eqref{eq:kernel} 
\begin{equation}\label{assump:main_1}
    \Big\lVert f^0\mathcal{B}^{-1}-\mathcal{G}\Big\rVert_{L_{\beta+1}^{\infty}}<c,\quad f^0\mathcal{B}^{-1}-\mathcal{G}\in\mathcal{K}^{\perp}. \tag{H1}
\end{equation}
The kernel $\mathcal{K}$ is defined as
\begin{equation}\label{eq:kernel}
\mathcal{K}:=\textup{span}\Big\{\mathcal{G},\mathcal{G}(v)v_1,...,\mathcal{G}(v)v_d,\mathcal{G}(v)|v|^2\Big\}.
\end{equation}
Throughout the paper, the constant $c>0$ is always taken to be small enough and properly tunned according to other parameters.

The terminal data $e^{g(t)}\mathcal{B}$ is assumed to be close to the coupled equilibrium $\mathcal{G}$ with the perturbation orthogonal to the kernel $\mathcal{K}$ at any time $t>0$
    \begin{equation}\label{assump:main_2}
    \Big\lVert e^{g(t)}\mathcal{B}-\mathcal{G}\Big\rVert_{L_{\beta+1}^{\infty}}< c,\quad e^{g(t)}\mathcal{B}-\mathcal{G}\in\mathcal{K}^{\perp}.  \tag{H2}
    \end{equation}
These orthogonality conditions are common in the literature for the solution of Boltzmann equations on torus to have decay in time. For example, the readers may see Theorem 2.3.1 of \cite{ukai_note}.

The function $g$ and the forward initial data $f^0\mathcal{B}^{-1}$ are assumed to have certain regularity and continuity
\begin{equation}\label{assump:main_3}
    \phi(s,x,v) \equiv 0,\ \textup{$g$ has uniformly bounded derivatives in $t$ and $x$},\quad S_{\tau}f^0\mathcal{B}^{-1}\in \textup{Lip}\big(\mathds{R}_{\tau};L_{\beta}^{\infty}\big)  \tag{H3}.
\end{equation}
Throughout the paper, the dimension $d$ will be taken as $d\geq 3$.

\begin{theorem}\label{th:main_theorem_1}
For arbitrary $\beta>4$ and $\sigma>1$, we can take constants $c>0$ and $a_*>0$ depending on $\beta,\sigma$ such that for any $f^0$ and $g$ satisfying the assumptions \eqref{assump:main_1}-\eqref{assump:main_3} and any time $t>0$, there exists a unique mild solution $(\psi_t,\eta_t)$ of the coupled Boltzmann equations \eqref{eq:coupled_Boltzmann_sym_intro} in the function class below 
\begin{equation*}
\sup_{0\leq s\leq t}(1+s)^{\sigma}\lVert \psi_t(s)-\mathcal{G}\rVert_{L_{\beta}^{\infty}}<a_*,\quad \sup_{0\leq s\leq t}\big(1+(t-s)\big)^{\sigma}\lVert \eta_t(s)-\mathcal{G}\rVert_{L_{\beta}^{\infty}}<a_*.
\end{equation*}
Furthermore, the functional $\widehat{\mathcal{I}}(t,g)$ in \eqref{eq:characteristic_functional_intro} is well-defined for any functions $f^0$ and $g$ satisfying the assumptions \eqref{assump:main_1}-\eqref{assump:main_3}, and is a global-in-time mild solution of the Hamilton-Jacobi equation \eqref{eq:functional_HJ}. The mild solution $\widehat{\mathcal{I}}(t,g)$ is uniformly bounded for any time $t\geq 0$ and any functions $f^0,g$ satisfying the assumptions \eqref{assump:main_1}-\eqref{assump:main_3}. This solution $\widehat{\mathcal{I}}(t,g)$ also converges to a non-trivial stationary solution as $t\rightarrow +\infty$.
\end{theorem}

We present now a similar result (Theorem \ref{th:main_theorem_2}) for forcing $g$ satisfying a different set of assumptions. The assumptions are more general in one aspect, while being more restrictive in another aspect.

For Theorem \ref{th:main_theorem_2}, the initial data $f^0\mathcal{B}^{-1}$ is only assumed to be close to the coupled equilibrium $\mathcal{G}$, without the orthogonality condition
    \begin{equation}\label{assump:main_1_th2}
    \Big\lVert f^0\mathcal{B}^{-1}-\mathcal{G}\Big\rVert_{L_{\beta+1}^{\infty}}<c.\tag{H4}
    \end{equation}
The terminal data $e^{g(t)}\mathcal{B}$ is assumed to be close to the coupled equilibrium $\mathcal{G}$ without orthogonality condition, but its perturbation is assumed to decay exponentially in time with $\sigma>0$
    \begin{equation}\label{assump:main_2_th2}
    \Big\lVert e^{g(t)}\mathcal{B}-\mathcal{G}\Big\rVert_{L_{\beta+1}^{\infty}}< e^{-\sigma t}c.\tag{H5}
    \end{equation}
The function $g$ and the forward initial data $f^0\mathcal{B}^{-1}$ are assumed to have certain regularity and continuity, where the function $\phi$ is defined in \eqref{eq:transport} and $\eps_{\phi}>0$ is a small enough positive constant
\begin{equation}\label{assump:main_3_th2}
\begin{split}
    &\lVert \phi\rVert_{L_t^1(L_{x,v}^{\infty})}+\lVert \phi\rVert_{C_t^0(L_{x,v}^{\infty})}\leq \eps_{\phi},\ \textup{$g$ has uniformly bounded derivatives in $(t,x)$},\\
    &\qquad\qquad\qquad\qquad\qquad\qquad\quad S_{\tau}f^0\mathcal{B}^{-1}\in \textup{Lip}\big(\mathds{R}_{\tau};L_{\beta}^{\infty}\big) 
    \end{split}\tag{H6}
\end{equation}
\begin{theorem}\label{th:main_theorem_2}
For arbitrary $\beta>4$ and $\sigma>1$, we can take constants $c>0$ and $a_*>0$ depending on $\beta,\sigma$ such that for any $f^0$ and $g$ satisfying the assumptions \eqref{assump:main_1_th2}-\eqref{assump:main_3_th2} and any time $t>0$, there exists a unique mild solution $(\psi_t,\eta_t)$ of the coupled Boltzmann equations \eqref{eq:coupled_Boltzmann_sym_intro} in the function class below 
\begin{equation*}
\sup_{0\leq s\leq t}\lVert \psi_t(s)-\mathcal{G}\rVert_{L_{\beta}^{\infty}}<a_*,\quad \sup_{0\leq s\leq t}e^{\sigma s}\lVert \eta_t(s)-\mathcal{G}\rVert_{L_{\beta}^{\infty}}<a_*.
\end{equation*}
Furthermore, the functional $\widehat{\mathcal{I}}(t,g)$ in \eqref{eq:characteristic_functional_intro} is well-defined for any functions $f^0$ and $g$ satisfying the assumptions \eqref{assump:main_1_th2}-\eqref{assump:main_3_th2}, and is a global-in-time mild solution of the Hamilton-Jacobi equation \eqref{eq:functional_HJ}. The mild solution $\widehat{\mathcal{I}}(t,g)$ is uniformly bounded for any time $t\geq 0$ and any functions $f^0,g$ satisfying the assumptions \eqref{assump:main_1_th2}-\eqref{assump:main_3_th2}.
\end{theorem}
\begin{figure}[H]
    \centering
    \includegraphics[scale=0.2]{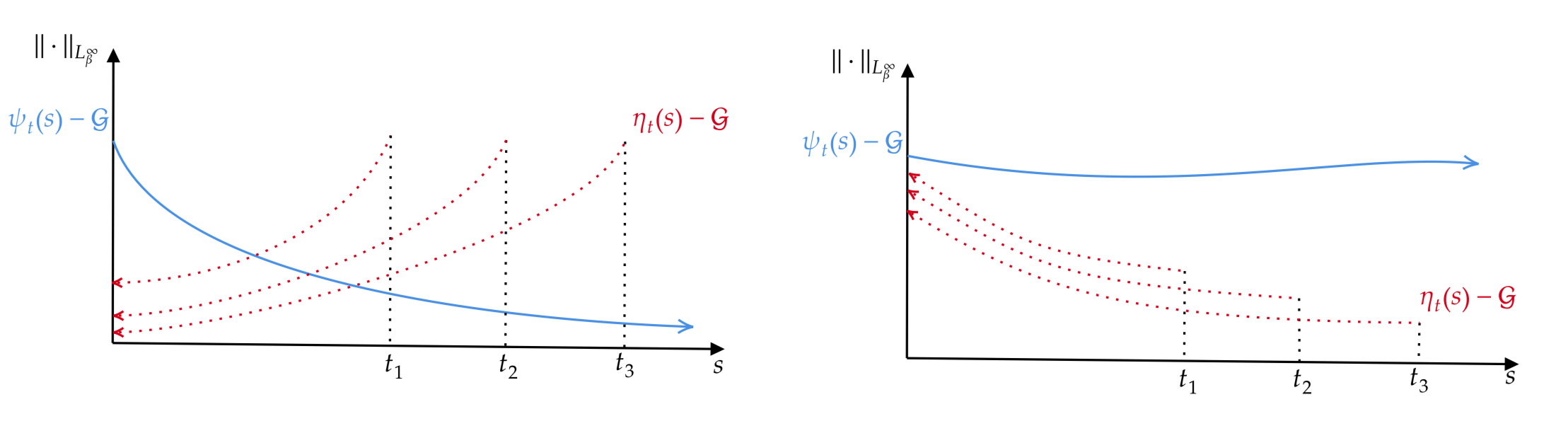}
    \caption{The blue arrow refers to the forward component $\psi_t$, and the red arrow refers to the backward component $\eta_t$. The numbers $t_1$, $t_2$, and $t_3$ refer to different terminal times. For each terminal time $t$, we solve the coupled Boltzmann equations on $[0,t]$.}
    \label{fig:equilibriate}
\end{figure}
Figure \ref{fig:equilibriate} illustrates the shape of the perturbations $\psi_t(s)-\mathcal{G}$ and $\eta_t(s)-\mathcal{G}$ in different cases: Theorem \ref{th:main_theorem_1} or \ref{th:main_theorem_2}. 

In Theorem \ref{th:main_theorem_1}, for each terminal time $t$, the forward component $\psi_t(s)$ decays forwards to the coupled equilibrium $\mathcal{G}$, and the backward component $\eta_t$ also decays backwards to the $\mathcal{G}$. The polynomial decay rate $\sigma>1$, and the size $\lVert \eta_t(t)-\mathcal{G}\rVert_{L_{\beta}^{\infty}}$ of perturbation at terminal time $t$ is uniform for all terminal time $t$.

In Theorem \ref{th:main_theorem_2}, we do not have the estimate for the convergence towards equilibrium. This is because we do not assume orthogonality condition for the initial data $f^0$ and the terminal data $e^{g(t)}$, as well as the function $D_sg$ could be not constant zero. These may give some (hydrodynamic) modes of constant order that will be preserved in the evolution. The price of having this generality is that, we assume that the size $\lVert \eta_t(t)-\mathcal{G}\rVert_{L_{\beta}^{\infty}}$ of the perturbation at terminal time $t$ must decay as $t\rightarrow+\infty$.

\subsection{Future Directions}\label{sec:implications}

\underline{\textbf{Global Solution of Forced Boltzmann Equation:}} Based on the results of the current paper, it would be interesting to look at the global-in-time solution $\varphi$ of the forced Boltzmann equation with forcing $p$ given by
\begin{equation}\label{eq:forced_Boltzmann}
D_s\varphi=\int \Big(\varphi(v')\varphi(v_*')\exp(-\Delta p)-\varphi(v)\varphi(v_*)\exp(\Delta p)\Big)\big((v_*-v)\cdot\omega\big)_+d\omega dv_*.
\end{equation}
This type of modified equation is crucial for the large deviation theory established in \cite{BGSS_2023} for a hard sphere gas. It is shown that for appropriate function $p$, we are able to look at the asymptotic probability of empirical measure $\pi_t^{\eps}$ converging to an atypical density $\varphi(t)$
\begin{equation*}
\pi_{t}^{\eps}\rightarrow \varphi(t),\quad \eps\rightarrow 0,
\end{equation*}
when $\varphi$ is a solution of \eqref{eq:forced_Boltzmann}.  

Previously only local-in-time result about the forced Boltzmann equation is known. The relation between this future direction and the present results is that, the forced equation \eqref{eq:forced_Boltzmann} is exactly the forward equation in the Euler-Lagrange system \eqref{eq:Hamiltonian_physical_intro}. The difference is that, in this paper we solve \eqref{eq:Hamiltonian_physical_intro} with given $g$, while for \eqref{eq:forced_Boltzmann} we solve it with a given forcing $p$.

\underline{\textbf{Relation with Schrödinger Problem:}} In the previous paragraph we mentioned that for $\varphi$ being a solution of \eqref{eq:forced_Boltzmann}, we can study the large deviation cost of it. The $(\varphi,p)$ is related with $(\psi,\eta)$ through the change of variables \eqref{eq:change_of_variable_alpha}, with $\varphi=\psi\eta$.
\begin{figure}[h]
    \centering
    \includegraphics[scale=0.18]{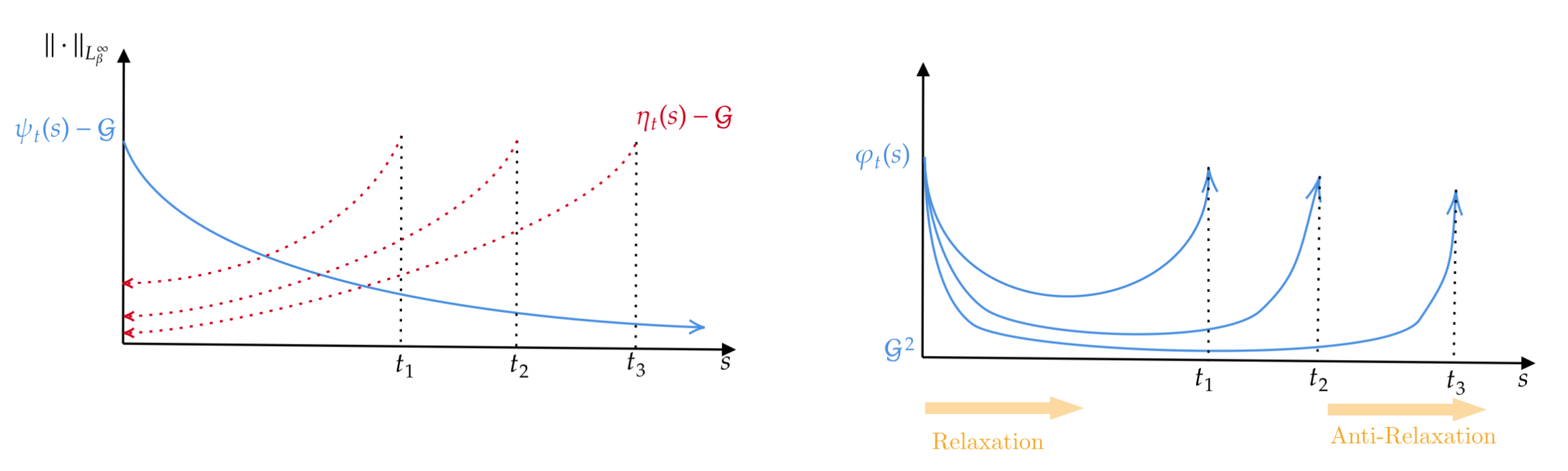}
    \caption{The figure on the left hand side is the illustration of $(\psi_t,\eta_t)$ given in Theorem \ref{th:main_theorem_1}. The figure on the right hand side is the illustration of the corresponding density profile $\varphi_t$. When the time $s$ is distant from both the initial time $0$ and the terminal time $t$, two components $\psi_t$ and $\eta_t$ are both near to $\mathcal{G}$. Thus by $\varphi_t(s)=\psi_t(s)\eta_t(s)$, the density profile $\varphi_t(s)$ should be close to $\mathcal{G}^2$, which is the physical reference measure for the perturbation theory.}
    \label{fig:enter-label}
\end{figure}

Based on the solution $(\psi_t,\eta_t)$ given in Theorem \ref{th:main_theorem_1}, as $t\rightarrow +\infty$ the corresponding density profile $\varphi_t$ converges to a 'Relaxation and Anti-Relaxation' dynamics: it first relaxes to an equilibrium, and then anti-relaxes to an atypical density profile. This behaviour is related to the Schrödinger problem, namely the computation of the optimal path, given a large deviation cost function, followed by particle system from a given density at time $0$ to another density at time $t$. The mean-field version of this relation has been investigated in \cite{Léonard_2020}. For a survey of the Schrödinger Problem and its connection with optimal transport, see \cite{Leonard_2014}. 

\underline{\textbf{Uniform Control of the Limiting Cumulant Generating Functional:}}  As we have explained in Subsection \ref{sec;hard_sphere_gas}, the following functional is used to encode the correlation information of a hard sphere gas with diameter $\eps$
\begin{equation}\label{eq:cumulant_future}
\Lambda^{\eps}(t,h):=\frac{1}{\mu_{\eps}}\log\mathds{E}_{\eps}\Big[\exp\Big(\mu_{\eps}\pi_t^{\eps}(h)\Big)\Big].
\end{equation}
After taking the Boltzmann-Grad limit, this functional should converge to the functional $\Lambda(t,h)$ encoding correlation of the limiting particle system. 
\begin{figure}[H]
    \centering
    \includegraphics[scale=0.15]{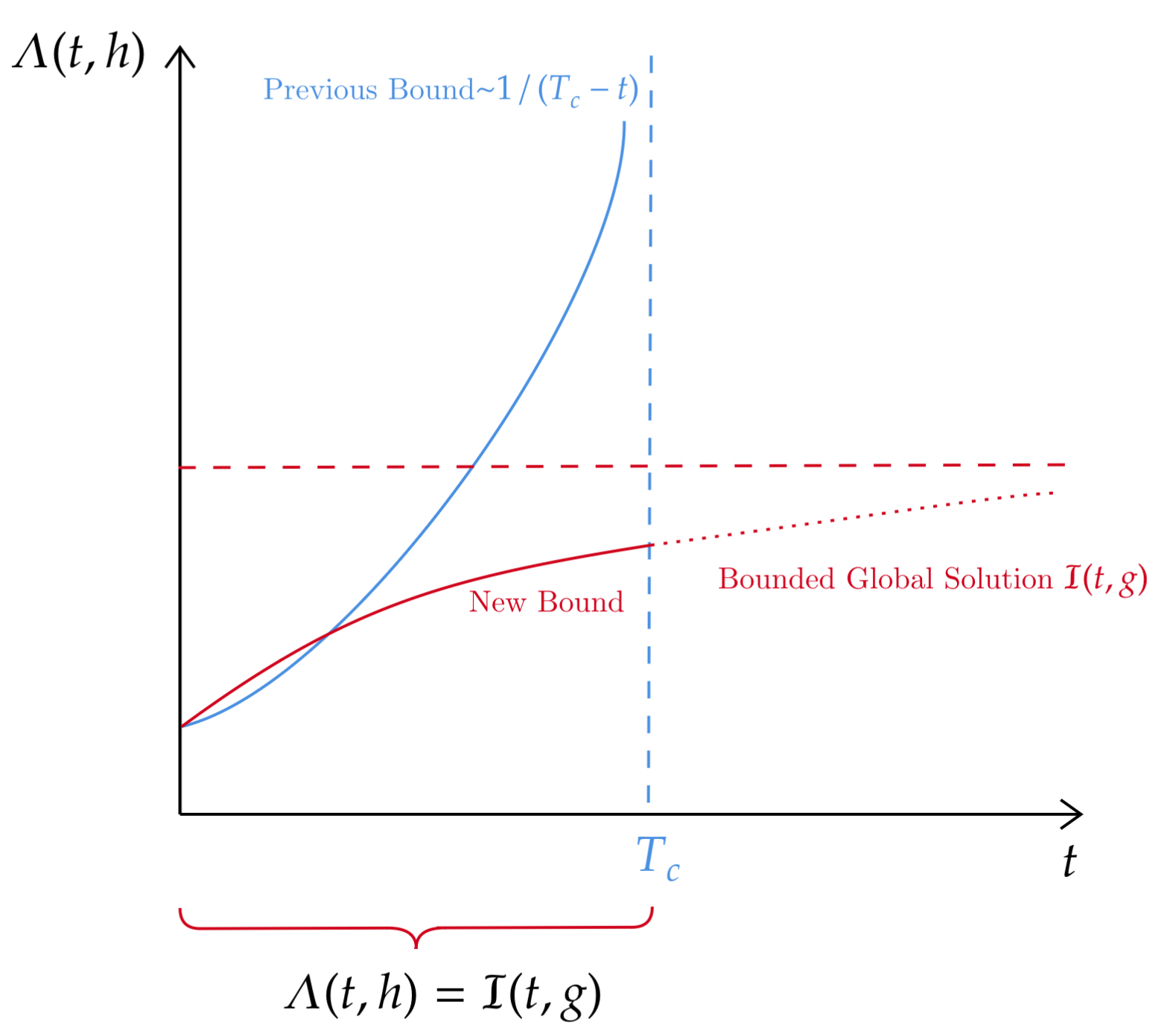}
    \caption{By constructing global-in-time solution $\mathcal{I}(t,g)$ of the Hamilton-Jacobi equation, we can give a uniform upper bound for the limiting cumulant generating functional $\Lambda(t,h)$.}
    \label{fig:blowup}
\end{figure}
In \cite{BGSS_2023} it has been shown the functional $\Lambda(t,h)$ coincide with the solution $\mathcal{I}(t,g)$ of the Hamilton-Jacobi equation \eqref{eq:functional_HJ} in a finite time interval $[0,T_c)$ for $T_c<+\infty$, with $h$ determined by $g$ as in \eqref{eq:def_of_g_by_h}. The coincidence between $\Lambda(t,h)$ and $\mathcal{I}(t,g)$ for any time $t$ still remains to be proved.

One of the main difficulties for hard sphere gas to have global-in-time results about kinetic limit, dynamical fluctuations, and large deviations is the divergence of the upper bound for cumulant generating functionals when the time $t$ approaches $T_c$. By establishing global-in-time solution $\mathcal{I}$, with the coincidence between $\Lambda(t,h)$ and $\mathcal{I}(t,g)$ for $t\in[0,T_c)$, we can provide a uniform upper bound for the limiting cumulant generating functional (Figure \ref{fig:blowup}). If this coincidence between functionals can be extended to the whole time interval, then the solution $\mathcal{I}$ will provide a global-in-time uniform control of the limiting cumulant generating functional.

However the current result does not provide uniform control for the $\eps$-cumulant generating functional. This will be left to future work. A uniform control of the cumulant generating functionals is expected to be an important step in proving long-time results about kinetic limit, dynamical fluctuations, and large deviations.

\section{Symmetrization and Perturbation Regime for Coupled Boltzmann Equations}\label{sec:perturbation}
In this section, we will perform the symmetrization procedure and define the perturbation regime, needed to solve the coupled Boltzmann equations \eqref{eq:coupled_Boltzmann_sym_intro}. 

As we have discussed in the previous sections, the symmetrization procedure consists in the change of variables 
\begin{equation}\label{eq:change_of_variables_sec3}
\Big(\psi_t(s),\eta_t(s)\Big)=\Big(\varphi_t(s)e^{-p_t(s)+\alpha'|v|^2},e^{p_t(s)-\alpha'|v|^2}\Big),\quad \alpha'=\frac{1}{2}(\frac{1}{2}+\alpha).
\end{equation}
It will be proved in Lemma \ref{lem:symmetrization} that the pair $(\psi_t,\eta_t)$ satisfies the coupled Boltzmann equation \eqref{eq:coupled_Boltzmann_sym_intro}. 

After the symmetrization, we will perform a perturbation decomposition to $(\psi,\eta)$: we look at the evolution for the perturbation of $(\psi,\eta)$ from the coupled equilibrium, and we denote the perturbation as $(\psi_p,\eta_p)$. The evolution equation \eqref{eq:transformed_evolution_perturb} of the perturbation is given as \eqref{eq:transformed_evolution_perturb}. Finding a mild solution of equation \eqref{eq:coupled_Boltzmann_sym_intro} is equivalent to finding a mild solution of equation \eqref{eq:transformed_evolution_perturb}, and we will use a fixed-point method to find the solution of the latter.

\begin{lemma}\label{lem:symmetrization}
A pair of functions $(\varphi_t,p_t)$ is a mild solution of the coupled Boltzmann equations \eqref{eq:Hamiltonian_physical_intro}, if and only if the pair of functions $(\psi_t,\eta_t)$ given by the change of variables \eqref{eq:change_of_variables_sec3} is a mild solution of the coupled Boltzmann equations with $\mathcal{B}$ defined in \eqref{eq:normalizing_function}
\begin{equation}\label{eq:transformed_evolution}
\begin{split}
D_s \psi(s)&=-\psi(s)\phi(s)+\mathcal{Q}_{\eta}(\psi,\psi),\quad \psi(0)=f^0\mathcal{B}^{-1},\\
D_s \eta(s)&=\eta(s)\phi(s)-\mathcal{Q}_{\psi}(\eta,\eta),\quad \eta(t)=e^{g(t)}\mathcal{B}.
\end{split}
\end{equation}
\end{lemma}
\begin{proof}
Expanding the Hamiltonian system \eqref{eq:Hamiltonian_physical_intro}, we have the following equation with $\Delta p_t$ defined in \eqref{eq:Delta_p} 
\begin{equation*}
    \begin{split}
    &D_s\varphi_t=\int \Big(\varphi_t(v')\varphi_t(v_*')\exp(-\Delta p_t)-\varphi_t(v)\varphi_t(v_*)\exp(\Delta p_t)\Big)\big((v_*-v)\cdot\omega\big)_+d\omega dv_*,\\
    &D_sp_t=\phi(s)-\int \varphi_t(v_*)\Big(\exp(\Delta p_t)-1\Big)\big((v_*-v)\cdot\omega\big)_+d\omega dv_*,\\
    &\varphi_t(0)=f^0e^{p_t(0)},\quad p_t(t)=g(t)
    \end{split}
\end{equation*}
According to the definition of $\Delta p_t$, a pair $(\varphi_t,p_t)$ is a solution of the equation if and only if $(\varphi_t,p_t-\alpha'|v|^2)$ is a solution of the equation. Thus
\begin{equation*}
    \begin{split}
    &D_s\varphi_t=\int \Big(\varphi_t(v')\varphi_t(v_*')\exp(-\Delta p_t)-\varphi_t(v)\varphi_t(v_*)\exp(\Delta p_t)\Big)\big((v_*-v)\cdot\omega\big)_+d\omega dv_*,\\
    &D_sp_t=\phi(s)-\int \varphi_t(v_*)\Big(\exp(\Delta p_t)-1\Big)\big((v_*-v)\cdot\omega\big)_+d\omega dv_*,\\
    &\varphi_t(0)=f^0e^{\alpha'|v|^2}e^{p_t(0)},\quad p_t(t)=g(t)-\alpha'|v|^2
    \end{split}
\end{equation*}
This concludes the proof of the lemma.
\end{proof}

It is natural to look at the evolution of perturbations, with the hope that the smallness of the initial and terminal perturbations could imply the global well-posedness of the equation
\begin{equation}\label{eq:perturbation_decomposition}
    \psi_p:=\psi-\mathcal{G},\quad \eta_p:=\eta-\mathcal{G}.
\end{equation}

The evolution of the perturbations $(\psi_p,\eta_p)$ should satisfy of the following coupled Boltzmann equations
\begin{equation}\label{eq:transformed_evolution_perturb}
\begin{split}
&\p_s\psi_p(s)\\
=&\underbrace{-v\cdot\nabla_x \psi_p(s)+2\mathcal{Q}_{\mathcal{G}}(\psi_p,\mathcal{G})}_{\textup{\blue{linear}}}+\underbrace{\mathcal{Q}_{\eta_p}(\psi_p,\psi_p)+2\mathcal{Q}_{\eta_p}(\psi_p,\mathcal{G})+\mathcal{Q}_{\mathcal{G}}(\psi_p,\psi_p)}_{\textup{\red{nonlinear}}}-\psi_p\phi(s)-\mathcal{G}\phi(s),\\
&\p_s\eta_p(s)\\
=&\underbrace{-v\cdot\nabla_x \eta_p(s)-2\mathcal{Q}_{\mathcal{G}}(\eta_p,\mathcal{G})}_{\textup{\blue{linear}}}-\underbrace{\mathcal{Q}_{\psi_p}(\eta_p,\eta_p)-2\mathcal{Q}_{\psi_p}(\eta_p,\mathcal{G})-\mathcal{Q}_{\mathcal{G}}(\eta_p,\eta_p)}_{\textup{\red{nonlinear}}}+\eta_p\phi(s)+\mathcal{G}\phi(s).
\end{split}
\end{equation}

The mild solution of \eqref{eq:transformed_evolution_perturb} is defined as the fixed point of the fixed-point map $\Gamma=(\Gamma^+,\Gamma^-)$, which will be introduced in Definition \ref{def:fixed_point}. For simplicity, we identify the mild solutions of \eqref{eq:coupled_Boltzmann_sym_intro} with the mild solutions of \eqref{eq:transformed_evolution_perturb}. At a rigorous level, for the solution $(\psi_p,\eta_p)$ of \eqref{eq:transformed_evolution_perturb} considered in this paper, by performing series expansion the corresponding $(\psi,\eta)$ can be shown to be a solution of \eqref{eq:coupled_Boltzmann_sym_intro}.

For the formal proof, We only detail it for the evolution of forward perturbation $\psi_p$, while the proof for the backward perturbation $\eta_p$ is almost the same.

The definition of perturbation \eqref{eq:perturbation_decomposition} implies
\begin{equation}\label{eq:inverse_decomposition}
\psi=\mathcal{G}+\psi_p,\quad \eta=\mathcal{G}+\eta_p.
\end{equation}
Use the equation above to replace $(\psi,\eta)$ with $(\psi_p,\eta_p)$ in equation \eqref{eq:transformed_evolution}. For the evolution of the forward component, we have
\begin{equation*}
\p_s \big(\mathcal{G}+\psi_p\big)(s)=-v\cdot\nabla_x\big(\mathcal{G}+\psi_p\big)(s)-\big(\mathcal{G}+\psi_p\big)(s)\phi(s)+\mathcal{Q}_{\mathcal{G}+\psi_p}(\mathcal{G}+\psi_p,\mathcal{G}+\psi_p).
\end{equation*}
Since the reference function is independent of the time $t$ and space $x$ variables, we get
\begin{equation*}
\p_s \psi_p(s)=-v\cdot\nabla_x\psi_p(s)-\big(\mathcal{G}+\psi_p\big)(s)\phi(s)+\mathcal{Q}_{\mathcal{G}+\psi_p}(\mathcal{G}+\psi_p,\mathcal{G}+\psi_p).
\end{equation*}
Now we only need the following equality to expand the third-order nonlinear collision term
\begin{equation*}
\mathcal{Q}_{\mathcal{G}+\psi_p}(\mathcal{G}+\psi_p,\mathcal{G}+\psi_p)=2\mathcal{Q}_{\mathcal{G}}(\psi_p,\mathcal{G})+2\mathcal{Q}_{\eta_p}(\psi_p,\mathcal{G})+\mathcal{Q}_{\mathcal{G}}(\psi_p,\psi_p)+\mathcal{Q}_{\eta_p}(\psi_p,\psi_p).
\end{equation*}
This equality can be checked directly, using the fact that $\mathcal{G}$ is the exponential of a collision invariant
\begin{equation*}
\mathcal{Q}_{\mathcal{G}}(\mathcal{G},\mathcal{G})=0.
\end{equation*}

The coupled Boltzmann equations \eqref{eq:transformed_evolution_perturb} for the perturbation $(\psi_p,\eta_p)$ will be one of our central objects in the rest of the paper. In the definition below, we define the relevant notation needed to study that equation.

\begin{definition}\label{def:fixed_point}
We define the linear operators $\rB$ and $\lB$ separately as the linearized Boltzmann operators for the forward perturbation $\psi_p$ or the backward perturbation $\eta_p$
\begin{equation*}
\rB\psi_p:=-v\cdot \nabla_x \psi_p+2\mathcal{Q}_{\mathcal{G}}(\psi_p,\mathcal{G}),\quad \lB\eta_p:=v\cdot \nabla_x \eta_p+2\mathcal{Q}_{\mathcal{G}}(\eta_p,\mathcal{G}).
\end{equation*}
The nonlinearity in the evolution of perturbation is denoted as
\begin{equation*}
\mathcal{N}[\psi_p,\eta_p]:=2\mathcal{Q}_{\eta_p}(\psi_p,\mathcal{G})+\mathcal{Q}_{\mathcal{G}}(\psi_p,\psi_p)+\mathcal{Q}_{\eta_p}(\psi_p,\psi_p).
\end{equation*}
We introduce the map $\Gamma:(\psi_p,\eta_p)\mapsto (\rP[\psi_p,\eta_p],\lP[\psi_p,\eta_p])$, whose fixed-point is a mild solution of the coupled Boltzmann equations \eqref{eq:transformed_evolution_perturb} 
\begin{equation}\label{eq:fixed_point_problem}
\left\{
\begin{split}
&\rP[\psi_p,\eta_p](s):=e^{s\rB}\psi_p(0)-\int_{0}^s e^{(s-\tau)\rB}(\mathcal{G}+\psi_p)\phi(\tau)d\tau+\int_{0}^se^{(s-\tau)\rB}\mathcal{N}[\psi_p,\eta_p]d\tau,\\
&\lP[\psi_p,\eta_p](s):=e^{(t-s)\lB}\eta_{p}(t)-\int_{s}^t e^{(\tau-s)\lB}(\mathcal{G}+\eta_p)\phi(\tau)d\tau+\int_{s}^t e^{(\tau-s)\lB}\mathcal{N}[\eta_p,\psi_p]d\tau.
\end{split}
\right.
\end{equation}
\end{definition}
In the notations, the $+$ sign of $B^+$ and $\rP$ means that the operators are associated with the forward perturbation $\psi_p$, while the $-$ sign of $B^-$ and $\lP$ means that the operators are associated with the backward perturbation $\eta_p$. Specifically, since the terminal data of $\eta_p$ is given at time $t$, we will consider the evolution of $\eta_p$ in the reversed time. This caused the operator $B^-$ having a different sign from the linearized collision operator in the second line of equation \eqref{eq:transformed_evolution_perturb}.

Sections \ref{sec:decay_semigroup}, \ref{sec:nonliearity}, and \ref{sec:fixed_point} will be devoted to proving the existence and uniqueness of the fixed-point $(\psi_p,\eta_p)$ of $\Gamma$,
\begin{equation}\label{eq:fixed_point_problem_2}
\rP[\psi_p,\eta_p]=\psi_p,\quad \lP[\psi_p,\eta_p]=\eta_p.
\end{equation}
Specifically, we will prove that the map $\Gamma$ is a contraction map in certain function spaces. In section \ref{sec:decay_semigroup} we prove the decay estimates for the semigroups generated by $\rB$ and $\lB$; In section \ref{sec:nonliearity} we prove estimates of the nonlinear terms involved in the fixed-point problem; In section \ref{sec:fixed_point}, we prove the contraction property of the map $\Gamma$.

\section{Estimate of Relevant Semigroups in $L_{\beta}^{\infty}$ Norm}\label{sec:decay_semigroup}
The operator $\rB$ (resp. $\lB$) defined in Definition \ref{def:fixed_point} generates a strongly continuous semigroup $e^{s\rB}$ (resp. $e^{s\lB}$) on $L^2(\mathds{T}_x^d\times \mathds{R}_v^d)$. We can decompose the two semigroups as
\begin{equation}\label{eq:decomposition_semigroup}
e^{sB^+}=\mathcal{D}_1^{+}(s)+\mathcal{D}_2^{+}(s),\quad e^{sB^-}=\mathcal{D}_1^{-}(s)+\mathcal{D}_2^{-}(s).
\end{equation}
The components $\mathcal{D}_1^+(s)$ and $\mathcal{D}_1^-(s)$ have explicit expressions as
\begin{equation}\label{eq:explicit_expression_D_1}
\mathcal{D}_1^{+}(s)f(x,v)=e^{-\nu(v)s}f(x-sv,v),\quad \mathcal{D}_1^{-}(s)f(x,v)=e^{-\nu(v)s}f(x+sv,v),
\end{equation}
where $\nu(v)$ is the frequency multiplier defined in \eqref{eq:multiplier_and_K}. For hard spheres, there are positive constants $c_1,c_2>0$ such that $c_1(1+|v|)\leq \nu(v)\leq c_2(1+|v|)$.

The components $\mathcal{D}_2^+(s)$ and $\mathcal{D}_2^-(s)$ are respectively defined in \eqref{eq:D_1 and D_2 +} and \eqref{eq:D_1 and D_2 -}.

Recall the definition of $\mathcal{G}$ as 
\begin{equation*}
\mathcal{G}:=M^{\frac{1}{2}}\mathcal{E}^{\frac{1}{2}}=(2\pi)^{-\frac{d}{2}}\exp\Big(-\frac{1}{2}(\frac{1}{2}-\alpha)|v|^2\Big).
\end{equation*}
We define $\mathcal{P}$ as the projection operator from $L^2(\mathds{T}_x^d\times \mathds{R}_v^d)$ to the subspace $\mathcal{K}$ defined in \eqref{eq:kernel}. The normalized orthogonal basis spanning $\mathcal{K}$ is independent of the variable $x$, and is denoted as $\{f_i\}_{0\leq i\leq d+1}$. Notice that $\{f_i\}_{0\leq i\leq d+1}$ is different from the basis in \eqref{eq:kernel}, since that basis may is not orthogonal and normalized. For a function $f\in L^2(\mathds{T}_x^d\times \mathds{R}_v^d)$, we say $f\in\mathcal{K}^{\perp}$ if $\mathcal{P}f=0$. 

Lemma \ref{lem:L_beta_decay} is essentially a reorganization of the results in \cite{Ukai_2006}. Its proof will be recalled in Appendix \ref{app:decomposition_of_semigroup}.

\begin{lemma}\label{lem:L_beta_decay}
\textbf{\textup{[$L_\beta^\infty$ Estimate] }}For $\beta>4$, there exist constants $\nu_*>0$ and $C>0$ such that for any function $f\in\mathcal{K}^{\perp}$ and $s\geq 0$, we have
\begin{equation*}
\lVert \mathcal{D}_2^{+}(s)f\rVert_{L_{\beta}^\infty}\leq Ce^{-\nu_* s}\lVert (1+|v|)^{-1}f\rVert_{L_{\beta}^{\infty}},\qquad \lVert \mathcal{D}_2^{-}(s)f\rVert_{L_{\beta}^\infty}\leq Ce^{-\nu_* s}\lVert (1+|v|)^{-1}f\rVert_{L_{\beta}^{\infty}},
\end{equation*}
and for any function $f\in L_{\beta}^{\infty}$ and $s\geq 0$, we have
\begin{equation*}
\begin{split}
&\lVert\mathcal{D}_2^+(s)f\rVert_{L_{\beta}^{\infty}}\leq C\lVert f\rVert_{L_{\beta}^{\infty}},\quad \lVert\mathcal{D}_2^-(s)f\rVert_{L_{\beta}^{\infty}}\leq C\lVert f\rVert_{L_{\beta}^{\infty}},\\
&\lVert\mathcal{D}_1^+(s)f\rVert_{L_{\beta}^{\infty}}\leq Ce^{-\nu_*s}\lVert f\rVert_{L_{\beta}^{\infty}},\quad \lVert\mathcal{D}_1^-(s)f\rVert_{L_{\beta}^{\infty}}\leq Ce^{-\nu_*s}\lVert f\rVert_{L_{\beta}^{\infty}}.
\end{split}
\end{equation*}
\end{lemma}

For general data $f\in L_{\beta}^{\infty}$, the estimates of $\mathcal{D}_2^+(s)f$ and $\mathcal{D}_2^-(s)f$ are improved by the next Proposition.

\begin{proposition}\label{prop:semigroup_bound_for_general}
For $\beta>4$, there exists a constant $C>0$ such that for arbitrary function $f\in L_{\beta}^{\infty}$ and $s\geq 0$ we have
\begin{equation}\label{eq:semigroup_bound_for_general}
\lVert \mathcal{D}_2^{+}(s)f\rVert_{L_{\beta}^\infty}\leq C\lVert (1+|v|)^{-1}f\rVert_{L_{\beta}^{\infty}},\qquad \lVert \mathcal{D}_2^-(s)f\rVert_{L_{\beta}^\infty}\leq C\lVert (1+|v|)^{-1}f\rVert_{L_{\beta}^{\infty}}.
\end{equation}
\end{proposition}
\begin{proof}
We detail the proof only for the forward component $\mathcal{D}_2^+$ since the proof for the other component $\mathcal{D}_2^-$ is exactly the same.

First we consider the decomposition of $f$ as $f=\mathcal{P}f+(\id-\mathcal{P})f$
\begin{equation*}
\mathcal{D}_2^+(s)f=\mathcal{D}_2^+(s)\mathcal{P}f+\mathcal{D}_2^+(s)(\id-\mathcal{P})f.
\end{equation*}
Using Lemma \ref{lem:L_2_decay} and Lemma \ref{lem:L_beta_decay} as well as the decomposition of $f$, there is
\begin{equation*}
\lVert \mathcal{D}_2^+(s)f  \rVert_{L_{\beta}^\infty}\leq \lVert \mathcal{D}_2^+(s)\mathcal{P}f\rVert_{L_{\beta}^\infty}+\lVert \mathcal{D}_2^+(s)(\id-\mathcal{P})f  \rVert_{L_{\beta}^\infty}\leq C\lVert \mathcal{P}f  \rVert_{L_{\beta}^\infty}+C\lVert (1+|v|)^{-1}(\id-\mathcal{P})f  \rVert_{L_{\beta}^\infty}
\end{equation*}
Thus to derive the desired estimate \eqref{eq:semigroup_bound_for_general}, it is enough to prove
\begin{equation}\label{eq:general_data_control_0}
\lVert \mathcal{P}f  \rVert_{L_{\beta}^\infty}\leq C\lVert (1+|v|)^{-1}f\rVert_{L_{\beta}^{\infty}},\quad \lVert (1+|v|)^{-1}(\id-\mathcal{P})f  \rVert_{L_{\beta}^\infty}\leq C\lVert (1+|v|)^{-1}f\rVert_{L_{\beta}^{\infty}}
\end{equation}
For $\lVert \mathcal{P}f\rVert_{L_{\beta}^{\infty}}$, we use the fact that $\mathcal{P}f\in\mathcal{K}$, and that $\mathcal{K}$ is spanned by a family of orthogonal functions $\{f_j\}_{0\leq j\leq d+1}$ with exponential decay in $v$, therefore belonging to $L_{\beta}^{\infty}$
\begin{equation}\label{eq:general_data_control}
\begin{split}
\Big\lVert \mathcal{P}f\Big\rVert_{L_{\beta}^{\infty}} = &\Big\lVert \sum_{j=0}^{d+1}\langle f,f_j\rangle_{L_{x,v}^2}f_j\Big\rVert_{L_{\beta}^{\infty}}\leq \sum_{j=0}^{d+1}\Big\lVert  \langle f,f_j\rangle_{L_{x,v}^2}f_j\Big\rVert_{L_{\beta}^{\infty}}.
\end{split}
\end{equation}
Using Cauchy-Schwartz inequality for the inner product $\langle f,f_j\rangle$ as well as the fact $\lVert f_j\rVert_{L_{\beta}^{\infty}}$ is bounded, we further have
\begin{equation*}
\Big\lVert  \langle f,f_j\rangle_{L_{x,v}^2}f_j\Big\rVert_{L_{\beta}^{\infty}}\leq C\Big\lVert \lVert f\rVert_{L_{x,v}^2}\lVert f_j\rVert_{L_{x,v}^2} f_j\Big\rVert_{L_{\beta}^{\infty}}\leq C\lVert f\rVert_{L_{x,v}^2}.
\end{equation*}
Since $\beta>4$, the norm $L_{\beta-1}^{\infty}$ is stronger than the norm $L_{x,v}^2$. This implies
\begin{equation}\label{eq:general_data_control_2}
\Big\lVert  \langle f,f_j\rangle_{L_{x,v}^2}f_j\Big\rVert_{L_{\beta}^{\infty}}\leq  C\lVert f\rVert_{L_{x,v}^2}\leq C\lVert f\rVert_{L_{\beta-1}^{\infty}}\leq C\lVert (1+|v|)^{-1}f\rVert_{L_{\beta}^{\infty}}.
\end{equation}
Combining \eqref{eq:general_data_control} with \eqref{eq:general_data_control_2}, we derive the first inequality in \eqref{eq:general_data_control_0}
\begin{equation}\label{eq:general_data_control_3}
\lVert \mathcal{P}f\rVert_{L_{\beta}^{\infty}}\leq C\lVert (1+|v|)^{-1}f\rVert_{L_{\beta}^{\infty}}.
\end{equation}
For the second inequality in \eqref{eq:general_data_control_0}, using the triangle inequality we have
\begin{equation*}
\lVert (1+|v|)^{-1}(\id-\mathcal{P})f  \rVert_{L_{\beta}^\infty}\leq \lVert (1+|v|)^{-1}f  \rVert_{L_{\beta}^\infty}+\lVert (1+|v|)^{-1}\mathcal{P}f \rVert_{L_{\beta}^\infty}.
\end{equation*}
By \eqref{eq:general_data_control_3} there is
\begin{equation*}
    \lVert (1+|v|)^{-1}\mathcal{P}f  \rVert_{L_{\beta}^\infty}\leq \lVert \mathcal{P}f\rVert_{L_{\beta}^\infty}\leq C\lVert (1+|v|)^{-1}f\rVert_{L_{\beta}^{\infty}}
\end{equation*}
This implies the second inequality in \eqref{eq:general_data_control_0}, and thus concludes the proof of the theorem.
\end{proof}

\section{Control of the Nonlinearity}\label{sec:nonliearity}
This section is devoted to the control of the nonlinear terms in the fixed-point map $\Gamma$ in definition \ref{def:fixed_point}. Subsection \ref{sec:biased_operator} proves the control of the biased collision operator. Subsection \ref{sec:convo_nonlinear_sigma1} further gives the estimate of nonlinear terms needed for Theorem \ref{th:main_theorem_1}, while Subsection \ref{sec:convo_nonlinear_sigma0} is for Theorem \ref{th:main_theorem_2}. The framework of the estimates in this section originates from \cite{Ukai_2006}, with some additional analysis to handle the coupled Boltzmann equations.

\subsection{Estimates for the Biased Collision Operator}\label{sec:biased_operator}
\begin{lemma}\label{lem:multiplicative_nonlinearity}
For any parameter $\beta>4$ and functions $\eta,\psi_1,\psi_2\in L_{\beta}^{\infty}$, the biased collision operator $\mathcal{Q}_{\eta}(\psi_1,\psi_2)$ is bounded from above by
\begin{equation}\label{eq:multiplicative_nonlinearity}
\lVert (1+|v|)^{-1}\mathcal{Q}_{\eta}(\psi_1,\psi_2) \rVert_{L_\beta^{\infty}}\leq C\lVert \eta\rVert_{L_{\beta}^{\infty}}\lVert \psi_1\rVert_{L_{\beta}^{\infty}}\lVert \psi_2 \rVert_{L_{\beta}^{\infty}}.
\end{equation}
\end{lemma}
\begin{proof}
Recall the definition \eqref{eq:biased_collision_operator} of the biased collision operator
\begin{equation*}
\begin{split}
&\mathcal{Q}_\eta(\psi_1,\psi_2):=\frac{1}{2}\int \big((v_*-v)\cdot\omega\big)_+ \eta(v_*)\Big[\psi_1(v')\psi_2(v_*')+\psi_2(v')\psi_1(v_*')\\
&\qquad\qquad\qquad\qquad\qquad\qquad\qquad\qquad\ \ \ -\psi_1(v)\psi_2(v_*)-\psi_2(v)\psi_1(v_*)\Big]d\omega dv_*.
\end{split}
\end{equation*}
The operator is a summation of four terms. We will detail the proof of the following inequality
\begin{equation}\label{eq:one_of_four}
\Big\lVert(1+|v|)^{-1}\Big[\int \big((v_*-v)\cdot\omega\big)_+ \eta(v_*)\psi_1(v')\psi_2(v_*')d\omega dv_*\Big]\Big\rVert_{L_{\beta}^{\infty}}\leq C\lVert \eta\rVert_{L_{\beta}^{\infty}}\lVert \psi_1\rVert_{L_{\beta}^{\infty}}\lVert \psi_2 \rVert_{L_{\beta}^{\infty}}.
\end{equation}
The inequality above gives an upper bound for one of the four terms in the biased collision operator. The proof of upper bound for the other three terms is essentially the same. These together imply \eqref{eq:multiplicative_nonlinearity}.

By the definition \eqref{eq:weighted_L_infty} of the $L_{\beta}^{\infty}$ norm, we have for arbitrary function $g\in L_{\beta}^{\infty}$
\begin{equation*}
|g(v)|\leq \lVert g\rVert_{L_{\beta}^{\infty}} (1+|v|^{\beta})^{-1}.
\end{equation*}
Using the inequality above, we deduce
\begin{equation*}
\begin{split}
&\Big\lVert(1+|v|)^{-1}\Big[\int \big((v_*-v)\cdot\omega\big)_+ \eta(v_*)\psi_1(v')\psi_2(v_*')d\omega dv_*\Big]\Big\rVert_{L_{\beta}^{\infty}}\\
\leq &\Big\lVert  \int_{\mathds{R}^d\times \mathds{S}^{2}} (1+|v|)^{-1}\big((v_*-v)\cdot\omega\big)_+(1+|v_*|)^{-\beta}(1+|v'|)^{-\beta}(1+|v_*'|)^{-\beta} d\omega dv_*\Big\rVert_{L_{\beta}^\infty}\\
&\hspace{9cm}\times \lVert\eta \rVert_{L_{\beta}^{\infty}}\lVert \psi_1\rVert_{L_{\beta}^{\infty}}\lVert \psi_2 \rVert_{L_{\beta}^{\infty}}.
\end{split}
\end{equation*}
According to the definition \eqref{eq:weighted_L_infty} of the $L_{\beta}^{\infty}$ norm, this term is less than
\begin{equation}\label{eq:one_of_four_1}
\begin{split}
&\Big\lVert  \int_{\mathds{R}^d\times \mathds{S}^{d-1}} (1+|v|)^{-1}\big((v_*-v)\cdot\omega\big)_+(1+|v|)^{\beta}(1+|v_*|)^{-\beta}(1+|v'|)^{-\beta}(1+|v_*'|)^{-\beta} d\omega dv_*\Big\rVert_{L^\infty}\\
&\hspace{10cm}\times \lVert\eta \rVert_{L_{\beta}^{\infty}}\lVert \psi_1\rVert_{L_{\beta}^{\infty}}\lVert \psi_2 \rVert_{L_{\beta}^{\infty}}.
\end{split}
\end{equation}
To control the $L^\infty$-norm, we write
\begin{equation*}
\Big((v_*-v)\cdot\omega\Big)_+\leq |v-v_*|\leq |v|+|v_*|.
\end{equation*}
With the inequality above, we are able to control the $L^{\infty}$ norm in \eqref{eq:one_of_four_1}
\begin{equation*}
\begin{split}
&\Big\lVert  \int_{\mathds{R}^d\times \mathds{S}^{d-1}} (1+|v|)^{-1}\big((v-v_*)\cdot\omega\big)_+(1+|v|)^\beta(1+|v_*|)^{-\beta}(1+|v'|)^{-\beta}(1+|v_*'|)^{-\beta} d\omega dv_*\Big\rVert_{L^{\infty}}\\
\leq &\Big\lVert  \int_{\mathds{R}^d\times \mathds{S}^{d-1}} (1+|v|)^{-1}(|v|+|v_*|)(1+|v|)^\beta(1+|v_*|)^{-\beta}(1+|v'|)^{-\beta}(1+|v_*'|)^{-\beta} d\omega dv_*\Big\rVert_{L^{\infty}}.
\end{split}
\end{equation*}
By the definition of the pre-collisional configuration $(v',v_*')$, we have
\begin{equation*}
(1+|v'|)^{-\beta}(1+|v_*'|)^{-\beta}\leq (1+|v'|+|v_*'|+|v'||v_*'|)^{-\beta}\leq C(1+|v|)^{-\beta}.
\end{equation*}
Using the fact that $\beta>4$, the inequality above further implies
\begin{equation*}
\begin{split}
&\Big\lVert  \int_{\mathds{R}^d\times \mathds{S}^{d-1}} (1+|v|)^{-1}\big((v-v_*)\cdot\omega\big)_+(1+|v|)^\beta(1+|v_*|)^{-\beta}(1+|v'|)^{-\beta}(1+|v_*'|)^{-\beta} d\omega dv_*\Big\rVert_{L^{\infty}}\\
\leq &C\Big \lVert  \int_{\mathds{R}^d\times \mathds{S}^{d-1}} (1+|v|^{-1})(|v|+|v_*|)(1+|v_*|)^{-\beta} d\omega dv_*\Big\rVert_{L^{\infty}}\leq C.
\end{split}
\end{equation*}
This concludes the proof of inequality \eqref{eq:one_of_four}, and thus concludes the proof of the lemma.
\end{proof}
Based on Lemma \ref{lem:multiplicative_nonlinearity}, we can derive the corollary below, by replacing the $\eta$ or $\psi_2$ in Lemma \ref{lem:multiplicative_nonlinearity} with the coupled equilibrium $\mathcal{G}$.
\begin{corollary}
For any parameter $\beta>4$ and functions $\eta,\psi_1,\psi_2\in L_{\beta}^{\infty}$, we have
\begin{equation*}
\lVert (1+|v|)^{-1}\mathcal{Q}_{\mathcal{G}}(\psi_1,\psi_2) \rVert_{L_{\beta}^{\infty}}\leq C\lVert \psi_1\rVert_{L_{\beta}^{\infty}}\lVert \psi_2 \rVert_{L_{\beta}^{\infty}},\quad \lVert (1+|v|)^{-1}\mathcal{Q}_{\eta}(\psi_1,\mathcal{G}) \rVert_{L_{\beta}^{\infty}}\leq C\lVert \eta\rVert_{L_{\beta}^{\infty}}\lVert \psi_1 \rVert_{L_{\beta}^{\infty}}.
\end{equation*}
\end{corollary}
From now on the terminal time $t$ is fixed, but all the constants are independent of $t$.
\begin{definition}\label{def:convo_nonlinear}
\textbf{\textup{[Nonlinearity After Convolution] }}For the forward component, we define the convolution of the semigroup $e^{sB^+}$ and the biased collision operator as
\begin{equation*}
\Psi^{+}[\eta,\psi_1,\psi_2](s):=\int_{0}^se^{(s-\tau)\rB}\mathcal{Q}_{\eta(\tau)}\big(\psi_1(\tau),\psi_2(\tau)\big)d\tau.
\end{equation*}
Due to the decomposition \eqref{eq:decomposition_semigroup} of $e^{sB^+}$, we also define the decomposition components of $\Psi^+$
\begin{equation*}
\begin{split}
&\Psi_1^{+}[\eta,\psi_1,\psi_2](s):=\int_{0}^s\mathcal{D}_1^{+}(s-\tau)\mathcal{Q}_{\eta(\tau)}\big(\psi_1(\tau),\psi_2(\tau)\big)d\tau\\
&\Psi_2^{+}[\eta,\psi_1,\psi_2](s):=\int_{0}^s\mathcal{D}_2^{+}(s-\tau)\mathcal{Q}_{\eta(\tau)}\big(\psi_1(\tau),\psi_2(\tau)\big)d\tau
\end{split}
\end{equation*}
For the backward component, we define the convolution of the semigroup $e^{sB^-}$ and the biased collision operator as
\begin{equation*}
\Psi^{-}[\psi,\eta_1,\eta_2](s)=\int_{s}^te^{(\tau-s)\lB}\mathcal{Q}_{\psi(\tau)}\big(\eta_1(\tau),\eta_2(\tau)\big)d\tau.
\end{equation*}
Due to the decomposition \eqref{eq:decomposition_semigroup} of $e^{sB^-}$, we also define the decomposition components of $\Psi^+$
\begin{equation*}
\begin{split}
&\Psi_1^{-}[\eta,\psi_1,\psi_2](s):=\int_{s}^t\mathcal{D}_1^{-}(\tau-s)\mathcal{Q}_{\eta(\tau)}\big(\psi_1(\tau),\psi_2(\tau)\big)d\tau\\
&\Psi_2^{-}[\eta,\psi_1,\psi_2](s):=\int_{s}^t\mathcal{D}_2^{-}(\tau-s)\mathcal{Q}_{\eta(\tau)}\big(\psi_1(\tau),\psi_2(\tau)\big)d\tau
\end{split}
\end{equation*}
\end{definition}
Based on the definitions of $\Psi^+$ and $\Psi^-$, we can rewrite the fixed-point map $\Gamma$ in a more useful form
\begin{equation}\label{eq:fixed_point_map_psi}
\begin{split}
&\rP[\psi_p,\eta_p]\\
=&e^{sB^+}\psi_p(0)-\int_0^se^{(s-\tau)B^+}(\mathcal{G}+\psi_p)\phi  d\tau
+\Psi^+[\eta_p,\psi_p,\psi_p]+2\Psi^+[\eta_p,\psi_p,\mathcal{G}]+\Psi^+[\mathcal{G},\psi_p,\psi_p],\\
&\lP[\psi_p,\eta_p]\\
=&e^{(t-s)B^-}\eta_p(t)-\int_s^te^{(\tau-s)B^-}(\mathcal{G}+\eta_p)\phi  d\tau
+\Psi^+[\psi_p,\eta_p,\eta_p]+2\Psi^+[\psi_p,\eta_p,\mathcal{G}]+\Psi^+[\mathcal{G},\eta_p,\eta_p].
\end{split}
\end{equation}
A proper norm must be chosen to prove that $\Gamma$ is a contraction map. For this purpose we define the norm $P_{\beta}^{\sigma}$.
\begin{definition}\label{def:v-s-t_norm}
We define the norm $\lVert\cdot\rVert_{P_{\beta}^{\sigma}}$ as
\begin{equation}\label{eq:forward_P_norm}
\lVert\psi_p\rVert_{P_{\beta}^{\sigma}}:=\sup_{0\leq s\leq t}\Big(1+s\Big)^\sigma \lVert\psi_p(s)\rVert_{L_{\beta}^{\infty}}.
\end{equation}
We also define the norm $\lVert\cdot\rVert_{E_{\beta}^{\sigma}}$ as
\begin{equation}\label{eq:forward_E_norm}
\lVert\psi_p\rVert_{E_{\beta}^{\sigma}}:=\sup_{0\leq s\leq t}e^{\sigma s} \lVert\psi_p(s)\rVert_{L_{\beta}^{\infty}}.
\end{equation}
Recalling that $t$ is the terminal time, we introduce the time reversal operator $\reverse$ as
\begin{equation*}
\eta_p^{\reverse}(s)=\eta_p(t-s).
\end{equation*}
\end{definition}
According to the definition of $P_{\beta}^{\sigma}$ and $\reverse$, we have
\begin{equation}\label{eq:bound_provided_norm_P}
\lVert \psi_p(s)\rVert_{L_{\beta}^{\infty}}\leq (1+s)^{-\sigma}\lVert \psi_p\rVert_{P_{\beta}^{\sigma}},\qquad\lVert\eta_p(s)\rVert_{L_{\beta}^{\infty}}\leq \Big(1+(t-s)\Big)^{-\sigma}\lVert \eta_p^{\reverse}\rVert_{P_{\beta}^{\sigma}},
\end{equation}
where the second inequality is due to
\begin{equation*}
\begin{split}
\lVert \eta_p^{\reverse}\rVert_{P_{\beta}^{\sigma}}:=\sup_{0\leq s\leq t}(1+s)^{\sigma}\lVert \eta_p^{\reverse}(s)\rVert_{L_{\beta}^{\infty}}=&\sup_{0\leq s\leq t}\big(1+s\big)^{\sigma}\lVert \eta_p(t-s)\rVert_{L_{\beta}^{\infty}}\\
=&\sup_{0\leq s\leq t}\big(1+(t-s)\big)^{\sigma}\lVert \eta_p(s)\rVert_{L_{\beta}^{\infty}}.
\end{split}
\end{equation*}
Similar inequalities are also true for $E_{\beta}^{\sigma}$
\begin{equation}
\lVert \psi_p(s)\rVert_{L_{\beta}^{\infty}}\leq e^{-\sigma s}\lVert \psi_p\rVert_{E_{\beta}^{\sigma}},\qquad\lVert\eta_p(s)\rVert_{L_{\beta}^{\infty}}\leq e^{-\sigma (t-s)}\lVert \eta_p^{\reverse}\rVert_{E_{\beta}^{\sigma}}.
\end{equation}

\subsection{Control of Convolutional Nonlinearity for Theorem \ref{th:main_theorem_1}}\label{sec:convo_nonlinear_sigma1}
In this subsection, we provide the control of $P_{\beta}^{\sigma}$ norm for $\Psi^+$ and $\Psi^-$, which is useful for Theorem \ref{th:main_theorem_1}.

In subsections \ref{sec:convo_nonlinear_sigma1} and \ref{sec:convo_nonlinear_sigma0}, all the estimates of $\Psi^+$ (resp. $\Psi^-$) will be reduced to the estimates of $\Psi_1^+$ and $\Psi_2^+$ (resp. $\Psi_1^-$ and $\Psi_2^-$). 
\begin{lemma}\label{lem:(1,1,1)-estimate}
\textup{\textbf{[Control for the Forward Component]}} For any parameters $\beta>4$, $\sigma>1$, and functions $\eta^{\reverse},\psi_1,\psi_2\in P_{\beta}^{\sigma}$, we have the upper bound for the $P_{\beta}^{\sigma}$ norm of $\Psi^+[\eta,\psi_1,\psi_2]$ as
\begin{equation}\label{eq:(1,1,1)-estimate_2}
\lVert \Psi^{+}[\eta,\psi_1,\psi_2]\rVert_{P_{\beta}^{\sigma}}\leq C\lVert \eta^{\reverse}\rVert_{P_{\beta}^{\sigma}}\lVert \psi_1\rVert_{P_{\beta}^{\sigma}}\lVert \psi_2 \rVert_{P_{\beta}^{\sigma}}.
\end{equation}
\end{lemma}
\begin{proof}
\underline{Estimate of $\Psi_2^+$:} First we would like to prove the upper bound for the $L_{\beta}^{\infty}$ norm of $\Psi_2^+(s)$, with $0\leq s \leq t$. According to the definition of $\Psi^+$, there is
\begin{equation*}
\lVert \Psi_2^{+}[\eta,\psi_1,\psi_2](s)\rVert_{L_{\beta}^{\infty}}= \Big\lVert \int_{0\leq \tau\leq s}\mathcal{D}_2^+(s-\tau)\mathcal{Q}_{\eta(\tau)}\big(\psi_1(\tau),\psi_2(\tau)\big)d\tau \Big\rVert_{L_{\beta}^{\infty}}.
\end{equation*}
Using the boundedness of $\mathcal{D}_2^+(s)$ from $L_{\beta-1}^{\infty}$ to $L_{\beta}^{\infty}$ given by Proposition \ref{prop:semigroup_bound_for_general}, we further have
\begin{equation*}
\lVert \Psi_2^{+}[\eta,\psi_1,\psi_2](s)\rVert_{L_{\beta}^{\infty}}\leq C\int_{0\leq\tau\leq s}\Big\lVert (1+|v|)^{-1}\mathcal{Q}_{\eta(\tau)}\big(\psi_1(\tau),\psi_2(\tau)\big)\Big\rVert_{L_{\beta}^{\infty}} d\tau.
\end{equation*}
To conclude the estimate of the $L_\beta^{\infty}$ norm, we use Lemma \ref{lem:multiplicative_nonlinearity} to control the $L_{\beta}^{\infty}$ norm related to the biased collision operator
\begin{equation}\label{eq:(1,1,1)-estimate_1}
\lVert \Psi_2^{+}[\eta,\psi_1,\psi_2](s)\rVert_{L_{\beta}^{\infty}}\leq C\int_{0\leq \tau\leq s}\lVert \eta(\tau)\rVert_{L_{\beta}^{\infty}}\lVert \psi_1(\tau)\rVert_{L_{\beta}^{\infty}}\lVert \psi_2(\tau) \rVert_{L_{\beta}^{\infty}}d\tau.
\end{equation}
Based on the estimate \eqref{eq:(1,1,1)-estimate_1} of the $L_{\beta}^{\infty}$ norm, we want to further control the $P_{\beta}^{\sigma}$ norm of $\Psi_2^+$. According to the definition \eqref{eq:forward_P_norm} of the $P_{\beta}^{\sigma}$ norm
\begin{equation}\label{eq:W_Psi+_1}
\begin{split}
\lVert \Psi_2^+[\eta,\psi_1,\psi_2]\rVert_{P_{\beta}^{\sigma}}&=\sup_{s\in[0,t]}(1+s)^{\sigma}\lVert \Psi_2^+[\eta,\psi_1,\psi_2](s)\rVert_{L_{\beta}^{\infty}}\\
&\leq \sup_{s\in[0,t]}C(1+s)^{\sigma}\int_{0\leq \tau\leq s}\lVert \eta(\tau)\rVert_{L_{\beta}^{\infty}}\lVert \psi_1(\tau)\rVert_{L_{\beta}^{\infty}}\lVert \psi_2(\tau) \rVert_{L_{\beta}^{\infty}}d\tau\\
&\leq C(1+t)^{\sigma}\int_{0\leq \tau\leq t}\lVert \eta(\tau)\rVert_{L_{\beta}^{\infty}}\lVert \psi_1(\tau)\rVert_{L_{\beta}^{\infty}}\lVert \psi_2(\tau) \rVert_{L_{\beta}^{\infty}}d\tau.
\end{split}
\end{equation}
To get the upper bound in \eqref{eq:(1,1,1)-estimate_2}, we will use the upper bounds \eqref{eq:bound_provided_norm_P} provided by the norm $P_{\beta}^{\sigma}$. The inequality \eqref{eq:bound_provided_norm_P} implies
\begin{equation}\label{eq:three_norm_bound}
\lVert \eta(\tau)\rVert_{L_{\beta}^{\infty}}\lVert \psi_1(\tau)\rVert_{L_{\beta}^{\infty}}\lVert \psi_2(\tau) \rVert_{L_{\beta}^{\infty}}\leq \big(1+(t-\tau)\big)^{-\sigma}(1+\tau)^{-2\sigma}\lVert \eta^{\reverse}\rVert_{P_{\beta}^{\sigma}}\lVert \psi_1\rVert_{P_{\beta}^{\sigma}}\lVert \psi_2 \rVert_{P_{\beta}^{\sigma}}.
\end{equation}
Combining \eqref{eq:three_norm_bound} with \eqref{eq:W_Psi+_1} leads to
\begin{equation*}
\begin{split}
&\lVert \Psi_2^+[\eta,\psi_1,\psi_2]\rVert_{P_{\beta}^{\sigma}}\\
\leq& C(1+t)^{\sigma}\int_{0\leq \tau\leq t}\big(1+(t-\tau)\big)^{-\sigma}(1+\tau)^{-2\sigma}d\tau\lVert \eta^{\reverse}\rVert_{P_{\beta}^{\sigma}}\lVert \psi_1\rVert_{P_{\beta}^{\sigma}}\lVert \psi_2 \rVert_{P_{\beta}^{\sigma}}\\
\leq& C(1+t)^{\sigma}(1+t)^{-\min\{\sigma,2\sigma\}}\lVert \eta^{\reverse}\rVert_{P_{\beta}^{\sigma}}\lVert \psi_1\rVert_{P_{\beta}^{\sigma}}\lVert \psi_2 \rVert_{P_{\beta}^{\sigma}},
\end{split}
\end{equation*}
where the second inequality is due to Lemma \ref{lem:convolution_inequality}. Noticing the simple fact that $\min\{\sigma,2\sigma\}=\sigma$ concludes the estimate of $\Psi_2^+$.

\underline{Estimate of $\Psi_1^+$:} Using the explicit expression \eqref{eq:explicit_expression_D_1} of $\mathcal{D}_1^{+}$, we get
\begin{equation*}
\begin{split}
&|\Psi_1^+[\eta,\psi_1,\psi_2](s,x,v)|\\
\leq& C\int_0^se^{-\nu(v)(s-\tau)}\Big|\sup_{x\in\mathds{T}^d}\mathcal{Q}_{\eta(\tau)}\big(\psi_1(\tau),\psi_2(\tau)\big)(x,v)\Big|d\tau\\
\leq& C\int_0^se^{-\nu(v)(s-\tau)}(1+\tau)^{-2\sigma}\nu(v)(1+|v|)^{-\beta}\underbrace{\Big((1+\tau)^{2\sigma}(1+|v|)^{\beta}\frac{1}{\nu(v)}\Big|\sup_{x\in\mathds{T}^d}\mathcal{Q}_{\eta(\tau)}\big(\psi_1(\tau),\psi_2(\tau)\big)(x,v)\Big|\Big)d\tau}_{\textup{\blue{(I)}}}.
\end{split}
\end{equation*}
Taking the supremum over $\tau$ and $v$ in term (I), along with the fact that $\nu(v)$ is equivalent to $1+|v|$ up to constants, we further have
\begin{equation}\label{eq:Psi_1_initial}
\begin{split}
&\big|\Psi_1^+[\eta,\psi_1,\psi_2](s,x,v)\big|\\
\leq&\int_{0}^se^{-\nu(v)(s-\tau)}(1+\tau)^{-2\sigma}\nu(v)d\tau (1+|v|)^{-\beta} \Big(\sup_{0\leq \tau\leq s}(1+\tau)^{2\sigma}\lVert\frac{1}{\nu(v)}\mathcal{Q}_{\eta(\tau)}\big(\psi_1(\tau),\psi_2(\tau)\big)\rVert_{L_{\beta}^{\infty}}\Big)\\
\leq & C\int_{0}^se^{-\nu(v)(s-\tau)}(1+\tau)^{-2\sigma}\nu(v)d\tau (1+|v|)^{-\beta}\Big(\sup_{0\leq \tau\leq s}(1+\tau)^{2\sigma}\lVert \eta(\tau)\rVert_{L_{\beta}^{\infty}}\lVert \psi_1(\tau)\rVert_{L_{\beta}^{\infty}} \lVert \psi_2(\tau)\rVert_{L_{\beta}^{\infty}}\Big)\\
\leq &C\int_{0}^se^{-\nu(v)(s-\tau)}(1+\tau)^{-2\sigma}\nu(v)d\tau (1+|v|)^{-\beta}\lVert \eta^{\reverse}\rVert_{P_{\beta}^{\sigma}}\lVert \psi_1\rVert_{P_{\beta}^{\sigma}}\lVert \psi_2 \rVert_{P_{\beta}^{\sigma}},
\end{split}
\end{equation}
where in the second inequality we have used Lemma \ref{lem:multiplicative_nonlinearity}. For the time integral in the inequalities above, we decompose the integral and get
\begin{equation}
\begin{split}
&\int_{0}^se^{-\nu(v)(s-\tau)}(1+\tau)^{-2\sigma}\nu(v)d\tau\\
=&\int_{0}^{\frac{s}{2}}e^{-\nu(v)(s-\tau)}(1+\tau)^{-2\sigma}\nu(v)d\tau+\int_{\frac{s}{2}}^se^{-\nu(v)(s-\tau)}(1+\tau)^{-2\sigma}\nu(v)d\tau\\
\leq &e^{-\nu_*\frac{s}{4}}\int_{0}^{\frac{s}{2}}e^{-\nu(v)\frac{(s-\tau)}{2}}(1+\tau)^{-2\sigma}\nu(v)d\tau+(1+\frac{s}{2})^{-2\sigma}\int_{\frac{s}{2}}^se^{-\nu(v)(s-\tau)}\nu(v)d\tau\leq C(1+s)^{-\sigma}.
\end{split}
\end{equation}
The convolution inequality above implies
\begin{equation}
\lVert\Psi_1^+[\eta,\psi_1,\psi_2](s)\rVert_{L_{\beta}^{\infty}}\leq (1+s)^{-\sigma}\lVert \eta^{\reverse}\rVert_{P_{\beta}^{\sigma}}\lVert \psi_1\rVert_{P_{\beta}^{\sigma}}\lVert \psi_2 \rVert_{P_{\beta}^{\sigma}}.
\end{equation}
Thus by the definition \eqref{eq:forward_P_norm} of the $P_{\beta}^{\sigma}$ norm
\begin{equation}
\lVert \Psi_1^{+}[\eta,\psi_1,\psi_2]\rVert_{P_{\beta}^{\sigma}}\leq C\lVert \eta^{\reverse}\rVert_{P_{\beta}^{\sigma}}\lVert \psi_1\rVert_{P_{\beta}^{\sigma}}\lVert \psi_2 \rVert_{P_{\beta}^{\sigma}}.
\end{equation}
This along with the estimate of $\Psi_2^+$ concludes the proof.
\end{proof}

\begin{lemma}\label{lem:(1,1,0)}
\textup{\textbf{[Control for the Forward Component]}} For any parameters $\beta>4,\sigma>1$, and functions $\eta^{\reverse},\psi_1\in P_{\beta}^{\sigma}$, we have the upper bound for the $P_{\beta}^{\sigma}$ norm of $\Psi^+[\eta,\psi_1,\mathcal{G}]$ as
\begin{equation*}
\lVert \Psi^{+}[\eta,\psi_1,\mathcal{G}]\rVert_{P_{\beta}^{\sigma}}\leq C\lVert \eta^{\reverse}\rVert_{P_{\beta}^{\sigma}}\lVert \psi_1\rVert_{P_{\beta}^{\sigma}}.
\end{equation*}
\end{lemma}

\begin{proof}
\underline{Estimate of $\Psi_2^+$:} Using exactly the same method of proving equation \eqref{eq:(1,1,1)-estimate_1} in Lemma \ref{lem:(1,1,1)-estimate}, we have the upper bound for the $L_{\beta}^{\infty}$ norm of $\Psi^+[\eta,\psi_1,\mathcal{G}]$ as
\begin{equation*}
\begin{split}
\lVert \Psi_2^{+}[\eta,\psi_1,\mathcal{G}](s)\rVert_{L_{\beta}^{\infty}}&\leq C\int_{0\leq \tau\leq s}\lVert \eta(\tau)\rVert_{L_{\beta}^{\infty}}\lVert \psi_1(\tau)\rVert_{L_{\beta}^{\infty}}d\tau.
\end{split}
\end{equation*}
This further implies 
\begin{equation*}
\begin{split}
\lVert \Psi_2^{+}[\eta,\psi_1,\mathcal{G}](s)\rVert_{P_{\beta}^{\sigma}}&\leq C\sup_{0\leq s\leq t}(1+s)^{\sigma}\int_{0\leq \tau\leq s}\lVert \eta(\tau)\rVert_{L_{\beta}^{\infty}}\lVert \psi_1(\tau)\rVert_{L_{\beta}^{\infty}}d\tau\\
&\leq C(1+t)^{\sigma}\int_{0\leq \tau\leq t}\lVert \eta(\tau)\rVert_{L_{\beta}^{\infty}}\lVert \psi_1(\tau)\rVert_{L_{\beta}^{\infty}}d\tau.
\end{split}
\end{equation*}
Again we use the upper bounds \eqref{eq:bound_provided_norm_P} provided by the norm $P_{\beta}^{\sigma}$, just as in the proof of Lemma \ref{lem:(1,1,1)-estimate},
\begin{equation*}
\begin{split}
\lVert \Psi_2^{+}[\eta,\psi_1,\mathcal{G}]\rVert_{P_{\beta}^{\sigma}}\leq& C(1+t)^{\sigma}\int_{0\leq \tau\leq t}\big(1+(t-\tau)\big)^{-\sigma}(1+\tau)^{-\sigma}\lVert \eta^{\reverse}\rVert_{P_{\beta}^{\sigma}}\lVert \psi_1\rVert_{P_{\beta}^{\sigma}}d\tau\\
\leq& C(1+t)^{\sigma}(1+t)^{-\min\{\sigma,\sigma\}}\lVert \eta^{\reverse}\rVert_{P_{\beta}^{\sigma}}\lVert \psi_1\rVert_{P_{\beta}^{\sigma}}.
\end{split}
\end{equation*}
Noticing the simple fact that $\min\{\sigma,\sigma\}=\sigma$ concludes the estimate of $\Psi_2^+$.

\underline{Estimate of $\Psi_1^+$:} The estimate of $\Psi_1^+[\eta,\psi_1,\mathcal{G}]$ is similar to the estimate of $\Psi_1^+[\eta,\psi_1,\psi_2]$. Similar to \eqref{eq:Psi_1_initial}, we have
\begin{equation}
\big|\Psi_1^+[\eta,\psi_1,\mathcal{G}](s,x,v)\big|\leq C\int_{0}^se^{-\nu(v)(s-\tau)}(1+\tau)^{-\sigma}\nu(v)d\tau (1+|v|)^{-\beta}\lVert \eta^{\reverse}\rVert_{P_{\beta}^{\sigma}}\lVert \psi_1\rVert_{P_{\beta}^{\sigma}}.
\end{equation}
Then the time integral gives a factor $(1+s)^{-\sigma}$
\begin{equation}
\lVert\Psi_1^+[\eta,\psi_1,\mathcal{G}](s)\rVert_{L_{\beta}^{\infty}}\leq (1+s)^{-\sigma}\lVert \eta^{\reverse}\rVert_{P_{\beta}^{\sigma}}\lVert \psi_1\rVert_{P_{\beta}^{\sigma}}.
\end{equation}
This eventually gives
\begin{equation}
\lVert \Psi_1^{+}[\eta,\psi_1,\mathcal{G}]\rVert_{P_{\beta}^{\sigma}}\leq C\lVert \eta^{\reverse}\rVert_{P_{\beta}^{\sigma}}\lVert \psi_1\rVert_{P_{\beta}^{\sigma}}.
\end{equation}
This along with the estimate of $\Psi_2^+$ conclude the proof.
\end{proof}

\begin{lemma}\label{lem:(0,1,1)}
\textup{\textbf{[Control for the Forward Component]}} For any parameters $\beta>4,\sigma>1$, and functions $\psi_1,\psi_2\in P_{\beta}^{\sigma}$, we have the upper bound for the $P_{\beta}^{\sigma}$ norm of $\Psi^+[\mathcal{G},\psi_1,\psi_2]$ as
\begin{equation*}
\lVert \Psi^{+}[\mathcal{G},\psi_1,\psi_2]\rVert_{P_{\beta}^{\sigma}}\leq C\lVert \psi_1\rVert_{P_{\beta}^{\sigma}}\lVert \psi_2\rVert_{P_{\beta}^{\sigma}}.
\end{equation*}
\end{lemma}
\begin{proof}
\underline{Estimate of $\Psi_2^+$:} The proof of this lemma is slightly different from the proof of Lemma \ref{lem:(1,1,1)-estimate} and Lemma \ref{lem:(1,1,0)}. Again due to the definition of $\Psi^+$ there is
\begin{equation}\label{eq:second_order_1}
\lVert \Psi_2^{+}[\mathcal{G},\psi_1,\psi_2](s)\rVert_{L_{\beta}^{\infty}}= \Big\lVert \int_{0\leq \tau\leq s}e^{(s-\tau)\rB}\mathcal{Q}_{\mathcal{G}}\big(\psi_1(\tau),\psi_2(\tau)\big)d\tau \Big\rVert_{L_{\beta}^{\infty}}.
\end{equation}
The reason for the difference is that for arbitrary $\tau\geq 0$ we have
\begin{equation}\label{eq:second_order_ortho}
\mathcal{Q}_{\mathcal{G}}\big(\psi_1(\tau),\psi_2(\tau)\big)\in\mathcal{K}^{\perp}.
\end{equation}
This equality can be verified as the following. Suppose $h$ is a collision invariant. For each $x$ we perform the integration over $v$ 
\begin{equation*}
\begin{split}
&\int \mathcal{G}h(v)\mathcal{Q}_{\mathcal{G}}\big(\psi_1(\tau),\psi_2(\tau)\big)(v)dvdx\\
=&\frac{1}{2}\int \big((v_*-v)\cdot\omega\big)_+ h(v)\mathcal{G}(v)\mathcal{G}(v_*)\\
&\hspace{2.5cm}\times\Big(\psi_1(v')\psi_2(v_*')+\psi_2(v')\psi_1(v_*')-\psi_1(v)\psi_2(v_*)-\psi_2(v)\psi_1(v_*)\Big)d\omega dv_*dvdx.
\end{split}
\end{equation*}
Using the symmetry of the collision measure
\begin{equation*}
\big((v_*-v)\cdot\omega\big)_+d\omega dv_*dv,
\end{equation*}
we can perform the change of variables $(v,v_*,\omega)\mapsto (v_*,v,-\omega)$ or $(v,v_*,\omega)\mapsto (v',v_*',-\omega)$. Then we would have
\begin{equation*}
\begin{split}
&\int \mathcal{G}h(v)\mathcal{Q}_{\mathcal{G}}\big(\psi_1(\tau),\psi_2(\tau)\big)(v)dvdx\\
=&\frac{1}{8}\int \big((v_*-v)\cdot\omega\big)_+ \mathcal{G}(v)\mathcal{G}(v_*)\big(h(v)+h(v_*)-h(v')-h(v_*')\big)\\
&\hspace{2.5cm}\times\Big(\psi_1(v')\psi_2(v_*')+\psi_2(v')\psi_1(v_*')-\psi_1(v)\psi_2(v_*)-\psi_2(v)\psi_1(v_*)\Big)d\omega dv_*dvdx.
\end{split}
\end{equation*}
Since $h$ is a collision invariant, we have $h(v')+h(v_*')-h(v)-h(v_*)=0$. This implies 
\begin{equation*}
\langle f_i, \mathcal{Q}_{\mathcal{G}}\big(\psi_1(\tau),\psi_2(\tau)\big)\rangle_{L_{x,v}^2}=0
\end{equation*}
for arbitrary $0\leq i\leq d+1$, where $\{f_i\}_{0\leq i\leq d+1}$ is the basis of the kernel $\mathcal{K}$. Thus there is $\mathcal{Q}_{\mathcal{G}}\big(\psi_1(\tau),\psi_2(\tau)\big)\in\mathcal{K}^{\perp}$.

Using the orthogonality \eqref{eq:second_order_ortho} and Lemma \ref{lem:L_beta_decay}, we can transform \eqref{eq:second_order_1} into
\begin{equation*}
\lVert \Psi_2^{+}[\mathcal{G},\psi_1,\psi_2](s)\rVert_{L_{\beta}^{\infty}}\leq C\int_{0\leq \tau\leq s}e^{-\nu_*(s-\tau)}\Big\lVert (1+|v|)^{-1}\mathcal{Q}_{\mathcal{G}}\big(\psi_1(\tau),\psi_2(\tau)\big)\Big\rVert_{L_{\beta}^{\infty}}d\tau.
\end{equation*}
We can further control the $L_{\beta}^{\infty}$ norm of the biased collision term by Lemma \ref{lem:multiplicative_nonlinearity}
\begin{equation}\label{eq:(0,1,1)-estimate_2}
\lVert \Psi_2^{+}[\mathcal{G},\psi_1,\psi_2](s)\rVert_{L_{\beta}^{\infty}}\leq C\int_{0\leq \tau\leq s}e^{-\nu_*(s-\tau)}\lVert \psi_1(\tau)\rVert_{L_{\beta}^{\infty}}\lVert \psi_2(\tau)\rVert_{L_{\beta}^{\infty}}d\tau.
\end{equation}
According to the definition of the $P_{\beta}^{\sigma}$ norm, we have the inequality
\begin{equation*}
\begin{split}
\lVert \Psi_2^{+}[\mathcal{G},\psi_1,\psi_2]\rVert_{P_{\beta}^{\sigma}}\leq& C(1+s)^{\sigma}\int_{0\leq \tau\leq s}e^{-\nu_*(s-\tau)}(1+\tau)^{-2\sigma}d\tau\lVert \psi_1\rVert_{P_{\beta}^{\sigma}}\lVert \psi_2 \rVert_{P_{\beta}^{\sigma}}\\
\leq & C(1+s)^{\sigma}(1+s)^{-2\sigma}\lVert \psi_1\rVert_{P_{\beta}^{\sigma}}\lVert \psi_2 \rVert_{P_{\beta}^{\sigma}}.
\end{split}
\end{equation*}
This concludes the estimate of $\Psi_2^+$.

\underline{Estimate of $\Psi_1^+$:} The estimate of $\Psi_1^+[\mathcal{G},\psi_1,\psi_2]$ is also similar to the estimate of $\Psi_1^+[\eta,\psi_1,\psi_2]$. As in \eqref{eq:Psi_1_initial}, we have
\begin{equation}
\lVert\Psi_1^+[\mathcal{G},\psi_1,\psi_2](s)\rVert_{L_{\beta}^{\infty}}\leq C\int_{0}^se^{-\nu(v)(s-\tau)}(1+\tau)^{-\sigma}\nu(v)d\tau (1+|v|)^{-\beta}\lVert \psi_1\rVert_{P_{\beta}^{\sigma}}\lVert \psi_2\rVert_{P_{\beta}^{\sigma}}.
\end{equation}
Consequently we get
\begin{equation*}
\lVert\Psi_1^+[\mathcal{G},\psi_1,\psi_2](s)\rVert_{P_{\beta}^{\sigma}}\leq C\lVert \psi_1\rVert_{P_{\beta}^{\sigma}}\lVert \psi_2\rVert_{P_{\beta}^{\sigma}}
\end{equation*}
This along with the estimate of $\Psi_2^+$ concludes the proof.
\end{proof}

Since the evolution of the perturbations $(\psi_p,\eta_p)$ is symmetric, we straightforwardly have the lemma below.
\begin{lemma}
\textup{\textbf{[Control for the Backward Component]}} For any parameters $\beta>4$, $\sigma>1$, and functions $\psi,\eta_1^{\reverse},\eta_2^{\reverse}\in P_{\beta}^{\sigma}$, we have the following upper bounds for the $P_{\beta}^{\sigma}$ norm of various terms
\begin{equation*}
\left\{
\begin{split}
&\lVert (\Psi^{-}[\psi,\eta_1,\eta_2])^{\reverse}\rVert_{P_{\beta}^{\sigma}}\leq C\lVert \psi\rVert_{P_{\beta}^{\sigma}}\lVert \eta_1^{\reverse}\rVert_{P_{\beta}^{\sigma}}\lVert \eta_2^{\reverse} \rVert_{P_{\beta}^{\sigma}},\\
&\lVert (\Psi^{-}[\psi,\eta_1,\mathcal{G}])^{\reverse}\rVert_{P_{\beta}^{\sigma}}\leq C\lVert \psi\rVert_{P_{\beta}^{\sigma}}\lVert \eta_1^{\reverse}\rVert_{P_{\beta}^{\sigma}},\\
&\lVert (\Psi^{-}[\mathcal{G},\eta_1,\eta_2])^{\reverse}\rVert_{P_{\beta}^{\sigma}}\leq C\lVert \eta_1^{\reverse}\rVert_{P_{\beta}^{\sigma}}\lVert \eta_2^{\reverse}\rVert_{P_{\beta}^{\sigma}}.
\end{split}
\right.
\end{equation*}
\end{lemma}

\subsection{Control of Convolutional Nonlinearity for Theorem \ref{th:main_theorem_2}}\label{sec:convo_nonlinear_sigma0}
In this section we provide the control of the $E_{\beta}^0$-norm (see Definition \ref{def:v-s-t_norm}) of $\Psi^+$ and the $E_{\beta}^{-\sigma}$ norm of $\Psi^-$, which is useful for Theorem \ref{th:main_theorem_2}.
\begin{lemma}\label{lem:convolu_nonlinear_E_+}
\textup{\textbf{[Control for the Forward Component]}} For any parameters $\beta>4,\sigma>0$, functions $\psi_1,\psi_2\in E_{\beta}^{0}$ and $\eta^{\reverse}\in E_{\beta}^{-\sigma}$, we have the following upper bounds for the $E_{\beta}^{0}$ norm of various terms
\begin{equation}\label{eq:convolu_nonlinear_E_+}
\left\{
\begin{split}
&\lVert \Psi^{+}[\eta,\psi_1,\psi_2]\rVert_{E_{\beta}^{0}}\leq C e^{\sigma t}\lVert \eta^{\reverse}\rVert_{E_{\beta}^{-\sigma}}\lVert \psi_1\rVert_{E_{\beta}^{0}}\lVert \psi_2 \rVert_{E_{\beta}^{0}},\\
&\lVert \Psi^{+}[\eta,\psi_1,\mathcal{G}]\rVert_{E_{\beta}^{0}}\leq C e^{\sigma t}\lVert \eta^{\reverse}\rVert_{E_{\beta}^{-\sigma}}\lVert \psi_1\rVert_{E_{\beta}^{0}},\\
&\lVert \Psi^{+}[\mathcal{G},\psi_1,\psi_2]\rVert_{E_{\beta}^{0}}\leq C \lVert \psi_1\rVert_{E_{\beta}^{0}}\lVert \psi_2 \rVert_{E_{\beta}^{0}}.
\end{split}
\right.
\end{equation}
\end{lemma}
\begin{proof}
\ \\
\underline{The first inequality in \eqref{eq:convolu_nonlinear_E_+}:}\\
\myuwave{The $\Psi_2^+$ Term:} We use the upper bound for the $L_{\beta}^\infty$ norm of $\Psi_2^+$, which has been proved in \eqref{eq:(1,1,1)-estimate_1}
\begin{equation*}
\lVert \Psi_2^{+}[\eta,\psi_1,\psi_2](s)\rVert_{L_{\beta}^{\infty}}\leq C\int_{0\leq \tau\leq s}\lVert \eta(\tau)\rVert_{L_{\beta}^{\infty}}\lVert \psi_1(\tau)\rVert_{L_{\beta}^{\infty}}\lVert \psi_2(\tau) \rVert_{L_{\beta}^{\infty}}d\tau.
\end{equation*}
Since the $E_{\beta}^{0}$ norm is defined as the supremum of $L_{\beta}^{\infty}$ for different time $0\leq s\leq t$, we have
\begin{equation*}
\lVert \Psi_2^{+}[\eta,\psi_1,\psi_2]\rVert_{E_{\beta}^{0}}\leq C\int_{0\leq \tau\leq t}e^{\sigma(t-\tau)}\lVert \eta^{\reverse} \rVert_{E_{\beta}^{-\sigma}}\lVert \psi_1\rVert_{E_{\beta}^{0}}\lVert \psi_2 \rVert_{E_{\beta}^{0}}d\tau\leq C e^{\sigma t}\lVert \eta^{\reverse}\rVert_{E_{\beta}^{-\sigma}}\lVert \psi_1\rVert_{E_{\beta}^{0}}\lVert \psi_2 \rVert_{E_{\beta}^{0}}.
\end{equation*}
\myuwave{The $\Psi_1^+$ term:} For $\Psi_1^+$, we have
\begin{equation}\label{eq:Psi_1_initial_Exp}
\begin{split}
&\big|\Psi_1^+[\eta,\psi_1,\psi_2](s,x,v)\big|\\
\leq&C\int_{0}^se^{-\nu(v)(s-\tau)}e^{\sigma (t-\tau)}\nu(v)d\tau (1+|v|)^{-\beta} \Big(\sup_{0\leq \tau\leq s}e^{-\sigma(t-\tau)}\lVert\frac{1}{1+|v|}\mathcal{Q}_{\eta(\tau)}\big(\psi_1(\tau),\psi_2(\tau)\big)\rVert_{L_{\beta}^{\infty}}\Big)\\
\leq &C \int_{0}^se^{-\nu(v)(s-\tau)}e^{\sigma(t-\tau)}\nu(v)d\tau (1+|v|)^{-\beta}\Big(\sup_{0\leq \tau\leq s}e^{-\sigma (t-\tau)}\lVert \eta(\tau)\rVert_{L_{\beta}^{\infty}}\lVert \psi_1(\tau)\rVert_{L_{\beta}^{\infty}} \lVert \psi_2(\tau)\rVert_{L_{\beta}^{\infty}}\Big)\\
\leq &e^{\sigma t}\int_{0}^se^{-\nu(v)(s-\tau)}e^{-\sigma \tau}\nu(v)d\tau (1+|v|)^{-\beta}\lVert \eta^{\reverse}\rVert_{E_{\beta}^{-\sigma}}\lVert \psi_1\rVert_{E_{\beta}^{0}}\lVert \psi_2 \rVert_{E_{\beta}^{0}},
\end{split}
\end{equation}
The time integral is uniformly bounded for all $s$ and $v$
\begin{equation*}
\int_{0}^se^{-\nu(v)(s-\tau)}e^{-\sigma \tau}\nu(v)d\tau\leq C.
\end{equation*}
Consequently
\begin{equation*}
\lVert \Psi_1^{+}[\eta,\psi_1,\psi_2]\rVert_{E_{\beta}^{0}}\leq C e^{\sigma t}\lVert \eta^{\reverse}\rVert_{E_{\beta}^{-\sigma}}\lVert \psi_1\rVert_{E_{\beta}^{0}}\lVert \psi_2 \rVert_{E_{\beta}^{0}}.
\end{equation*}
This concludes the proof of the first inequality.

\underline{The second inequality in \eqref{eq:convolu_nonlinear_E_+}:} It can be derived by replacing $\psi_2$ with $\mathcal{G}$.

\underline{The third inequality in \eqref{eq:convolu_nonlinear_E_+}:}\\
\myuwave{The $\Psi_2^+$ term:} We first use the following upper bound for the $L_{\beta}^{\infty}$, which has been proved as \eqref{eq:(0,1,1)-estimate_2} in Lemma \ref{lem:(0,1,1)}
\begin{equation*}
\lVert \Psi_2^{+}[\mathcal{G},\psi_1,\psi_2](s)\rVert_{L_{\beta}^{\infty}}\leq C\int_{0\leq \tau\leq s}e^{-\nu_*(s-\tau)}\lVert \psi_1(\tau)\rVert_{L_{\beta}^{\infty}}\lVert \psi_2(\tau)\rVert_{L_{\beta}^{\infty}}d\tau.
\end{equation*}
Then due to the definition of the $E_{\beta}^0$ norm, we have
\begin{equation*}
\lVert \Psi_2^{+}[\mathcal{G},\psi_1,\psi_2]\rVert_{E_{\beta}^{0}}\leq C\int_{0}^te^{-\nu_*(s-\tau)}\lVert \psi_1\rVert_{E_{\beta}^{0}}\lVert \psi_2\rVert_{E_{\beta}^{0}}d\tau\leq C\lVert \psi_1\rVert_{E_{\beta}^{0}}\lVert \psi_2\rVert_{E_{\beta}^{0}}.
\end{equation*}
This concludes the estimate of $\Psi_2^+$.

\myuwave{The $\Psi_1^+$ term:} Similar to \eqref{eq:Psi_1_initial_Exp}, we obtain
\begin{equation*}
\begin{split}
&\big|\Psi_1^+[\mathcal{G},\psi_1,\psi_2](s,x,v)\big|\\
\leq&C\int_{0}^se^{-\nu(v)(s-\tau)}\nu(v)d\tau (1+|v|)^{-\beta} \Big(\sup_{0\leq \tau\leq s}\lVert\frac{1}{1+|v|}\mathcal{Q}_{\mathcal{G}}\big(\psi_1(\tau),\psi_2(\tau)\big)\rVert_{L_{\beta}^{\infty}}\Big)\\
\leq &C\int_{0}^se^{-\nu(v)(s-\tau)}\nu(v)d\tau (1+|v|)^{-\beta}\lVert \psi_1\rVert_{E_{\beta}^{0}}\lVert \psi_2 \rVert_{E_{\beta}^{0}},
\end{split}
\end{equation*}
The time integral gives a uniformly bounded constant. As a consequence
\begin{equation*}
\lVert \Psi_1^{+}[\mathcal{G},\psi_1,\psi_2]\rVert_{E_{\beta}^{0}}\leq C \lVert \psi_1\rVert_{E_{\beta}^{0}}\lVert \psi_2 \rVert_{E_{\beta}^{0}}.
\end{equation*}
This, along with the estimate of $\Psi_2^+$, concludes the proof.

\end{proof}

\begin{lemma}\label{lem:convolu_nonlinear_E_-}
\textup{\textbf{[Control for the Backward Component]}} For any parameters $\beta>4,\sigma>0$, any functions $\psi\in E_{\beta}^{0}$ and $\eta_1^{\reverse},\eta_2^{\reverse}\in E_{\beta}^{-\sigma}$, we have the following upper bounds for the norm of various terms
\begin{equation}\label{eq:convolu_nonlinear_E_-}
\left\{
\begin{split}
&\lVert (\Psi^{-}[\psi,\eta_1,\eta_2])^{\reverse}\rVert_{E_{\beta}^{-\sigma}}\leq Ce^{\sigma t}\lVert \psi\rVert_{E_{\beta}^{0}}\lVert \eta_1^{\reverse}\rVert_{E_{\beta}^{-\sigma}}\lVert \eta_2^{\reverse} \rVert_{E_{\beta}^{-\sigma}},\\
&\lVert (\Psi^{-}[\psi,\eta_1,\mathcal{G}])^{\reverse}\rVert_{E_{\beta}^{-\sigma}}\leq C\lVert \psi\rVert_{E_{\beta}^{0}}\lVert \eta_1^{\reverse}\rVert_{E_{\beta}^{-\sigma}},\\
&\lVert (\Psi^{-}[\mathcal{G},\eta_1,\eta_2])^{\reverse}\rVert_{E_{\beta}^{-\sigma}}\leq Ce^{\sigma t}\lVert \eta_1^{\reverse}\rVert_{E_{\beta}^{-\sigma}}\lVert \eta_2^{\reverse}\rVert_{E_{\beta}^{-\sigma}}.
\end{split}
\right.
\end{equation}
\end{lemma}
\begin{proof}
\ \\
\underline{The first inequality in \eqref{eq:convolu_nonlinear_E_-}:} \\
\myuwave{The $\Psi_2^-$ term:} Similar to \eqref{eq:(1,1,1)-estimate_1}, we have
\begin{equation*}
\lVert \Psi_2^{-}[\psi,\eta_1,\eta_2](s)\rVert_{L_{\beta}^{\infty}}\leq C\int_{s\leq \tau\leq t}\lVert \psi(\tau)\rVert_{L_{\beta}^{\infty}}\lVert \eta_1(\tau)\rVert_{L_{\beta}^{\infty}}\lVert \eta_2(\tau) \rVert_{L_{\beta}^{\infty}}d\tau.
\end{equation*}
According to the definition \eqref{eq:forward_E_norm} of the relevant norms, the equation above implies
\begin{equation}
\begin{split}
\lVert (\Psi_2^-[\psi,\eta_1,\eta_2])^{\reverse} \rVert_{E_{\beta}^{-\sigma}}&:=\sup_{0\leq s\leq t}e^{-\sigma (t-s)}\lVert \Psi_2^{-}[\psi,\eta_1,\eta_2](s)\rVert_{L_{\beta}^{\infty}}\\
&\leq C\sup_{0\leq s\leq t}e^{-\sigma (t-s)}\int_{s\leq \tau\leq t}\lVert \psi(\tau)\rVert_{L_{\beta}^{\infty}}\lVert \eta_1(\tau)\rVert_{L_{\beta}^{\infty}}\lVert \eta_2(\tau) \rVert_{L_{\beta}^{\infty}}d\tau\\
&\leq C\sup_{0\leq s\leq t}e^{-\sigma (t-s)}\int_{s\leq \tau\leq t}e^{2\sigma(t-\tau)}\lVert \psi\rVert_{E_{\beta}^{0}}\lVert \eta_1\rVert_{E_{\beta}^{-\sigma}}\lVert \eta_2 \rVert_{E_{\beta}^{-\sigma}}d\tau
\end{split}
\end{equation}
After the integrating over time $s\leq \tau\leq t$, there is
\begin{equation}
\begin{split}
\lVert (\Psi_2^-[\psi,\eta_1,\eta_2])^{\reverse} \rVert_{E_{\beta}^{-\sigma}}&\leq C\sup_{0\leq s\leq t}e^{\sigma (t-s)}\lVert \psi\rVert_{E_{\beta}^{0}}\lVert \eta_1\rVert_{E_{\beta}^{-\sigma}}\lVert \eta_2 \rVert_{E_{\beta}^{-\sigma}}\\
&\leq Ce^{\sigma t}\lVert \psi\rVert_{E_{\beta}^{0}}\lVert \eta_1^{\reverse}\rVert_{E_{\beta}^{-\sigma}}\lVert \eta_2^{\reverse} \rVert_{E_{\beta}^{-\sigma}}.
\end{split}
\end{equation}
\myuwave{The $\Psi_1^-$ term:} For $\Psi_1^-$, we have
\begin{equation}\label{eq:Psi_1_initial_exp_-}
\begin{split}
&\big|\Psi_1^-[\eta,\psi_1,\psi_2](s,x,v)\big|\\
\leq&C\int_{s}^te^{-\nu(v)(\tau-s)}e^{2\sigma (t-\tau)}\nu(v)d\tau (1+|v|)^{-\beta} \Big(\sup_{0\leq \tau\leq s}e^{-2\sigma(t-\tau)}\lVert\frac{1}{1+|v|}\mathcal{Q}_{\psi(\tau)}\big(\eta_1(\tau),\eta_2(\tau)\big)\rVert_{L_{\beta}^{\infty}}\Big)\\
\leq &C \int_{s}^te^{-\nu(v)(\tau-s)}e^{2\sigma(t-\tau)}\nu(v)d\tau (1+|v|)^{-\beta}\Big(\sup_{0\leq \tau\leq s}e^{-2\sigma (t-\tau)}\lVert \psi(\tau)\rVert_{L_{\beta}^{\infty}}\lVert \eta_1(\tau)\rVert_{L_{\beta}^{\infty}} \lVert \eta_2(\tau)\rVert_{L_{\beta}^{\infty}}\Big)\\
\leq &C\int_{s}^te^{-\nu(v)(\tau-s)}e^{2\sigma(t- \tau)}\nu(v)d\tau (1+|v|)^{-\beta}\lVert \psi\rVert_{E_{\beta}^{0}}\lVert \eta_1^{\reverse}\rVert_{E_{\beta}^{-\sigma}}\lVert \eta_2^{\reverse} \rVert_{E_{\beta}^{-\sigma}},
\end{split}
\end{equation}
For the time integral, there is
\begin{equation*}
\int_{s}^te^{-\nu(v)(s-\tau)}e^{2\sigma(s- \tau)}\nu(v)d\tau=e^{2\sigma(t-s)}\int_{0}^se^{-(\nu(v)+2\sigma)(s-\tau)}\nu(v)d\tau\leq C  e^{2\sigma(t-s)}.
\end{equation*}
Consequently
\begin{equation*}
\lVert \Psi_1^-[\psi,\eta_1,\eta_2](s)\rVert_{L_{\beta}^{\infty}}\leq Ce^{2\sigma(t-s)}\lVert \psi\rVert_{E_{\beta}^{0}}\lVert \eta_1^{\reverse}\rVert_{E_{\beta}^{-\sigma}}\lVert \eta_2^{\reverse} \rVert_{E_{\beta}^{-\sigma}}.
\end{equation*}
According to the definition \eqref{eq:forward_E_norm} of the $E_{\beta}^{-\sigma}$ norm, we further have
\begin{equation*}
\lVert \Psi_1^-[\psi,\eta_1,\eta_2]\rVert_{E_{\beta}^{-\sigma}}\leq Ce^{\sigma t}\lVert \psi\rVert_{E_{\beta}^{0}}\lVert \eta_1^{\reverse}\rVert_{E_{\beta}^{-\sigma}}\lVert \eta_2^{\reverse} \rVert_{E_{\beta}^{-\sigma}}.
\end{equation*}
This concludes the proof of the first inequality in \eqref{eq:convolu_nonlinear_E_-}. 

\underline{The second inequality in \eqref{eq:convolu_nonlinear_E_-}:} \\
\myuwave{The $\Psi_2^-$ term:} In a way similar to \eqref{eq:(1,1,1)-estimate_1}, there is
\begin{equation*}
\lVert \Psi_2^{-}[\psi,\eta_1,\mathcal{G}](s)\rVert_{L_{\beta}^{\infty}}\leq C\int_{s\leq \tau\leq t}\lVert \psi(\tau)\rVert_{L_{\beta}^{\infty}}\lVert \eta_1(\tau)\rVert_{L_{\beta}^{\infty}}d\tau,
\end{equation*}
which implies
\begin{equation*}
\begin{split}
\lVert (\Psi_2^-[\psi,\eta_1,\mathcal{G}])^{\reverse} \rVert_{E_{\beta}^{-\sigma}}&\leq C\sup_{0\leq s\leq t}e^{-\sigma (t-s)}\int_{s\leq \tau\leq t}e^{\sigma(t-\tau)}\lVert \psi\rVert_{E_{\beta}^{0}}\lVert \eta_1^{\reverse}\rVert_{E_{\beta}^{-\sigma}}d\tau\leq C\lVert \psi\rVert_{E_{\beta}^{0}}\lVert \eta_1^{\reverse}\rVert_{E_{\beta}^{-\sigma}}.
\end{split}
\end{equation*}
\myuwave{The $\Psi_1^-$ term:} Similar to \eqref{eq:Psi_1_initial_exp_-}, we have
\begin{equation*}
\big|\Psi_1^-[\eta,\psi_1,\mathcal{G}](s,x,v)\big|\leq C\int_{s}^te^{-\nu(v)(\tau-s)}e^{\sigma(t- \tau)}\nu(v)d\tau (1+|v|)^{-\beta}\lVert \psi\rVert_{E_{\beta}^{0}}\lVert \eta_1^{\reverse}\rVert_{E_{\beta}^{-\sigma}}.
\end{equation*}
Consequently
\begin{equation*}
\lVert (\Psi_1^{-}[\psi,\eta_1,\mathcal{G}])^{\reverse}\rVert_{E_{\beta}^{-\sigma}}\leq C\lVert \psi\rVert_{E_{\beta}^{0}}\lVert \eta_1^{\reverse}\rVert_{E_{\beta}^{-\sigma}}.
\end{equation*}
This concludes the proof of the second inequality.            

\underline{The third inequality in \eqref{eq:convolu_nonlinear_E_-}:} It can be derived by replacing $\psi$ with $\mathcal{G}$.
\end{proof}

\section{Solving the Fixed-Point Problem}\label{sec:fixed_point}
In this section, we will prove the fixed-point map $(\rP,\lP)$ defined in \eqref{eq:fixed_point_problem} is a contraction map. Thus it has a unique fixed point in a certain function class, while a fixed point of $(\Gamma^+,\Gamma^-)$ is consequently a mild solution of the coupled Boltzmann equations \eqref{eq:transformed_evolution_perturb}.  Under the assumptions of Theorem \ref{th:main_theorem_1}, it is a contraction w.r.t. $\psi_p\in P_{\beta}^{\sigma}$ and $\eta_p^{\reverse}\in P_{\beta}^{\sigma}$ with $\beta>4,\sigma>1$; Under the assumptions of Theorem \ref{th:main_theorem_2}, it is a contraction w.r.t. $\psi_p\in E_{\beta}^{0}$ and $\eta_p^{\reverse}\in E_{\beta}^{-\sigma}$ with $\beta>4,\sigma>0$.
\subsection{Fixed Point for Theorem \ref{th:main_theorem_1}}

\begin{lemma}\label{lem:bounds-for-Phi}
For any parameters $\beta>4,\sigma>1$ and assuming \eqref{assump:main_3}, we have the following estimates of the fixed point map $\Gamma$: for the forward component $\rP$ there is
\begin{equation*}
\begin{split}
\lVert\rP[\psi_p,\eta_p]\rVert_{P_{\beta}^{\sigma}}&\leq C\Big(\lVert \psi_{p}(0)\rVert_{L_{\beta}^{\infty}}+\lVert \eta_p^{\reverse}\rVert_{P_{\beta}^{\sigma}}\lVert \psi_p\rVert_{P_{\beta}^{\sigma}}^2+\lVert \eta_p^{\reverse}\rVert_{P_{\beta}^{\sigma}}\lVert \psi_p\rVert_{P_{\beta}^{\sigma}}+\lVert \psi_p\rVert_{P_{\beta}^{\sigma}}^2\Big),
\end{split}
\end{equation*}
and for the backward component $\lP$ there is
\begin{equation*}
\begin{split}
\lVert(\lP[\psi_p,\eta_p])^{\reverse}\rVert_{P_{\beta}^{\sigma}}&\leq C\Big(\lVert \eta_{p}(t)\rVert_{L_{\beta}^{\infty}}+\lVert \eta_p^{\reverse}\rVert_{P_{\beta}^{\sigma}}^2\lVert \psi_p \rVert_{P_{\beta}^{\sigma}}+\lVert \eta_p^{\reverse}\rVert_{P_{\beta}^{\sigma}}\lVert \psi_p\rVert_{P_{\beta}^{\sigma}}+\lVert \eta_p^{\reverse}\rVert_{P_{\beta}^{\sigma}}^2\Big).
\end{split}
\end{equation*}
\end{lemma}
\begin{proof}
We will detail the proof for the forward component $\rP$. The proof for the backward component $\lP$ is almost exactly the same.

According to the decomposition of $\rP[\psi_p,\eta_p]$ in \eqref{eq:fixed_point_map_psi} and the assumption \eqref{assump:main_3} of $\phi\equiv 0$, we can use the triangle inequality for the $P_{\beta}^{\sigma}$ norm
\begin{equation}\label{eq:fixed_point_1}
\begin{split}
&\lVert\rP[\psi_p,\eta_p]\rVert_{P_{\beta}^{\sigma}}\\
\leq &\lVert e^{s\rB}\psi_{p}(0)\rVert_{P_{\beta}^{\sigma}}+\lVert \Psi^+[\eta_p,\psi_p,\psi_p]\rVert_{P_{\beta}^{\sigma}}+\lVert \Psi^+[\eta_p,\psi_p,\mathcal{G}]\rVert_{P_{\beta}^{\sigma}}+\lVert \Psi^+[\mathcal{G},\psi_p,\psi_p]\rVert_{P_{\beta}^{\sigma}}.
\end{split}
\end{equation}
For the term associated with the initial perturbation $\psi_p(0)$ in \eqref{eq:fixed_point_1}, we have
\begin{equation*}
\lVert e^{s\rB}\psi_{p}(0)\rVert_{P_{\beta}^{\sigma}}=\sup_{0\leq s\leq t}(1+s)^{\sigma}\lVert e^{s\rB}\psi_{p}(0)\rVert_{L_{\beta}^{\infty}}\leq C\sup_{0\leq s\leq t}(1+s)^{\sigma}e^{-\nu_*s}\lVert \psi_{p}(0)\rVert_{L_{\beta}^{\infty}}\leq C\lVert \psi_{p}(0)\rVert_{L_{\beta}^{\infty}},
\end{equation*}
where the first inequality is according to the assumption $\psi_p(0)\in\mathcal{K}^{\perp}$ and Lemma \ref{lem:L_beta_decay}.
To control the other three terms involving the $P_{\beta}^{\sigma}$ norm of $\Psi^+$ in \eqref{eq:fixed_point_1}, we use Lemmas \ref{lem:(1,1,1)-estimate}, \ref{lem:(1,1,0)}, and \ref{lem:(0,1,1)}
\begin{equation*}
\begin{split}
\lVert\rP[\psi_p,\eta_p]\rVert_{P_{\beta}^{\sigma}}&\leq C\lVert \psi_{p}(0)\rVert_{L_{\beta}^{\infty}}+C\lVert \eta_p^{\reverse}\rVert_{P_{\beta}^{\sigma}}\lVert \psi_p\rVert_{P_{\beta}^{\sigma}}^2+C\lVert \eta_p^{\reverse}\rVert_{P_{\beta}^{\sigma}}\lVert \psi_p\rVert_{P_{\beta}^{\sigma}}+C\lVert \psi_p\rVert_{P_{\beta}^{\sigma}}^2.
\end{split}
\end{equation*}
This concludes the estimate of $\rP[\psi_p,\eta_p]$. The estimate of $\lP[\psi_p,\eta_p]$ is the same, which concludes the proof of the lemma.
\end{proof}

To use the contraction principle to find the fixed-point, we work in the following function class $\Omega$ where $a_*$ is a positive constant
\begin{equation*}
\Omega:=\Big\{(\psi_p,\eta_p)\Big|\lVert\psi_p\rVert_{P_{\beta}^{\sigma}}\leq a_*,\ \lVert\eta_p^{\reverse}\rVert_{P_{\beta}^{\sigma}}\leq a_*,\ \psi_p(0)=(f^0-M)\mathcal{B}^{-1},\ \eta_p(t)=(e^{g(t)}-\mathcal{E})\mathcal{B}\Big\}.
\end{equation*}
It is equipped with the norm $\lVert (\psi_p,\eta_p)\rVert_{\Omega}:=\lVert \psi_p\rVert_{P_{\beta}^{\sigma}}+\lVert \eta_p^{\reverse}\rVert_{P_{\beta}^{\sigma}}$. The goal is to prove $\Gamma$ maps the region $\Omega$ into itself, and is a contraction map with respect to $\lVert\cdot\rVert_{\Omega}$. 

\begin{lemma}\label{lem:map_region}
For the constant $c$ in assumption \eqref{assump:main_1} and \eqref{assump:main_2} being small enough, there exists a constant $a_*>0$ such that the fixed-point map $\Gamma=(\rP,\lP)$ maps the region $\Omega$ to itself.
\end{lemma}

\begin{proof}
Using Lemma \ref{lem:bounds-for-Phi}, we have
\begin{equation*}
\lVert\rP[\psi_p,\eta_p]\rVert_{P_{\beta}^{\sigma}}\leq C\Big(c+a_*^3+a_*^2+a_*^2\Big)\quad \lVert(\lP[\psi_p,\eta_p])^{\reverse}\rVert_{P_{\beta}^{\sigma}}\leq C\Big(c+a_*^3+a_*^2+a_*^2\Big),
\end{equation*}
since $(\psi_p,\eta_p)\in\Omega$. Thus to make sure $\Gamma$ maps $\Omega$ into itself, we only need 
\begin{equation*}
C(c+a_*^3+a_*^2+a_*^2)\leq a_*.
\end{equation*}
It is convenient to assume $C>1$. If $c$ is small enough, for example $c<\frac{1}{10}C^{-4}$, we can choose $a_*=\frac{1}{5}C^{-3}$. This concludes the proof.
\end{proof}

\begin{theorem}\label{th:global_fixed_point_P}
\textbf{\textup{[Fixed-Point Problem] }}For arbitrary $\beta>4$ and $\sigma>1$, suppose the functions $f^0$ and $g$ satisfy the assumptions \eqref{assump:main_1}-\eqref{assump:main_3}. Then for arbitrary terminal time $0<t<+\infty$, there exists a unique fixed point $(\psi_p,\eta_p)$ of $\Gamma$, thus a mild solution of the coupled Boltzmann equations \eqref{eq:transformed_evolution_perturb}, in the function class below with constant $a_*>0$ depending on $c$
\begin{equation*}
\lVert \psi_p\rVert_{P_{\beta}^{\sigma}}<a_*,\quad \lVert \eta_p^{\reverse}\rVert_{P_{\beta}^{\sigma}}<a_*.
\end{equation*}
\end{theorem}
\begin{proof}
With Lemma \ref{lem:map_region}, we only need to verify $\Gamma$ is a contraction map on $\Omega$.

We need to prove for arbitrary $(\psi_p,\eta_p)\in\Omega$ and $(\psi_p^*,\eta_p^*)\in\Omega$, there is
\begin{equation*}
\begin{split}
&\Big\lVert \rP[\psi_p,\eta_p]- \rP[\psi_p^*,\eta_p^*]\Big\rVert_{P_{\beta}^{\sigma}}\leq \frac{1}{4}\Big(\lVert \psi_p-\psi_p^*\rVert_{P_{\beta}^{\sigma}}+\lVert (\eta_p-\eta_p^*)^{\reverse}\rVert_{P_{\beta}^{\sigma}}\Big),\\
&\Big\lVert \Big(\lP[\psi_p,\eta_p]- \lP[\psi_p^*,\eta_p^*]\Big)^{\reverse}\Big\rVert_{P_{\beta}^{\sigma}}\leq \frac{1}{4}\Big(\lVert \psi_p-\psi_p^*\rVert_{P_{\beta}^{\sigma}}+\lVert (\eta_p-\eta_p^*)^{\reverse}\rVert_{P_{\beta}^{\sigma}}\Big).
\end{split}
\end{equation*}

Since the two pair of functions share the same initial and terminal data, the difference between $\rP[\psi_p,\eta_p]$ and $\rP[\psi_p^*,\eta_p^*]$ would be
\begin{equation}\label{eq:fixed_point_difference}
\begin{split}
\rP[\psi_p,\eta_p]- \rP[\psi_p^*,\eta_p^*]=&\int_{0}^se^{(s-\tau)\rB}\Big(\mathcal{Q}_{\eta_p}\big(\psi_p,\psi_p\big)-\mathcal{Q}_{\eta_p^*}\big(\psi_p^*,\psi_p^*\big)\Big)d\tau\\
+&\int_{0}^se^{(s-\tau)\rB}\Big(\mathcal{Q}_{\eta_p}\big(\psi_p,\mathcal{G}\big)-\mathcal{Q}_{\eta_p^*}\big(\psi_p^*,\mathcal{G}\big)\Big)d\tau\\
+&\int_{0}^se^{(s-\tau)\rB}\Big(\mathcal{Q}_{\mathcal{G}}\big(\psi_p,\psi_p\big)-\mathcal{Q}_{\mathcal{G}}\big(\psi_p^*,\psi_p^*\big)\Big)d\tau.
\end{split}
\end{equation}
Here we have ignored the variable $\tau$ for simplicity in notation. Now we consider the three terms in \eqref{eq:fixed_point_difference} separately. For the first line in the RHS of \eqref{eq:fixed_point_difference}
\begin{equation*}
\begin{split}
\mathcal{Q}_{\eta_p}(\psi_p,\psi_p)-\mathcal{Q}_{\eta_p^*}(\psi_p^*,\psi_p^*)&=\mathcal{Q}_{\eta_p-\eta_p^*}(\psi_p,\psi_p)+\Big(\mathcal{Q}_{\eta_p^*}(\psi_p,\psi_p)-\mathcal{Q}_{\eta_p^*}(\psi_p^*,\psi_p^*)\Big)\\
&=\mathcal{Q}_{\eta_p-\eta_p^*}(\psi_p,\psi_p)+\mathcal{Q}_{\eta_p^*}(\psi_p+\psi_p^*,\psi_p-\psi_p^*).
\end{split}
\end{equation*}
For the second line in the RHS of \eqref{eq:fixed_point_difference}
\begin{equation*}
\mathcal{Q}_{\mathcal{G}}(\psi_p,\psi_p)-\mathcal{Q}_{\mathcal{G}}(\psi_p^*,\psi_p^*)=\mathcal{Q}_{\mathcal{G}}(\psi_p+\psi_p^*,\psi_p-\psi_p^*),
\end{equation*}
For the third line in the RHS of \eqref{eq:fixed_point_difference}
\begin{equation*}
\mathcal{Q}_{\eta_p}(\psi_p,\mathcal{G})-\mathcal{Q}_{\eta_p^*}(\psi_p^*,\mathcal{G})=\mathcal{Q}_{\eta_p-\eta_p^*}(\psi_p,\mathcal{G})+\mathcal{Q}_{\eta_p^*}(\psi_p-\psi_p^*,\mathcal{G}).
\end{equation*}
According to the equations above as well as the definition \ref{def:convo_nonlinear} of $\Psi^+$, we have
\begin{equation}\label{eq:fixed_point_difference_2}
\begin{split}
\rP[\psi_p,\eta_p]-\rP[\psi_p^*,\eta_p^*]&=\Psi^+[\eta_p-\eta_p^*,\psi_p,\psi_p]+\Psi^+[\eta_p,\psi_p+\psi_p^*,\psi_p-\psi_p^*]\\
&+\Psi^+[\mathcal{G},\psi_p+\psi_p^*,\psi_p-\psi_p^*]\\
&+\Psi^+[\eta_p-\eta_p^*,\psi_p,\mathcal{G}]+\Psi^+[\eta_p^*,\psi_p-\psi_p^*,\mathcal{G}].
\end{split}
\end{equation}
To consider the $P_{\beta}^{\sigma}$ norm of this difference, we first use the triangle inequality for the $P_{\beta}^{\sigma}$ norm, and then use Lemmas \ref{lem:(1,1,1)-estimate}, \ref{lem:(1,1,0)}, and \ref{lem:(0,1,1)}. These would imply
\begin{equation}\label{eq:fixed_point_difference_3}
\begin{split}
&\lVert\rP[\psi_p,\eta_p]-\rP[\psi_p^*,\eta_p^*]\rVert_{P_{\beta}^{\sigma}}\\
\leq& C\Big(\lVert (\eta_p-\eta_p^*)^{\reverse} \rVert_{P_{\beta}^{\sigma}}\ \lVert \psi_p\rVert_{P_{\beta}^{\sigma}}^2+\lVert (\eta_p^*)^{\reverse}\rVert_{P_{\beta}^{\sigma}}\ \lVert \psi_p+\psi_p^*\rVert_{P_{\beta}^{\sigma}}\ \lVert \psi_p-\psi_p^*\rVert_{P_{\beta}^{\sigma}}\\
+&\lVert \psi_p+\psi_p^*\rVert_{P_{\beta}^{\sigma}}\ \lVert \psi_p-\psi_p^*\rVert_{P_{\beta}^{\sigma}}\\
+&\lVert (\eta_p-\eta_p^*)^{\reverse}\rVert_{P_{\beta}^{\sigma}}\ \lVert \psi_p\rVert_{P_{\beta}^{\sigma}}
+\lVert (\eta_p^*)^{\reverse}\rVert_{P_{\beta}^{\sigma}}\ \lVert \psi_p-\psi_p^*\rVert_{P_{\beta}^{\sigma}}\Big).
\end{split}
\end{equation}
Notice that each term in the RHS of \eqref{eq:fixed_point_difference_3} consists of either $\lVert \psi_p-\psi_p^*\rVert_{P_{\beta}^{\sigma}}$ or $\lVert (\eta_p-\eta_p^*)^{\reverse}\rVert_{P_{\beta}^{\sigma}}$. Since the constant $c$ in assumptions \eqref{assump:main_1}-\eqref{assump:main_2} is small enough, the constant $a_*$ constructed in Lemma \ref{lem:map_region} could also be small enough. This means the various terms $\lVert \psi_p\rVert_{P_{\beta}^{\sigma}},\lVert \eta_p^{\reverse}\rVert_{P_{\beta}^{\sigma}},\lVert \psi_p^*\rVert_{P_{\beta}^{\sigma}},\lVert \eta_p^{*,\reverse}\rVert_{P_{\beta}^{\sigma}}$ are also small enough. Consequently
\begin{equation*}
\lVert\rP[\psi_p,\eta_p]-\rP[\psi_p^*,\eta_p^*]\rVert_{P_{\beta}^{\sigma}}\leq \frac{1}{4}\Big(\lVert \psi_p-\psi_p^*\rVert_{P_{\beta}^{\sigma}}+\lVert (\eta_p-\eta_p^*)^{\reverse}\rVert_{P_{\beta}^{\sigma}}\Big).
\end{equation*}
Through exactly the same argument as above, we can show for $c$ small enough and thus $a_*$ small enough, there is
\begin{equation*}
\lVert(\lP[\psi_p,\eta_p]-\lP[\psi_p^*,\eta_p^*])^{\reverse}\rVert_{P_{\beta}^{\sigma}}\leq \frac{1}{4}\Big(\lVert \psi_p-\psi_p^*\rVert_{P_{\beta}^{\sigma}}+\lVert (\eta_p-\eta_p^*)^{\reverse}\rVert_{P_{\beta}^{\sigma}}\Big).
\end{equation*}
This shows $\Gamma$ is a contraction map from $\Omega$ to $\Omega$, which concludes the proof.
\end{proof}

\subsection{Fixed Point for Theorem \ref{th:main_theorem_2}}
\begin{lemma}\label{lem:bounds-for-Phi_sigma0}
For any parameters $\beta>4,\sigma>0$ and assuming \eqref{assump:main_3_th2}, we have the following estimates of the fixed point map $\Gamma$. If $\psi_p\in E_{\beta}^0$ and $\eta_p^{\reverse}\in E_{\beta}^{-\sigma}$, then for $\rP$ there is
\begin{equation*}
\begin{split}
&\lVert\rP[\psi_p,\eta_p]\rVert_{E_{\beta}^{0}}\\
\leq& C\Big(\lVert \psi_{p}(0)\rVert_{L_{\beta}^{\infty}}+\lVert \psi_p\rVert_{E_{\beta}^{0}}^2+\eps_{\phi}\lVert \psi_p\rVert_{E_{\beta}^0}\Big)+Ce^{\sigma t}\Big(\lVert \eta_p^{\reverse}\rVert_{E_{\beta}^{-\sigma}}\lVert \psi_p\rVert_{E_{\beta}^{0}}^2+\lVert \eta_p^{\reverse}\rVert_{E_{\beta}^{-\sigma}}\lVert \psi_p\rVert_{E_{\beta}^{0}}\Big),
\end{split}
\end{equation*}
and for $\lP$ there is
\begin{equation*}
\begin{split}
&\lVert(\lP[\psi_p,\eta_p])^{\reverse}\rVert_{E_{\beta}^{-\sigma}}\\
\leq& C\Big(\lVert \eta_{p}(t)\rVert_{L_{\beta}^{\infty}}+\eps_{\phi}\lVert\eta_p^{\reverse}\rVert_{E_{\beta}^{-\sigma}}+\lVert \eta_p^{\reverse}\rVert_{E_{\beta}^{-\sigma}}\lVert \psi_p\rVert_{E_{\beta}^{0}}\Big)+Ce^{\sigma t}\Big(\lVert \eta_p^{\reverse}\rVert_{E_{\beta}^{-\sigma}}^2\lVert \psi_p \rVert_{E_{\beta}^{0}}+\lVert \eta_p^{\reverse}\rVert_{E_{\beta}^{-\sigma}}^2\Big).
\end{split}
\end{equation*}
Here $\eps_{\phi}$ is the small positive constant introduced in assumption \eqref{assump:main_3_th2}.
\end{lemma}
\begin{proof}
The proof of this lemma is very similar to the proof of Lemma \ref{lem:bounds-for-Phi}.

\underline{The Forward Component $\rP$:} First according to the decomposition \eqref{eq:fixed_point_map_psi} of $\rP[\psi_p,\eta_p]$ and the triangle inequality for $E_{\beta}^0$ norm, we get
\begin{equation*}
\begin{split}
&\lVert\rP[\psi_p,\eta_p]\rVert_{E_{\beta}^{0}}\\
\leq&\lVert e^{s\rB}\psi_{p}(0)\rVert_{E_{\beta}^{0}}+\lVert \Psi^+[\eta_p,\psi_p,\psi_p]\rVert_{E_{\beta}^{0}}+\lVert \Psi^+[\eta_p,\psi_p,\mathcal{G}]\rVert_{E_{\beta}^{0}}+\lVert \Psi^+[\mathcal{G},\psi_p,\psi_p]\rVert_{E_{\beta}^{0}}\\
+&\Big\lVert \int_0^se^{(s-\tau)B^+}\psi_p(\tau)\phi(\tau)d\tau\Big\rVert_{E_{\beta}^0}.
\end{split}
\end{equation*}
Here we have the extra term involving $\phi$ since the assumption does not require the function to be constant $0$.

To control the term involving $\phi$, we write
\begin{equation}\label{eq:control_phi_+}
\begin{split}
\Big\lVert \int_0^se^{(s-\tau)B^+}\psi_p(\tau)\phi(\tau)d\tau\Big\rVert_{E_{\beta}^0}\leq \sup_{0\leq s\leq t}\Big(C\int_0^s\lVert \psi_p(\tau)\phi(\tau)\rVert_{L_{\beta}^\infty}d\tau\Big)&\leq C\int_0^t\lVert \psi_p(\tau)\phi(\tau)\rVert_{L_{\beta}^\infty}d\tau\\
&\leq C\lVert \psi_p\rVert_{E_{\beta}^0}\int_0^t\lVert \phi(\tau)\rVert_{L^{\infty}}d\tau\\
&\leq C\eps_{\phi}\lVert \psi_p\rVert_{E_{\beta}^0}.
\end{split}
\end{equation}
Then we use Lemma \ref{lem:convolu_nonlinear_E_+} to control the $E_{\beta}^0$ norms of various $\Psi^+$ terms. This yields the desired estimate of $\rP$.

\underline{The backward component $\lP$:} the proof for the backward component $\lP$ is quite similar. First we control the term in $\lP$ involving $\phi$
\begin{equation}\label{eq:control_phi_-}
\begin{split}
\Big\lVert \Big(\int_s^te^{(\tau-s)B^+}\eta_p(\tau)\phi(\tau)d\tau\Big)^{\reverse}\Big\rVert_{E_{\beta}^{-\sigma}}&\leq \sup_{0\leq s\leq t}e^{-(t-s)\sigma}C\int_s^t \lVert \eta_p(\tau)\rVert_{L_{\beta}^{\infty}}\lVert\phi \rVert_{L^{\infty}}d\tau\\
&\leq e^{-\sigma t}\eps_{\phi} \sup_{0\leq s\leq t}e^{s\sigma}C\int_s^t e^{(t-\tau)\sigma}\lVert \eta_p^{\reverse}\rVert_{E_{\beta}^{-\sigma}}d\tau\\
&\leq Ce^{-\sigma t}\eps_{\phi} \sup_{0\leq s\leq t}e^{s\sigma}e^{(t-s)\sigma} \lVert \eta_p^{\reverse}\rVert_{E_{\beta}^{-\sigma}}\leq C\eps_{\phi}\lVert \eta_p^{\reverse}\rVert_{E_{\beta}^{-\sigma}}.
\end{split}
\end{equation}
Next by Lemma \ref{lem:convolu_nonlinear_E_-}, we can control the $E_{\beta}^{-\sigma}$ norms of various $\Psi^-$ terms in $\lP$. This concludes the estimate of $\lP$ as well as the proof of the lemma.
\end{proof}

To use the contraction principle to find the fixed-point, we will work in the function class $\Omega$ where $a_*$ is a positive constant
\begin{equation*}
\Omega:=\Big\{(\psi_p,\eta_p)\Big|\lVert\psi_p\rVert_{E_{\beta}^0}\leq a_*,\ e^{\sigma t}\lVert\eta_p^{\reverse}\rVert_{E_{\beta}^{-\sigma}}\leq a_*,\ \psi_p(0)=(f^0-M)\mathcal{B}^{-1},\ \eta_p(t)=(e^{g(t)}-\mathcal{E})\mathcal{B}\Big\}.
\end{equation*}
It is equipped with the norm
\begin{equation}
\lVert (\psi_p,\eta_p)\rVert_{\Omega}:=\lVert \psi_p\rVert_{E_{\beta}^{0}}+e^{\sigma t}\lVert \eta_p^{\reverse}\rVert_{E_{\beta}^{-\sigma}}
\end{equation}
In terms of the estimate of norms, the two components $\psi_p$ and $\eta_t$ are no longer symmetric w.r.t. the time-reversal. A factor $e^{\sigma t}$ is needed before $\lVert \eta_p^{\reverse}\rVert_{E_{\beta}^{-\sigma}}$ in the definition of $\lVert\cdot\rVert_{\Omega}$.
\begin{lemma}\label{lem:map_region_sigma0}
For the constants $c$ and $\eps_{\phi}$ in assumption \eqref{assump:main_1_th2}-\eqref{assump:main_3_th2} being small enough, there exists a constant $a_*>0$ such that the fixed-point map $\Gamma=(\rP,\lP)$ maps the region $\Omega$ to itself.
\end{lemma}
\begin{proof}
According to Lemma \ref{lem:bounds-for-Phi_sigma0}, there is
\begin{equation*}
\lVert\rP[\psi_p,\eta_p]\rVert_{E_{\beta}^{0}}\leq C(c+a_*^2+\eps_{\phi}a_*)+Ce^{\sigma t}(e^{-\sigma t}a_*^3+e^{-\sigma t}a_*^2)=C\big(c+a_*^2+a_*^3+(\eps_{\phi}+1)a_*\big),
\end{equation*}
which is the same case as in the proof of Lemma \ref{lem:map_region} if $\eps_{\phi}$ is small enough. Thus the forward component $\rP$ can be estimated in a similar way. 

For the backward component $\lP$, using Lemma \ref{lem:bounds-for-Phi_sigma0} we have
\begin{equation}
\begin{split}
\lVert (\lP[\psi_p,\eta_p])^{\reverse}\rVert_{E_{\beta}^{-\sigma}}&\leq C\big(e^{-\sigma t}c+\eps_{\phi}e^{-\sigma t}a_*+e^{-\sigma t}a_*^2\big)+Ce^{\sigma t}\big((e^{-\sigma t}a_*)^2a_*+(e^{-\sigma t}a_*)^2\big)\\
&\leq Ce^{-\sigma t}\big(c+a_*^2+a_*^3+(\eps_{\phi}+1)a_*\big),
\end{split}
\end{equation}
or equivalently
\begin{equation*}
e^{\sigma t}\lVert (\lP[\psi_p,\eta_p])^{\reverse}\rVert_{E_{\beta}^{-\sigma}}\leq C\big(c+a_*^2+a_*^3+(\eps_{\phi}+1)a_*\big).
\end{equation*}
Consequently the estimate of $e^{\sigma t}\lVert (\lP[\psi_p,\eta_p])^{\reverse}\rVert_{E_{\beta}^{-\sigma}}$ is the same as that of $\lVert\rP[\psi_p,\eta_p]\rVert_{E_{\beta}^{0}}$. This concludes the proof.
\end{proof}
\begin{theorem}\label{th:global_fixed_point_E}
\textbf{\textup{[Fixed-Point Problem] }}For any parameters $\beta>4,\sigma>0$, suppose the functions $f^0$ and $g$ satisfy the assumptions \eqref{assump:main_1_th2}-\eqref{assump:main_3_th2}. Then for arbitrary terminal time $0<t<+\infty$, there exists a unique fixed point $(\psi_p,\eta_p)$ of $\Gamma$, thus a mild solution of the coupled Boltzmann equations \eqref{eq:transformed_evolution_perturb}, in the function class below with constant $a_*>0$ depending on $c$
\begin{equation*}
\lVert \psi_p\rVert_{E_{\beta}^{0}}<a_*,\quad e^{\sigma t}\lVert \eta_p^{\reverse}\rVert_{E_{\beta}^{-\sigma}}<a_*.
\end{equation*}
\end{theorem}
\begin{proof}
With Lemma \ref{lem:map_region_sigma0}, we only need to verify $\Gamma$ is a contraction map on $\Omega$.

The method is to prove for arbitrary $(\psi_p,\eta_p)\in\Omega$ and $(\psi_p^*,\eta_p^*)\in\Omega$ that the following holds
\begin{equation}\label{eq:contraction_exponential}
\begin{split}
&\Big\lVert \rP[\psi_p,\eta_p]- \rP[\psi_p^*,\eta_p^*]\Big\rVert_{E_{\beta}^{0}}\leq \frac{1}{4}\Big(\lVert \psi_p-\psi_p^*\rVert_{E_{\beta}^{0}}+e^{\sigma t}\lVert (\eta_p-\eta_p^*)^{\reverse}\rVert_{E_{\beta}^{-\sigma}}\Big),\\
&e^{\sigma t}\Big\lVert \Big(\lP[\psi_p,\eta_p]- \lP[\psi_p^*,\eta_p^*]\Big)^{\reverse}\Big\rVert_{E_{\beta}^{-\sigma}}\leq \frac{1}{4}\Big(\lVert \psi_p-\psi_p^*\rVert_{E_{\beta}^{0}}+e^{\sigma t}\lVert (\eta_p-\eta_p^*)^{\reverse}\rVert_{E_{\beta}^{-\sigma}}\Big).
\end{split}
\end{equation}

\underline{The Forward Component $\rP$:} Thus the difference between $\rP[\psi_p,\eta_p]$ and $\rP[\psi_p^*,\eta_p^*]$ is 
\begin{equation}\label{eq:map_difference_sigma0}
\begin{split}
\rP[\psi_p,\eta_p]-\rP[\psi_p^*,\eta_p^*]&=\Psi^+[\eta_p-\eta_p^*,\psi_p,\psi_p]+\Psi^+[\eta_p,\psi_p+\psi_p^*,\psi_p-\psi_p^*]\\
&+\Psi^+[\mathcal{G},\psi_p+\psi_p^*,\psi_p-\psi_p^*]\\
&+\Psi^+[\eta_p-\eta_p^*,\psi_p,\mathcal{G}]+\Psi^+[\eta_p^*,\psi_p-\psi_p^*,\mathcal{G}]\\
&+\int_0^se^{(s-\tau)B^+}(\psi_p-\psi_p^*)\phi d\tau.
\end{split}
\end{equation}
Modifying equation \eqref{eq:control_phi_+}, we obtain
\begin{equation*}
\Big\lVert \int_0^s e^{(s-\tau)B^+}(\psi_p-\psi_p^*)\phi d\tau\Big\rVert_{E_{\beta}^0}\leq C\eps_{\phi}\lVert \psi_p-\psi_p^*\rVert_{E_{\beta}^{0}}.
\end{equation*}
For the first three lines on the RHS of \eqref{eq:map_difference_sigma0}, their $E_{\beta}^0$ norms can be controlled using Lemma \ref{lem:convolu_nonlinear_E_+}. This implies
\begin{equation*}
\begin{split}
\lVert\rP[\psi_p,\eta_p]&-\rP[\psi_p^*,\eta_p^*]\rVert_{E_{\beta}^0}\\
\leq& C\Big(e^{\sigma t}\lVert (\eta_p-\eta_p^*)^{\reverse} \rVert_{E_{\beta}^{-\sigma}}\ \lVert \psi_p\rVert_{E_\beta^0}^2+e^{\sigma t}\lVert \eta_p^{*,\reverse}\rVert_{E_{\beta}^{-\sigma}}\ \lVert \psi_p+\psi_p^*\rVert_{E_\beta^0}\ \lVert (\psi_p-\psi_p^*)^{\reverse}\rVert_{E_\beta^0}\\
+&\lVert \psi_p+\psi_p^*\rVert_{E_\beta^0}\ \lVert \psi_p-\psi_p^*\rVert_{E_\beta^0}\\
+&e^{\sigma t}\lVert (\eta_p-\eta_p^*)^{\reverse}\rVert_{E_{\beta}^{-\sigma}}\ \lVert \psi_p\rVert_{E_\beta^0}
+e^{\sigma t}\lVert \eta_p^{*,\reverse}\rVert_{E_{\beta}^{-\sigma}}\ \lVert \psi_p-\psi_p^*\rVert_{E_\beta^0}+\eps_{\phi}\lVert\psi_p-\psi_p^*\rVert_{E_{\beta}^0}\Big),
\end{split}
\end{equation*}
which verifies the first equation in \eqref{eq:contraction_exponential} if we take $a_*>0$ small enough, similar to what is done in the proof of Theorem \ref{th:global_fixed_point_P}.

\underline{The Backward Component $\lP$:} Similar to the estimate of $\rP$, the difference between $\lP[\psi_p,\eta_p]$ and $\lP[\psi_p^*,\eta_p^*]$ is 
\begin{equation}\label{eq:map_difference_sigma0_-}
\begin{split}
\lP[\psi_p,\eta_p]-\lP[\psi_p^*,\eta_p^*]&=\Psi^-[\psi_p-\psi_p^*,\eta_p,\eta_p]+\Psi^-[\psi_p,\eta_p+\eta_p^*,\eta_p-\eta_p^*]\\
&+\Psi^-[\mathcal{G},\eta_p+\eta_p^*,\eta_p-\eta_p^*]\\
&+\Psi^-[\psi_p-\psi_p^*,\eta_p,\mathcal{G}]+\Psi^-[\psi_p^*,\eta_p-\eta_p^*,\mathcal{G}]\\
&+\int_s^te^{(\tau-s)B^-}(\eta_p-\eta_p^*)\phi d\tau.
\end{split}
\end{equation}
For the $E_{\beta}^{-\sigma}$ norm for the fourth line of \eqref{eq:map_difference_sigma0_-}, modifying equation \eqref{eq:control_phi_+} we obtain
\begin{equation*}
\Big\lVert \int_s^t e^{(\tau-s)B^-}(\eta_p-\eta_p^*)\phi d\tau\Big\rVert_{E_{\beta}^{-\sigma}}\leq C\eps_{\phi}\lVert \eta_p-\eta_p^*\rVert_{E_{\beta}^{-\sigma}}.
\end{equation*}
For the first three lines on the RHS of \eqref{eq:map_difference_sigma0_-}, their $E_{\beta}^{-\sigma}$ norms can be controlled using Lemma \ref{lem:convolu_nonlinear_E_-}. This implies
\begin{equation*}
\begin{split}
&\lVert(\lP[\psi_p,\eta_p]-\lP[\psi_p^*,\eta_p^*])^{\reverse}\rVert_{E_{\beta}^{-\sigma}}\\
\leq& C\Big(e^{\sigma t}\lVert \psi_p-\psi_p^* \rVert_{E_{\beta}^{0}}\ \lVert \eta_p^{\reverse}\rVert_{E_\beta^{-\sigma}}^2+e^{\sigma t}\lVert \psi_p^*\rVert_{E_{\beta}^{0}}\ \lVert (\eta_p+\eta_p^*)^{\reverse}\rVert_{E_\beta^{-\sigma}}\ \lVert (\eta_p-\eta_p^*)^{\reverse}\rVert_{E_\beta^{-\sigma}}\\
+&e^{\sigma t}\lVert (\eta_p+\eta_p^*)^{\reverse}\rVert_{E_\beta^{-\sigma}}\ \lVert (\eta_p-\eta_p^*)^{\reverse}\rVert_{E_\beta^{-\sigma}}\\
+&\lVert \psi_p-\psi_p^*\rVert_{E_{\beta}^{0}}\ \lVert \eta_p^{\reverse}\rVert_{E_\beta^{-\sigma}}
+\lVert \psi_p^*\rVert_{E_{\beta}^{0}}\ \lVert (\eta_p-\eta_p^*)^{\reverse}\rVert_{E_\beta^{-\sigma}}+\eps_{\phi}\lVert(\eta_p-\eta_p^*)^{\reverse}\rVert_{E_{\beta}^{-\sigma}}\Big).
\end{split}
\end{equation*}
This inequality can be reorganized by multiplying an $e^{\sigma t}$ over the two sides
\begin{equation*}
\begin{split}
&e^{\sigma t}\lVert(\lP[\psi_p,\eta_p]-\lP[\psi_p^*,\eta_p^*])^{\reverse}\rVert_{E_{\beta}^{-\sigma}}\\
\leq& C\Big(\lVert \psi_p-\psi_p^* \rVert_{E_{\beta}^{0}}\ (e^{\sigma t}\lVert \eta_p^{\reverse}\rVert_{E_\beta^{-\sigma}})^2+\lVert \psi_p^*\rVert_{E_{\beta}^{0}}\ e^{\sigma t}\lVert (\eta_p+\eta_p^*)^{\reverse}\rVert_{E_\beta^{-\sigma}}\ e^{\sigma t}\lVert (\eta_p-\eta_p^*)^{\reverse}\rVert_{E_\beta^{-\sigma}}\\
+&e^{\sigma t}\lVert (\eta_p+\eta_p^*)^{\reverse}\rVert_{E_\beta^{-\sigma}}\ e^{\sigma t}\lVert (\eta_p-\eta_p^*)^{\reverse}\rVert_{E_\beta^{-\sigma}}\\
+&\lVert \psi_p-\psi_p^*\rVert_{E_{\beta}^{0}}\ e^{\sigma t}\lVert \eta_p^{\reverse}\rVert_{E_\beta^{-\sigma}}
+\lVert \psi_p^*\rVert_{E_{\beta}^{0}}\ e^{\sigma t}\lVert (\eta_p-\eta_p^*)^{\reverse}\rVert_{E_\beta^{-\sigma}}+\eps_{\phi}e^{\sigma t}\lVert(\eta_p-\eta_p^*)^{\reverse}\rVert_{E_{\beta}^{-\sigma}}\Big),
\end{split}
\end{equation*}
which verifies the second inequality in \eqref{eq:contraction_exponential} if we take $a_*$ small enough. As a consequence $\Gamma$ is a contraction from $\Omega$ to $\Omega$, having a unique fixed point. This concludes the proof.
\end{proof}

\section{Justification of the Mild Functional Solution}\label{sec:characteristics}
In this section, we show in Theorem \ref{th:justification_mild} that we can construct mild solutions of the Hamilton-Jacobi equation, using the mild solution of the coupled Boltzmann equations. The notion of a mild solution for the Hamilton-Jacobi equation \eqref{eq:functional_HJ} has also been defined in \eqref{eq:mild_solution_HJ}. Theorem \ref{th:justification_mild} has been proved in \cite{BGSS_2023} under some analyticity assumptions, in the framework of the Cauchy-Kovalevskaya Theorem. In this paper we do not have these analyticity conditions, and our proof is a modification of the proof in \cite{BGSS_2023}, with some additional analysis.

To construct a mild solution of the functional Hamilton-Jacobi equation \eqref{eq:functional_HJ}, we will consider the Hamiltonian system characterized by the associated Euler-Lagrange equation. Specifically given a terminal time $t> 0$, we consider the following Hamiltonian system defined on $[0,t]$
\begin{equation}\label{eq:Hamiltonian_physical_1}
\begin{split}
&D_s\varphi_t=\frac{\p \mathcal{H}}{\p p}(\varphi_t,p_t),\quad \textup{with $\varphi_t(0)=f^0e^{p_t(0)}$}\\
&D_s(p_t-g)=-\frac{\p \mathcal{H}}{\p \varphi}(\varphi_t,p_t),\quad \textup{with $p_t(t)=g(t)$}.
\end{split}
\end{equation}
Here the Hamiltonian $\mathcal{H}$ has been defined in \eqref{eq:functional_hamiltonian} as
\begin{equation*}
\mathcal{H}(\varphi,p)=\frac{1}{2}\int\varphi(x,v)\varphi(x,v_*)\Big(e^{\Delta p (x,v,v_*)}-1\Big)\big((v_*-v)\cdot\omega\big)_+d\omega dv_* dvdx.
\end{equation*}
Given a mild solution of equation \eqref{eq:Hamiltonian_physical_1} on $[0,t]$, we define the functional $\widehat{\mathcal{I}}(t,g)$ as
\begin{equation}\label{eq:characteristic_functional}
    \widehat{\mathcal{I}}(t,g):=-1+\langle f^0,e^{p_t(0)} \rangle+\dl D_s\big(p_t(s)-g(s)\big),\varphi_t(s)\dr+\int_0^t\mathcal{H}\Big(\varphi_t(s),p_t(s)\Big)ds.
\end{equation}
Here the notation $\langle \cdot,\cdot\rangle$ refers to the inner product in $L^2(\mathds{T}_x^d\times \mathds{R}_v^d)$, while $\dl\cdot,\cdot\dr$ refers to the inner product in $L^2([0,t]\times \mathds{T}_x^d\times\mathds{R}_v^d)$, with given terminal time $t$.

However as the readers have seen in this paper, we do not directly deal with the Euler-Lagrange system \eqref{eq:Hamiltonian_physical_1}. Instead, we have performed the change of variables in \eqref{eq:change_of_variables_sec3} to make the system more symmetric
\begin{equation}\label{eq:change_of_variables}
    (\psi_t,\eta_t):=(\varphi_te^{-p_t+\alpha'|v|^2},e^{p_t-\alpha'|v|^2}).
\end{equation}
This change of variables provides a new Hamiltonian $\mathcal{H}'$ from the original $\mathcal{H}$
\begin{equation}\label{eq:new_hamiltonian}
\mathcal{H}'(\psi,\eta):=-\frac{1}{4}\int \Big(\psi(v')\psi(v_*')-\psi(v)\psi(v_*)\Big)\Big(\eta(v')\eta(v_*')-\eta(v)\eta(v_*)\Big)\big((v_*-v)\cdot\omega\big)_+d\omega dv_*dvdx.
\end{equation}
The Hamiltonian $\mathcal{H}'(\psi,\eta)$ can also be written equivalently as
\begin{equation*}
\begin{split}
\mathcal{H}'(\psi,\eta)&=\frac{1}{2}\int \psi(v)\psi(v_*)\Big(\eta(v')\eta(v_*')-\eta(v)\eta(v_*)\Big)\big((v_*-v)\cdot\omega\big)_+d\omega dv_*dvdx\\
&=\frac{1}{2}\int \Big(\psi(v')\psi(v_*')-\psi(v)\psi(v_*)\Big)\eta(v)\eta(v_*)\big((v_*-v)\cdot\omega\big)_+d\omega dv_*dvdx.
\end{split}
\end{equation*}
Replacing $(\varphi_t,\eta_t)$ by $(\psi_t,\eta_t)$ by the change of variables \eqref{eq:change_of_variables}, we have the evolution equation for $(\psi_t,\eta_t)$ during the time interval $[0,t]$
\begin{equation*}
\begin{split}
D_s \psi_t&=-\psi_t\phi+\mathcal{Q}_{\eta_t}(\psi_t,\psi_t),\quad \psi_t(0)=f^0\mathcal{B}^{-1}\\
D_s \eta_t&=\eta_t\phi-\mathcal{Q}_{\psi_t}(\eta_t,\eta_t),\quad \eta_t(t)=e^{g(t)}\mathcal{B}.
\end{split}
\end{equation*}
This equation for $(\psi_t,\eta_t)$ can also be written as
\begin{equation}\label{eq:coupled_Boltzmann_gradient}
\begin{split}
D_s \psi_t(s)&=-\psi_t(s)\phi(s)+\frac{\p \mathcal{H}'\big(\psi_t(s),\eta_t(s)\big)}{\p \eta},\quad \psi_t(0)=f^0\mathcal{B}^{-1},\\
D_s \eta_t(s)&=\eta_t(s)\phi(s)-\frac{\p \mathcal{H}'\big(\psi_t(s),\eta_t(s)\big)}{\p \psi},\quad \eta_t(t)=e^{g(t)}\mathcal{B}.
\end{split}
\end{equation}
In this case, the functional $\widehat{\mathcal{I}}(t,g)$ constructed in equation \eqref{eq:characteristic_functional} is equivalent to 
\begin{equation}\label{eq:characteristic_functional_2}
\widehat{\mathcal{I}}(t,g):=-1+\langle f^0\mathcal{B}^{-1},\eta_t(0) \rangle+\dl D_s\eta_t(s),\psi_t(s)\dr-\dl\phi(s),\psi_t(s)\eta_t(s)\dr+\int_0^t\mathcal{H}'\Big(\psi_t(s),\eta_t(s)\Big)ds.
\end{equation}
Theorem \ref{th:justification_mild} shows that if the $(\psi_t,\eta_t)$ in the definition of $\widehat{\mathcal{I}}(t,g)$ is the mild solution given in Theorem \ref{th:global_fixed_point_P} or \ref{th:global_fixed_point_E}, then the functional $\widehat{\mathcal{I}}(t,g)$ is a mild solution of the functional Hamilton-Jacobi equation.

Some ingredients are needed to prove Theorem \ref{th:justification_mild}, for example some continuity estimates. In Lemma \ref{lem:order_of_variation_2}, we will prove that the solution $(\psi_t(s),\eta_t(s))$ of the coupled Boltzmann equations, is continuous in $0\leq s\leq t$ under the $L_{\beta}^{\infty}$ norm. This enables us to give a precise definition of $D_s\psi_t(s)$ and $D_s\eta_t(s)$. Using Lemma \ref{lem:order_of_variation_2}, we are going to define $D_s\psi_t(s)$ and $D_s\eta_t(s)$ as the following limits in $L_{x,v}^2$
\begin{equation}\label{eq:definition_D_s}
\begin{split}
&D_s\psi_t(s)=\lim_{\tau\rightarrow 0}\Big(\frac{\psi_t(s+\tau)-\psi_t(s)}{\tau}-\frac{S_{\tau}\psi_t(s)-\psi_t(s)}{\tau}\Big),\\
&D_s\eta_t(s)=\lim_{\tau\rightarrow 0}\Big(\frac{\eta_t(s+\tau)-\eta_t(s)}{\tau}-\frac{S_{\tau}\eta_t(s)-\eta_t(s)}{\tau}\Big).
\end{split}
\end{equation}
Using the mild formulation of $(\psi_t(s),\eta_t(s))$, for each $\eps$ the difference is equal to a time integral. For example according to the mild formulation of the coupled Boltzmann equations
\begin{equation*}
\psi_t(s+\tau)=S_{\tau}\psi_t(s)-\int_{s}^{s+\tau}S_{u-s}\psi_t(u)\phi(u)du+\int_{s}^{s+\tau}S_{u-s}\frac{\p \mathcal{H}'\big(\psi_t(u),\eta_t(u)\big)}{\p \eta}du.
\end{equation*}
This implies
\begin{equation*}
\begin{split}
\lim_{\tau\rightarrow 0}\Big(\frac{\psi_t(s+\tau)-\psi_t(s)}{\tau}&-\frac{S_{\tau}\psi_t(s)-\psi_t(s)}{\tau}\Big)\\
=&\lim_{\tau\rightarrow 0}\Big(-\tau^{-1}\int_{s}^{s+\tau}S_{u-s}\psi_t(u)\phi(u)du+\tau^{-1}\int_{s}^{s+\tau}S_{u-s}\frac{\p \mathcal{H}'\big(\psi_t(u),\eta_t(u)\big)}{\p \eta}du\Big).
\end{split}
\end{equation*}
Thus by the continuity of $(\psi_t(s),\eta_t(s))$ under $L_{\beta}^{\infty}$ proved in Lemma \ref{lem:order_of_variation_2}, the integral is continuous with respect to $\tau$ and the limits in \eqref{eq:definition_D_s} exist
\begin{equation}\label{eq:D_s_as_limit}
D_s \psi_t(s)=-\psi_t(s)\phi(s)+\frac{\p \mathcal{H}'\big(\psi_t(s),\eta_t(s)\big)}{\p \eta},\ D_s \eta_t(s)=\eta_t(s)\phi(s)-\frac{\p \mathcal{H}'\big(\psi_t(s),\eta_t(s)\big)}{\p \psi}.
\end{equation}
Now the definition of $\widehat{\mathcal{I}}(t,g)$ is rigorous. 

\begin{theorem}\label{th:justification_mild} 
If we take $(\psi_t(s),\eta_t(s))$ to be the mild solution of the coupled Boltzmann equations given in Theorem \ref{th:global_fixed_point_P} or \ref{th:global_fixed_point_E}, then the functional $\widehat{\mathcal{I}}(t,g)$ defined in equation \eqref{eq:characteristic_functional_2} is a mild solution of the Hamilton-Jacobi equation \eqref{eq:functional_HJ}.
\end{theorem}
\begin{proof}
Take an arbitrary time $t$. We want to consider the difference $\widehat{\mathcal{I}}(t+\tau,g)-\widehat{\mathcal{I}}(t,g)$ with $\tau$ being a small positive number. We use $\delta_{\tau}$ to refer to the variation with respect to the terminal time
\begin{equation}
\delta_{\tau}\psi_t(s)=\psi_{t+\tau}(s)-\psi_{t}(s),\quad \delta_{\tau}\eta_t(s)=\eta_{t+\tau}(s)-\eta_{t}(s).
\end{equation}

\underline{Time differential of $\widehat{\mathcal{I}}$:}
according to the definition \eqref{eq:characteristic_functional_2} of $\widehat{\mathcal{I}}$, there is
\begin{equation}\label{eq:functional_difference_2}
\begin{split}
&\widehat{\mathcal{I}}(t+\tau)-\widehat{\mathcal{I}}(t,g)\\
=&\langle f^0,\delta_{\tau}\eta_t(0)\rangle+\dl D_s\eta_t,\delta_{\tau}\psi_t\dr+\dl D_s\delta_{\tau}\eta_t,\psi_t\dr+\underbrace{\dl D_s\delta_{\tau}\eta_t,\delta_{\tau}\psi_t\dr}_{\textup{\blue{Remainder}}}+\int_{t}^{t+\tau}\langle D_s\eta_{t+\tau}(s),\psi_{t+\tau}(s)\rangle ds\\
&-\dl\phi,\delta_{\tau}\psi_t\eta_t\dr-\dl\phi,\psi_t\delta_{\tau}\eta_t\dr-\underbrace{\dl\phi,\delta_{\tau}\psi_t\delta_{\tau}\eta_t\dr}_{\textup{\blue{Remainder}}}-\int_{t}^{t+\tau}\langle \phi(s),\psi_{t+\tau}(s)\eta_{t+\tau}(s)\rangle ds\\
&+\dl \delta_{\tau}\psi_t,\frac{\p \mathcal{H}'(\psi_t,\eta_t)}{\p \psi}\dr+\dl \delta_{\tau}\eta_t,\frac{\p \mathcal{H}'(\psi_t,\eta_t)}{\p \eta}\dr+\int_t^{t+\tau}\mathcal{H}'\Big(\psi_{t+\tau}(s),\eta_{t+\tau}(s)\Big)ds\\
&+\underbrace{\Big[\int_0^{t}\mathcal{H}'\Big(\psi_{t+\tau}(s),\eta_{t+\tau}(s)\Big)-\int_0^{t}\mathcal{H}'\Big(\psi_{t}(s),\eta_{t}(s)\Big)ds-\dl \delta_{\tau}\psi_t,\frac{\p \mathcal{H}'(\psi_t,\eta_t)}{\p \psi}\dr-\dl \delta_{\tau}\eta_t,\frac{\p \mathcal{H}'(\psi_t,\eta_t)}{\p \eta}\dr\Big]}_{\textup{\blue{Remainder}}}
\end{split}
\end{equation}
In the equation above, the first line on the RHS is the variation of $\langle D_s\eta_t,\psi_t \rangle$ with respect to $t$, and the second line is the variation of $\langle \phi,\psi_t\eta_t\rangle$ with respect to $t$. The third and the fourth lines are the variation of $\int_0^t \mathcal{H}'\big(\psi_t(s),\eta_t(s)\big)ds$.

It will be proved in Lemma \ref{lem:order_of_variation_2} that for arbitrary $0\leq s\leq t$, the variation is uniformly of order $\tau$
\begin{equation}\label{eq:order_of_variation}
\lVert \delta_{\tau}\psi_t(s)\rVert_{L_{\beta}^{\infty}}=O(\tau),\ \lVert \delta_{\tau}\eta_t(s)\rVert_{L_{\beta}^{\infty}}=O(\tau),
\end{equation}
since each component of $(\psi_t(s),\eta_t(s))$ is continuous in $0\leq s\leq t$ under the $L_{\beta}^{\infty}$ norm. Combine \eqref{eq:functional_difference_2} with \eqref{eq:order_of_variation}, we can show those higher-order remainders in \eqref{eq:functional_difference_2} are of order $o(\tau)$.

Performing integration by parts according to Lemma \ref{lem:integration_by_parts}, we have
\begin{equation}\label{eq:integration_by_parts}
\dl D_s\delta_{\tau}\eta_t,\psi_t\dr=-\dl \delta_{\tau}\eta_t,D_s\psi_t\dr+\langle \delta_{\tau}\eta_t(t),\psi_t(t)\rangle-\langle \delta_{\tau}\eta_t(0),\psi_t(0)\rangle.
\end{equation}
Combine \eqref{eq:functional_difference_2} with \eqref{eq:integration_by_parts} and Lemma \ref{lem:ipp_cancellation}, the difference $\widehat{\mathcal{I}}(t+\tau,g)-\widehat{\mathcal{I}}(t,g)$ becomes
\begin{equation}
\begin{split}
&\widehat{\mathcal{I}}(t+\tau,g)-\widehat{\mathcal{I}}(t,g)\\
=&\dl\delta_{\tau}\eta_t,-D_s \psi_t(s)-\psi_t(s)\phi(s)+\frac{\p \mathcal{H}'\big(\psi_t(s),\eta_t(s)\big)}{\p \eta}\dr\\
&+\dl \delta_{\tau}\psi_t,D_s \eta_t(s)-\eta_t(s)\phi(s)+\frac{\p \mathcal{H}'\big(\psi_t(s),\eta_t(s)\big)}{\p \psi}\dr+\int_t^{t+\tau}\mathcal{H}'\Big(\psi_{t+\tau}(s),\eta_{t+\tau}(s)\Big)ds+o(\tau).
\end{split}
\end{equation}
Using the fact that $(\psi_t,\eta_t)$ is also the mild solution of \eqref{eq:coupled_Boltzmann_gradient} and equation \eqref{eq:D_s_as_limit}, there is
\begin{equation*}
\widehat{\mathcal{I}}(t+\tau,g)-\widehat{\mathcal{I}}(t,g)=  \int_t^{t+\tau}\mathcal{H}'\Big(\psi_{t+\tau}(s),\eta_{t+\tau}(s)\Big)ds+o(\tau).
\end{equation*}
Since each component of $(\psi_t(s),\eta_t(s))$ is continuous in $0\leq s\leq t$ under the $L_{\beta}^{\infty}$ norm, the Hamiltonian $\mathcal{H}'\big(\psi_t(s),\eta_t(s)\big)$ is also continuous in $0\leq s\leq t$. Consequently the functional $\widehat{\mathcal{I}}(t,g)$ is differentiable in time
\begin{equation}\label{eq:time_differential_I}
\p_t\widehat{\mathcal{I}}(t,g)=\mathcal{H}'\big(\psi_t(t),\eta_t(t)\big)=\mathcal{H}\big(\varphi_t(t),p_t(t)\big).
\end{equation}
\underline{Differential of $\widehat{\mathcal{I}}$ with respect to $g$:} 
Now we want to fix $t$ and differentiate $\widehat{\mathcal{I}}(t,g)$ against $g(t)$. The rigorous proof is essentially the same as the proof of the differentiability in $t$. It is because if we have a variation of $g$, the terminal data $\eta_t(t)=e^{g(t)}\mathcal{B}$ will be changed accordingly. This is the same case as the change of terminal data when we are considering the differentiability in $t$. Here we give the formal proof for simplicity. Using $\delta$ to denote the variation, we differentiate $\widehat{\mathcal{I}}(t,g)$ against $g(t)$
\begin{equation}\label{eq:g_functional_derivative}
\begin{split}
\p\widehat{\mathcal{I}}(t,g)=&\langle f^0\mathcal{B}^{-1},\delta\eta_t(0)\rangle+\dl D_s\delta\eta_t,\psi_t\dr+\dl D_s\eta_t,\delta\psi_t\dr-\dl\phi,\delta\psi_t\eta_t\dr-\dl\phi,\psi_t\delta\eta_t\dr\\
&+\dl \delta\psi_t,\frac{\p \mathcal{H}'(\psi_t,\eta_t)}{\p \psi}\dr+\dl \delta\eta_t,\frac{\p \mathcal{H}'(\psi_t,\eta_t)}{\p \eta}\dr.
\end{split}
\end{equation}
Again we perform an integration by parts
\begin{equation}\label{eq:integration_by_parts_2}
\dl D_s\delta\eta_t,\psi_t\dr=-\dl \delta\eta_t,D_s\psi_t\dr+\langle \delta\eta_t(t),\psi_t(t)\rangle-\langle \delta\eta_t(0),\psi_t(0)\rangle.
\end{equation}
Using the same technique as in the study of $\p_t\widehat{\mathcal{I}}$, with $(\psi_t,\eta_t)$ being the mild solution of \eqref{eq:coupled_Boltzmann_gradient}, we get
\begin{equation}
\begin{split}
\p\widehat{\mathcal{I}}(t,g)=\langle \delta e^{g(t)}\mathcal{B},\psi_t(t)\rangle-\dl \delta \phi,\psi_t\eta_t\dr&=\langle \delta g(t) ,\eta_t(t)\psi_t(t)\rangle-\dl \delta \phi,\psi_t\eta_t\dr\\
&=\langle \delta g(t) ,\varphi_t(t)\rangle-\dl \delta \phi,\varphi_t(t)\dr.
\end{split}
\end{equation}
This shows that $\frac{\p \widehat{\mathcal{I}}(t,g)}{\p g(t)}=\varphi_t(t)$. Along with $p_t(t)=g(t)$, \eqref{eq:time_differential_I} is equivalently
\begin{equation}
\p_t\widehat{\mathcal{I}}(t,g)=\mathcal{H}\Big(\frac{\p \widehat{\mathcal{I}}(t,g)}{\p g(t)},g(t)\Big).
\end{equation}
This concludes the proof.
\end{proof}

\begin{lemma}\label{lem:order_of_variation_2}
\textbf{\textup{[Continuity in $0\leq s\leq t$]}} Consider the mild solution $(\psi_t(s),\eta_t(s))$ derived in Theorem \ref{th:global_fixed_point_P} or Theorem \ref{th:global_fixed_point_E} of the coupled Boltzmann equations \eqref{eq:transformed_evolution}. Each component is continuous w.r.t. $s$ and $t$ in the region $0\leq s\leq t$ under the $L_{\beta}^{\infty}$ norm.
\end{lemma}
\begin{proof}
We only detail the proof for $(\psi_t(s),\eta_t(s))$ in Theorem \ref{th:global_fixed_point_P}. The proof for $(\psi_t(s),\eta_t(s))$ in Theorem \ref{th:global_fixed_point_E} is essentially the same. In the proof, we will use $C(g)$ as a positive constant depending on $g$, $C(f^0)$ as a positive constant depending on $f^0$, and $C(f^0,g)$ as a positive constant dependent on $f^0$ and $g$.

\underline{Continuity in $s$:} suppose we want to consider the difference between $(\psi_{t,p}(s),\eta_{t,p}(s))$ and $(\psi_{t,p}(s+\tau),\eta_{t,p}(s+\tau))$ for $\tau>0$. We use the notation 
\begin{equation*}
\big(\psi_{t,p}^{\tau}(s),\eta_{t,p}^{\tau}(s)\big):=\big(\psi_{t,p}(s+\tau),\eta_{t,p}(s+\tau)\big),\quad 0\leq s\leq t-\tau.
\end{equation*}
The fixed-point $(\psi_{t,p},\eta_{t,p})$ is derived as the limit of $(\psi_{t,p}^{(n)},\eta_{t,p}^{(n)})_{n\geq 0}$ under the iteration $n\rightarrow +\infty$
\begin{equation}\label{eq:psi_t_n}
\begin{split}
&\psi_{t,p}^{(0)}:=e^{s\rB}(f^0\mathcal{B}^{-1}-\mathcal{G}),\ \eta_{t,p}^{(0)}:=e^{(t-s)\lB}(e^{g(t)}\mathcal{B}-\mathcal{G}),\\
&\psi_{t,p}^{(n+1)}=\rP_t[\psi_{t,p}^{(n)},\eta_{t,p}^{(n)}],\quad\ \eta_{t,p}^{(n+1)}=\lP_t[\psi_{t,p}^{(n)},\eta_{t,p}^{(n)}],\quad n\geq 0.
\end{split}
\end{equation}
It also satisfies the following iteration relation
\begin{equation*}
\begin{split}
&\psi_{t,p}^{(n+1)}(s)=e^{sB^+}(f^0\mathcal{B}^{-1}-\mathcal{G})+\int_0^s e^{(s-u)B^+}\mathcal{N}[\psi_{t,p}^{(n)},\eta_{t,p}^{(n)}]du,\\
&\eta_{t,p}^{(n+1)}(s)=e^{(t-\tau-s)B^-}e^{\tau B^-}(e^{g(t)}\mathcal{B}-\mathcal{G})+\int_s^{t-\tau} e^{(u-s)B^-}\mathcal{N}[\eta_{t,p}^{(n)},\psi_{t,p}^{(n)}]du+\underbrace{\int_{t-\tau}^{t}e^{(u-s)B^-}\mathcal{N}[\eta_{t,p}^{(n)},\psi_{t,p}^{(n)}]du}_{\textup{\blue{$O(\tau)$ Error}}}.
\end{split}
\end{equation*}
The fixed-point $(\psi_{t,p}^{\tau},\eta_{t,p}^{\tau})$ can be derived as the limit of $(\psi_{t,p}^{\tau,(n)},\eta_{t,p}^{\tau,(n)})_{n\geq 0}$ 
\begin{equation*}
\big(\psi_{t,p}^{\tau,(n)}(s),\eta_{t,p}^{\tau,(n)}(s)\big):=\big(\psi_{t,p}^{(n)}(s+\tau),\eta_{t,p}^{(n)}(s+\tau)\big),\quad 0\leq s\leq t-\tau.
\end{equation*}
It satisfies the iteration relation
\begin{equation*}
\begin{split}
&\psi_{t,p}^{\tau,(n+1)}(s)=e^{sB^+}e^{\tau B^+}(f^0\mathcal{B}^{-1}-\mathcal{G})+\int_0^s e^{(s-u)B^+}\mathcal{N}[\psi_{t,p}^{\tau,(n)},\eta_{t,p}^{\tau,(n)}]du+\underbrace{\int_0^\tau e^{(s-u)B^+}\mathcal{N}[\psi_{t,p}^{(n)},\eta_{t,p}^{(n)}]du}_{\textup{\blue{$O(\tau)$ Error}}},\\
&\eta_{t,p}^{\tau,(n+1)}(s)=e^{(t-\tau-s)B^-}(e^{g(t)}\mathcal{B}-\mathcal{G})+\int_s^{t-\tau} e^{(u-s)B^-}\mathcal{N}[\eta_{t-\tau,p}^{*,(n)},\psi_{t-\tau,p}^{*,(n)}]du.
\end{split}
\end{equation*}
For now we use $P_{\beta}^{\sigma}$ as the norm on the time interval $[0,t-\tau]$. For the forward component, there is 
\begin{equation}\label{eq:intermediate_contraction_+}
\begin{split}
&\lVert\psi_{t,p}^{\tau,(n+1)}-\psi_{t,p}^{(n+1)}\rVert_{P_{\beta}^{\sigma}}\\
\leq &\Big\lVert\int_{0}^se^{(s-u)\rB}\mathcal{N}[\psi_{t,p}^{\tau,(n)},\eta_{t,p}^{\tau,(n)}]du-\int_{0}^se^{(s-u)\rB}\mathcal{N}[\psi_{t,p}^{(n)},\eta_{t,p}^{(n)}]du\Big\rVert_{P_{\beta}^{\sigma}}+C(f^0)(1+t)^{\sigma}\tau\\
\leq &\frac{1}{4}\Big(\lVert\psi_{t,p}^{(n)}-\psi_{t,p}^{\tau,(n)}\rVert_{P_{\beta}^{\sigma}}+\lVert(\eta_{t,p}^{(n)}-\eta_{t,p}^{\tau,(n)})^{\reverse}\rVert_{P_{\beta}^{\sigma}}\Big)+C(f^0)(1+t)^{\sigma}\tau.
\end{split}
\end{equation}
Here the term $C(f^0)(1+t)^{\sigma}\tau$ is due to assumption \eqref{assump:main_1} and Lemma \ref{lem:beta continuity}
\begin{equation}\label{eq:7.22}
\begin{split}
\lVert e^{sB^+}e^{\tau B^+}(f^0\mathcal{B}^{-1}-\mathcal{G})-e^{sB^+}(f^0\mathcal{B}^{-1}-\mathcal{G})\rVert_{L_{\beta}^{\infty}}&\leq C\lVert e^{\tau B^+}(f^0\mathcal{B}^{-1}-\mathcal{G})-(f^0\mathcal{B}^{-1}-\mathcal{G})\rVert_{L_{\beta}^{\infty}}\\
&\leq C(f^0)(1+t)^{\sigma}\tau.
\end{split}
\end{equation}
Similarly for the backward component, we get
\begin{equation}\label{eq:intermediate_contraction_-}
\Big\lVert(\eta_{t,p}^{\tau,(n+1)}-\eta_{t,p}^{(n+1)})^{\reverse}\Big\rVert_{P_{\beta}^{\sigma}}\leq \frac{1}{4}\Big(\lVert\psi_{t,p}^{(n)}-\psi_{t,p}^{\tau,(n)}\rVert_{P_{\beta}^{\sigma}}+\lVert(\eta_{t,p}^{(n)}-\eta_{t,p}^{\tau,(n)})^{\reverse}\rVert_{P_{\beta}^{\sigma}}\Big)+C(g)(1+t)^{\sigma}\tau.
\end{equation}
Here the term $C(g)(1+t)^{\sigma}\tau$ is due to assumption \eqref{assump:main_3} and Lemma \ref{lem:beta continuity}, in a way similar to \eqref{eq:7.22}.

Equations \eqref{eq:intermediate_contraction_+} and \eqref{eq:intermediate_contraction_-} together imply the contraction relation with a $C(f^0,g)(1+t)^{\sigma}\tau$ error
\begin{equation*}
\begin{split}
&\lVert\psi_{t,p}^{(n+1)}-\psi_{t,p}^{\tau,(n+1)}\rVert_{P_{\beta}^{\sigma}}+\lVert\big(\eta_{t,p}^{(n+1)}-\eta_{t,p}^{\tau,(n+1)}\big)^{\reverse}\rVert_{P_{\beta}^{\sigma}}\\
\leq& \frac{1}{2}\Big(\lVert\psi_{t,p}^{(n)}-\psi_{t,p}^{\tau,(n)}\rVert_{P_{\beta}^{\sigma}}+\lVert(\eta_{t,p}^{(n)}-\eta_{t,p}^{\tau,(n)})^{\reverse}\rVert_{P_{\beta}^{\sigma}}\Big)+C(f^0,g)(1+t)^{\sigma}\tau.
\end{split}
\end{equation*}
Taking $n\rightarrow +\infty$, we have shown $(\psi_t(s),\eta_t(s))$ is right continuous in $s$
\begin{equation*}
\lVert\psi_{t,p}-\psi_{t,p}^{\tau}\rVert_{P_{\beta}^{\sigma}}+\lVert\big(\eta_{t,p}-\eta_{t,p}^{\tau}\big)^{\reverse}\rVert_{P_{\beta}^{\sigma}}\leq C(f^0,g)(1+t)^{\sigma}\tau\rightarrow 0,\quad \tau\rightarrow 0.
\end{equation*}
The same analysis also holds true if we take $\tau<0$, which proves $(\psi_t(s),\eta_t(s))$ is left continuous in $s$. This concludes the proof of the continuity in $s$.

\underline{Continuity in $t$:} to prove the lemma, we first notice
\begin{equation*}
\delta_{\tau}\psi_t(s)=\delta_{\tau}\psi_{t,p}(s)=\psi_{t+\tau,p}(s)-\psi_{t,p}(s),\quad  \delta_{\tau}\eta_t(s)=\delta_{\tau}\eta_{t,p}(s)=\eta_{t+\tau,p}(s)-\eta_{t,p}(s).
\end{equation*}
Here the $(\psi_{t,p},\eta_{t,p})$ is the fixed point of the map \eqref{eq:fixed_point_problem} with terminal time $t$. Recalling Theorem \ref{th:global_fixed_point_P}, we assume $\phi\equiv 0$
\begin{equation}\label{eq:Gamma_t}
\left\{
\begin{split}
&\rP_t[\psi_p,\eta_p](s):=e^{s\rB}(f^0\mathcal{B}^{-1}-\mathcal{G})+\int_{0}^se^{(s-u)\rB}\mathcal{N}[\psi_p,\eta_p]du,\\
&\lP_t[\psi_p,\eta_p](s):=e^{(t-s)\lB}(e^{g(t)}\mathcal{B}-\mathcal{G})+\int_{s}^t e^{(u-s)\lB}\mathcal{N}[\eta_p,\psi_p]du,
\end{split}
\right.
\end{equation}
while $(\psi_{t+\tau,p},\eta_{t+\tau,p})$ is the fixed point of the map \eqref{eq:fixed_point_problem} with terminal time $t+\tau$
\begin{equation*}
\left\{
\begin{split}
&\rP_{t+\tau}[\psi_p,\eta_p](s):=e^{s\rB}(f^0\mathcal{B}^{-1}-\mathcal{G})+\int_{0}^se^{(s-u)\rB}\mathcal{N}[\psi_p,\eta_p]du,\\
&\lP_{t+\tau}[\psi_p,\eta_p](s):=e^{(t+\tau-s)\lB}(e^{g(t+\tau)}\mathcal{B}-\mathcal{G})+\int_{s}^{t+\tau} e^{(u-s)\lB}\mathcal{N}[\eta_p,\psi_p]du.
\end{split}
\right.
\end{equation*}
The fixed-point $(\psi_{t,p},\eta_{t,p})$ is derived as the limit of $(\psi_{t,p}^{(n)},\eta_{t,p}^{(n)})_{n\geq 0}$ under the iteration $n\rightarrow +\infty$
\begin{equation}
\begin{split}
&\psi_{t,p}^{(0)}:=e^{s\rB}(f^0\mathcal{B}^{-1}-\mathcal{G}),\ \eta_{t,p}^{(0)}:=e^{(t-s)\lB}(e^{g(t)}\mathcal{B}-\mathcal{G}),\\
&\psi_{t,p}^{(n+1)}=\rP_t[\psi_{t,p}^{(n)},\eta_{t,p}^{(n)}],\quad\ \eta_{t,p}^{(n+1)}=\lP_t[\psi_{t,p}^{(n)},\eta_{t,p}^{(n)}],\quad n\geq 0.
\end{split}
\end{equation}
The fixed-point $(\psi_{t+\tau,p},\eta_{t+\tau,p})$ is derived as the limit of $(\psi_{t+\tau,p}^{(n)},\eta_{t+\tau,p}^{(n)})$ under the iteration $n\rightarrow +\infty$
\begin{equation}\label{eq:iteration_relation_t+tau}
\begin{split}
&\psi_{t+\tau,p}^{(0)}:=e^{s\rB}(f^0\mathcal{B}^{-1}-\mathcal{G}),\ \eta_{t+\tau,p}^{(0)}:=e^{(t+\tau-s)\lB}(e^{g(t+\tau)}\mathcal{B}-\mathcal{G}),\\
&\psi_{t+\tau,p}^{(n+1)}=\rP_{t+\tau}[\psi_{{t+\tau},p}^{(n)},\eta_{t+\tau,p}^{(n)}],\quad\ \eta_{t+\tau,p}^{(n+1)}=\lP_{t+\tau}[\psi_{t+\tau,p}^{(n)},\eta_{t+\tau,p}^{(n)}],\quad n\geq 0.
\end{split}
\end{equation}
Regarding the involved functions and the norm $P_{\beta}^{\sigma}$ as defined on the time interval $[0,t]$, we want to prove
\begin{equation}\label{eq:terminal_contraction}
\lVert\psi_{t,p}^{(n+1)}-\psi_{t+\tau,p}^{(n+1)}\rVert_{P_{\beta}^{\sigma}}+\lVert\big(\eta_{t,p}^{(n+1)}-\eta_{t+\tau,p}^{(n+1)}\big)^{\reverse}\rVert_{P_{\beta}^{\sigma}}\leq O(\tau)+\frac{1}{2}\Big(\lVert\psi_{t,p}^{(n)}-\psi_{t+\tau,p}^{(n)}\rVert_{P_{\beta}^{\sigma}}+\lVert(\eta_{t,p}^{(n)}-\eta_{t+\tau,p}^{(n)})^{\reverse}\rVert_{P_{\beta}^{\sigma}}\Big).
\end{equation}
Once \eqref{eq:terminal_contraction} is proved, we will have
\begin{equation}
\lVert\psi_{t,p}-\psi_{t+\tau,p}\rVert_{P_{\beta}^{\sigma}}+\lVert\big(\eta_{t,p}-\eta_{t+\tau,p}\big)^{\reverse}\rVert_{P_{\beta}^{\sigma}}\leq O(\tau).
\end{equation}
Now we prove \eqref{eq:terminal_contraction}. For the forward component, according to \eqref{eq:iteration_relation_t+tau} there is
\begin{equation}
\begin{split}
\lVert\psi_{t+\tau,p}^{(n+1)}-\psi_{t,p}^{(n+1)}\rVert_{P_{\beta}^{\sigma}}&=\Big\lVert\int_{0}^se^{(s-u)\rB}\mathcal{N}[\psi_{t+\tau,p}^{(n)},\eta_{t+\tau,p}^{(n)}]du-\int_{0}^se^{(s-u)\rB}\mathcal{N}[\psi_{t,p}^{(n)},\eta_{t,p}^{(n)}]du\Big\rVert_{P_{\beta}^{\sigma}}\\
&\leq \frac{1}{2}\Big(\lVert\psi_{t,p}^{(n)}-\psi_{t+\tau,p}^{(n)}\rVert_{P_{\beta}^{\sigma}}+\lVert(\eta_{t,p}^{(n)}-\eta_{t+\tau,p}^{(n)})^{\reverse}\rVert_{P_{\beta}^{\sigma}}\Big)
\end{split}
\end{equation}
For the backward component, we have
\begin{equation}
\begin{split}
\eta_{t+\tau,p}^{(n+1)}(s)-\eta_{t,p}^{(n+1)}(s)=&e^{(t+\tau-s)\lB}(e^{g(t+\tau)}\mathcal{B}-\mathcal{G})-e^{(t-s)\lB}(e^{g(t)}\mathcal{B}-\mathcal{G})\\
&+\int_{s}^{t+\tau}e^{(u-s)\lB}\mathcal{N}[\eta_{t+\tau,p}^{(n)},\psi_{t+\tau,p}^{(n)}]du-\int_{s}^te^{(u-s)\lB}\mathcal{N}[\eta_{t,p}^{(n)},\psi_{t,p}^{(n)}]du.
\end{split}
\end{equation}
To control the difference between these different terminal data, we decompose it as
\begin{equation*}
\begin{split}
&\Big\lVert e^{(t+\tau-s)\lB}(e^{g(t+\tau)}\mathcal{B}-\mathcal{G})-e^{(t-s)\lB}(e^{g(t)}\mathcal{B}-\mathcal{G})\Big\rVert_{L_{\beta}^{\infty}}\\
\leq &\Big\lVert\Big(e^{(t+\tau-s)\lB}(e^{g(t+\tau)}\mathcal{B}-\mathcal{G})-e^{(t+\tau-s)\lB}(e^{g(t)}\mathcal{B}-\mathcal{G})\Big)\Big\rVert_{L_{\beta}^{\infty}}\quad \textup{\blue{(I)}}\\
&+\Big\lVert\Big(e^{(t+\tau-s)\lB}(e^{g(t)}\mathcal{B}-\mathcal{G})-e^{(t-s)\lB}(e^{g(t)}\mathcal{B}-\mathcal{G})\Big)\Big\rVert_{L_{\beta}^{\infty}}\quad \ \ \textup{\blue{(II)}}
\end{split}
\end{equation*}
The term (I) is controlled using the boundedness of $e^{(t+\tau-s)B^-}$ in $L_{\beta}^{\infty}$ (see Lemma \ref{lem:L_beta_decay}) and assumption \eqref{assump:main_3}
\begin{equation*}
\textup{(I)}\leq C\lVert(e^{g(t+\tau)}\mathcal{B}-\mathcal{G})-(e^{g(t)}\mathcal{B}-\mathcal{G})\rVert_{L_{\beta}^{\infty}}\leq C(g)\tau.
\end{equation*} 

According to Lemma \ref{lem:beta continuity} and the boundedness of $e^{(t-s)B^-}$ in $L_{\beta}^{\infty}$, the term (II) is controlled as
\begin{equation*}
\begin{split}
\textup{(II)}&=\Big\lVert e^{(t-s)\lB}\Big(e^{\tau B^-}(e^{g(t)}\mathcal{B}-\mathcal{G})-(e^{g(t)}\mathcal{B}-\mathcal{G})\Big)\Big\rVert_{L_{\beta}^{\infty}}\\
&\leq C\lVert e^{\tau B^-}(e^{g(t)}\mathcal{B}-\mathcal{G})-(e^{g(t)}\mathcal{B}-\mathcal{G})\rVert_{L_{\beta}^{\infty}}\leq C(g)\tau.
\end{split}
\end{equation*}
These imply
\begin{equation}\label{eq:backward_terminal_continuity}
\begin{split}
\Big\lVert(\eta_{t+\tau,p}^{(n+1)}-\eta_{t,p}^{(n+1)})^{\reverse}\Big\rVert_{P_{\beta}^{\sigma}}&\leq \Big\lVert\Big(\lP_t[\psi_{t+\tau,p}^{(n)},\eta_{t+\tau,p}^{(n)}]-\lP_t[\psi_{t,p}^{(n)},\eta_{t,p}^{(n)}]\Big)^{\reverse}\Big\rVert_{P_{\beta}^{\sigma}}+C(g)(1+t)^{\sigma}\tau,\\
&\leq C(g)(1+t)^{\sigma}\tau+\frac{1}{2}\Big(\lVert\psi_{t,p}^{(n)}-\psi_{t+\tau,p}^{(n)}\rVert_{P_{\beta}^{\sigma}}+\lVert(\eta_{t,p}^{(n)}-\eta_{t+\tau,p}^{(n)})^{\reverse}\rVert_{P_{\beta}^{\sigma}}\Big).
\end{split}
\end{equation}
Now equation \eqref{eq:terminal_contraction} is proved, which concludes the proof of the continuity in $t$.
\end{proof}

We have performed two integrations by parts in the formal proof, which are \eqref{eq:integration_by_parts} and \eqref{eq:integration_by_parts_2}. Here we will justify the first one, while the second one can be justified in the same way. 
\begin{lemma}\label{lem:integration_by_parts}
\textup{\textbf{[Integration by Parts]}} Consider the solution $(\psi_t,\eta_t)$ derived in Theorem \ref{th:global_fixed_point_P} or Theorem \ref{th:global_fixed_point_E} of the coupled Boltzmann equations \eqref{eq:transformed_evolution}. For arbitrary $t>0$ and $\tau>0$, we have
\begin{equation}
\dl D_s\delta_{\tau}\eta_t,\psi_t\dr=-\dl \delta_{\tau}\eta_t,D_s\psi_t\dr+\langle \delta_{\tau}\eta_t(t),\psi_t(t)\rangle-\langle \delta_{\tau}\eta_t(0),\psi_t(0)\rangle.
\end{equation}
\end{lemma}
\begin{proof}
Since the operator $D_s$ is rigorously defined as the limit in \eqref{eq:definition_D_s} whose convergence is uniform in $s$, we have
\begin{equation}\label{eq:D_s_justification_1}
\begin{split}
&\int_0^t\langle D_s\delta_{\tau}\eta_t(s),\psi_t(s)\rangle ds\\
=&\lim_{\eps\rightarrow 0}\Big(\int_0^t \Big\langle \frac{\delta_{\tau}\eta_t(s+\eps)-\delta_{\tau}\eta_t(s)}{\eps},\psi_t(s)\Big\rangle ds-\int_0^t \Big\langle \frac{S_{\eps}\delta_{\tau}\eta_t(s)-\delta_{\tau}\eta_t(s)}{\eps},\psi_t(s)\Big\rangle ds\Big).
\end{split}
\end{equation}
For each $\eps>0$, the first term in the RHS of \eqref{eq:D_s_justification_1} equals to
\begin{equation*}
\int_0^t \Big\langle\delta_{\tau}\eta_t(s), \frac{\psi_t(s-\eps)-\psi_t(s)}{\eps}\Big\rangle ds+\int_t^{t+\eps} \Big\langle\delta_{\tau}\eta_t(s), \frac{\psi_t(s-\eps)}{\eps}\Big\rangle ds-\int_0^{\eps} \Big\langle\delta_{\tau}\eta_t(s), \frac{\psi_t(s-\eps)}{\eps}\Big\rangle ds.
\end{equation*}
The second term in the RHS of \eqref{eq:D_s_justification_1} equals to
\begin{equation*}
\int_0^t \Big\langle \frac{S_{\eps}\delta_{\tau}\eta_t(s)-\delta_{\tau}\eta_t(s)}{\eps},\psi_t(s)\Big\rangle ds=\int_0^t \Big\langle \delta_{\tau}\eta_t(s),\frac{S_{-\eps}\psi_t(s)-\psi_t(s)}{\eps}\Big\rangle ds.
\end{equation*}
These equations together imply
\begin{equation*}
\begin{split}
&\int_0^t\langle D_s\delta_{\tau}\eta_t(s),\psi_t(s)\rangle ds\\
=&\lim_{\eps\rightarrow 0}\Big[\int_0^t \Big\langle\delta_{\tau}\eta_t(s), \frac{\psi_t(s-\eps)-\psi_t(s)}{\eps}-\frac{S_{-\eps}\psi_t(s)-\psi_t(s)}{\eps}\Big\rangle ds\\
&\qquad\qquad\qquad +\int_t^{t+\eps} \Big\langle\delta_{\tau}\eta_t(s), \frac{\psi_t(s-\eps)}{\eps}\Big\rangle ds-\int_0^{\eps} \Big\langle\delta_{\tau}\eta_t(s), \frac{\psi_t(s-\eps)}{\eps}\Big\rangle ds\Big],
\end{split}
\end{equation*}
Each component of $(\psi_t(s),\eta_t(s))$ is continuous in $0\leq s\leq t$ under the $L_{\beta}^{\infty}$ norm. As a consequence the limit exists and implies
\begin{equation*}
\dl D_s\delta_{\tau}\eta_t(s),\psi_t(s)\dr =-\dl \delta_{\tau}\eta_t(s),D_s\psi_t(s)\dr +\langle \delta p_t(t),\varphi_t(t) \rangle-\langle \delta p_t(0),\psi_t(0) \rangle.
\end{equation*}
This concludes the proof of the lemma.
\end{proof}
To conclude this section, we give the postponed proof of a technical lemma used in the proof of Theorem \ref{th:justification_mild}.
\begin{lemma}\label{lem:ipp_cancellation}
Consider the mild solution $(\psi_t(s),\eta_t(s))$ derived in Theorem \ref{th:global_fixed_point_P} or Theorem \ref{th:global_fixed_point_E} of the coupled Boltzmann equations \eqref{eq:transformed_evolution}, the following estimate holds
\begin{equation}
\langle \delta_{\tau}\eta_t(t),\psi_t(t)\rangle+\int_{t}^{t+\tau}\langle D_s\eta_{t+\tau}(s),\psi_{t+\tau}(s)\rangle ds-\int_{t}^{t+\tau}\langle \phi(s),\psi_{t+\tau}(s)\eta_{t+\tau}(s)\rangle ds=o(\tau)
\end{equation}
\end{lemma}
\begin{proof}
By the equivalent expression \eqref{eq:D_s_as_limit} of $D_s\eta$, we have
\begin{equation*}
\int_{t}^{t+\tau}\langle D_s\eta_{t+\tau}(s),\psi_{t+\tau}(s)\rangle ds-\int_{t}^{t+\tau}\langle \phi(s),\psi_{t+\tau}(s)\eta_{t+\tau}(s)\rangle ds   =-\int_{t}^{t+\tau}\langle \mathcal{Q}_{\psi_{t+\tau}}(\eta_{t+\tau},\eta_{t+\tau}),\psi_{t+\tau} \rangle ds.
\end{equation*}
Using the continuity (Lemma \ref{lem:order_of_variation_2}) of $(\psi_t(s),\eta_t(s))$ under the $L_{\beta}^{\infty}$ norm, we further have
\begin{equation}\label{eq:tau_diff_1}
\lim_{\tau\rightarrow 0}\tau^{-1}\Big(-\int_{t}^{t+\tau}\langle \mathcal{Q}_{\psi_{t+\tau}}(\eta_{t+\tau},\eta_{t+\tau}),\psi_{t+\tau} \rangle ds\Big)=-\langle \mathcal{Q}_{\psi_{t}(t)}(\eta_{t}(t),\eta_{t}(t)),\psi_{t}(t) \rangle.
\end{equation}

The variation $\delta_{\tau}\eta_t(t)$ is decomposed as
\begin{equation}\label{eq:delta_eta_decomposition}
\begin{split}
\delta_{\tau}\eta_{t}(t)&=\big(\eta_{t+\tau}(t)-\eta_{t+\tau}(t+\tau)\big)+\big(\eta_{t+\tau}(t+\tau)-\eta_{t}(t)\big)\\
&=\big(\eta_{t+\tau}(t)-\eta_{t+\tau}(t+\tau)\big)+\big(e^{g(t+\tau)}\mathcal{B}-e^{g(t)}\mathcal{B}\big).
\end{split}
\end{equation}
Using the mild formulation of the coupled Boltzmann equations, there is
\begin{equation*}
\begin{split}
\eta_{t+\tau}(t)-\eta_{t+\tau}(t+\tau)=&S_{-\tau}e^{g(t+\tau)}\mathcal{B}-e^{g(t+\tau)}\mathcal{B}-\int_t^{t+\tau}S_{-(s-t)}\eta_{t+\tau}\phi ds\\
&+\int_{t}^{t+\tau} S_{-(s-t)}\mathcal{Q}_{\psi_{t+\tau}}(\eta_{t+\tau},\eta_{t+\tau}) ds.
\end{split}
\end{equation*}
Based on the continuity of $(\psi_t(s),\eta_t(s))$ under the $L_{\beta}^{\infty}$ norm (Lemma \ref{lem:order_of_variation_2}), the continuity of $\phi$ under the $L_{x,v}^{\infty}$ norm (assumption \eqref{assump:main_3} or \eqref{assump:main_3_th2}), as well as the continuity of the transport semigroup in $L_{x,v}^2$, we obtain the following convergence in $L_{x,v}^2$
\begin{equation}
\begin{split}
&\lim_{\tau\rightarrow +\infty}\frac{\eta_{t+\tau}(t+\tau)-\eta_{t+\tau}(\tau)}{\tau}\\
&\qquad\quad =\lim_{\tau\rightarrow +\infty}\Big(v\cdot\nabla_xe^{g(t+\tau)}\mathcal{B}-\tau^{-1}\int_t^{t+\tau}\eta_{t+\tau}\phi ds+\tau^{-1}\int_{t}^{t+\tau} \mathcal{Q}_{\psi_{t+\tau}}(\eta_{t+\tau},\eta_{t+\tau}) ds\Big)\\
&\qquad\quad=v\cdot\nabla_xe^{g(t+\tau)}\mathcal{B}-\eta_{t}(t)\phi(t)+\mathcal{Q}_{\psi_{t}(t)}(\eta_{t}(t),\eta_{t}(t)).
\end{split}
\end{equation}
Combine the equation above with \eqref{eq:delta_eta_decomposition}, we have
\begin{equation}\label{eq:tau_diff_2}
\lim_{\tau\rightarrow 0}\frac{\delta_{\tau}\eta_t(t)}{\tau}=D_se^{g(t)}\mathcal{B}-D_sg(t)e^{g(t)}\mathcal{B}+\mathcal{Q}_{\psi_{t}(t)}(\eta_{t}(t),\eta_{t}(t))=\mathcal{Q}_{\psi_{t}(t)}(\eta_{t}(t),\eta_{t}(t)).
\end{equation}
Equations \eqref{eq:tau_diff_1} and \eqref{eq:tau_diff_2} together conclude the proof of the lemma.
\end{proof}

\section{Uniform Boundedness, Long-time Behaviour, and Stationary Solutions}\label{sec:construct_solution}
In Theorem \ref{th:justification_mild} we have constructed a solution $\widehat{\mathcal{I}}(t,g)$ of the Hamilton-Jacobi equation. In this section we will give a description of its various properties.  Recall that $\langle \cdot,\cdot\rangle$ refers to the inner product in $L^2(\mathds{T}_x^d\times \mathds{R}_v^d)$, and $\dl\cdot,\cdot\dr$ refers to the inner product in $L^2([0,t]\times \mathds{T}_x^d\times\mathds{R}_v^d)$, with given terminal time $t$.  

According to Theorem \ref{th:justification_mild}, the functional $\widehat{\mathcal{I}}(t,g)$ defined as
\begin{equation*}
\widehat{\mathcal{I}}(t,g):=\langle f^0,\eta_t(0)-1 \rangle+\dl D_s\eta_t(s),\psi_t(s)\dr-\dl\eta_t(s),\psi_t(s)\phi(s)\dr+\int_0^t\mathcal{H}'\Big(\psi_t(s),\eta_t(s)\Big)ds.
\end{equation*}
is a mild solution of the Hamilton-Jacobi equation. The Hamiltonian $\mathcal{H}'$ is defined as
\begin{equation*}
\mathcal{H}'(\psi,\eta):=\frac{1}{2}\int \eta(v)\eta(v_*)\Big(\psi(v')\psi(v_*')-\psi(v)\psi(v_*)\Big)\big((v_*-v)\cdot\omega\big)_+d\omega dv_*dvdx,
\end{equation*}
Here $\big(\psi_t(s),\eta_t(s)\big)$ is the mild solution of the coupled Boltzmann equations \eqref{eq:transformed_evolution} during the time interval $[0,t]$. 
\subsection{Functional Solution for Theorem \ref{th:main_theorem_1}}
\begin{theorem}\label{th:global_HJ_solution_sigma1}
Under the assumptions \eqref{assump:main_1}-\eqref{assump:main_3}, the functional solution $\widehat{\mathcal{I}}(t,g)$ constructed in Theorem \ref{th:justification_mild} of the Hamilton-Jacobi equation is uniformly bounded for arbitrary time $t\geq 0$ and function $g$ satisfying the assumptions \eqref{assump:main_2}-\eqref{assump:main_3}.
\end{theorem}
\begin{proof}
First we perform an integration by parts
\begin{equation*}
\begin{split}
\dl D_s\eta_t(s),\psi_t(s)\dr&=\dl \eta_t(s),-D_s\psi_t(s)\dr+\langle \eta_t(t),\psi_t(t) \rangle-\langle \eta_t(0),\psi_t(0)\rangle\\
&=\dl \eta_t(s),-D_s\psi_t(s)\dr+\langle \eta_t(t),\psi_t(t) \rangle-\langle \eta_t(0),f^0\rangle,
\end{split}
\end{equation*}
where in the second equality we have used the fact that $\psi_t(0)=f^0$. The integration by parts can be justified in the spirit of Lemma \ref{lem:integration_by_parts}.

This integration by parts, along with the fact that $\int f^0dxdv=1$, implies 
\begin{equation*}
\widehat{\mathcal{I}}(t,g)=-1+\langle \eta_t(t),\psi_t(t) \rangle+\dl\eta_t(s),-D_s\psi_t(s) \dr-\dl \eta_t(s),\psi_t(s)\phi_t(s) \dr+\int_0^t\mathcal{H}'\Big(\psi_t(s),\eta_t(s)\Big)ds.
\end{equation*}
According to the definition of $\mathcal{H}'$, we have
\begin{equation}\label{eq:hamiltonian_derivative_equality}
\mathcal{H}'\Big(\psi_t(s),\eta_t(s)\Big)=\frac{1}{2}\Big\langle \eta_t(s),\mathcal{Q}_{\eta_t(s)}\big(\psi_t(s),\psi_t(s)\big)\Big\rangle,
\end{equation}
thus
\begin{equation*}
\begin{split}
&\widehat{\mathcal{I}}(t,g)\\
=&-1+\langle \eta_t(t),\psi_t(t) \rangle+\dl\eta_t(s),-D_s\psi_t(s) \dr-\dl \eta_t(s),\psi_t(s)\phi_t(s) \dr+\frac{1}{2}\Bdl\eta_t(s),\mathcal{Q}_{\eta_t(s)}\big(\psi_t(s),\psi_t(s)\big)\Bdr\\
=&-1+\langle \eta_t(t),\psi_t(t) \rangle+\Bdl\eta_t(s),-D_s\psi_t(s)-\psi_t(s)\phi(s)+\mathcal{Q}_{\eta_t(s)}\big(\psi_t(s),\psi_t(s)\big)\Bdr\\
&\hspace{9cm}-\frac{1}{2}\Bdl\eta_t(s),\mathcal{Q}_{\eta_t(s)}\big(\psi_t(s),\psi_t(s)\big)\Bdr.
\end{split}
\end{equation*}
Since $(\psi_t(s),\eta_t(s))$ is the mild solution of the coupled Boltzmann equations \eqref{eq:coupled_Boltzmann_sym_intro}, we have
\begin{equation}\label{eq:functional_decomposition_1}
\widehat{\mathcal{I}}(t,g)=-1+\underbrace{\langle \eta_t(t),\psi_t(t) \rangle}_{\textup{\blue{(I)}}}-\frac{1}{2}\underbrace{\Bdl\eta_t(s),\mathcal{Q}_{\eta_t(s)}\big(\psi_t(s),\psi_t(s)\big)\Bdr}_{\textup{\blue{(II)}}}.
\end{equation}
Recall the perturbations $\psi_{t,p}$ and $\eta_{t,p}$ introduced in \eqref{eq:perturbation_decomposition}. The estimate of term (I) is relatively immediate with $L_{\beta}^{\infty}$ being a stronger topology than $L^2$ due to $\beta>4$
\begin{equation*}
\begin{split}
\int \eta_t(t)\psi_t(t)dvdx\leq C \lVert \eta_t(t)\rVert_{L_{\beta}^\infty}\lVert \psi_t(t)\rVert_{L_\beta^\infty}&\leq C \Big(\lVert \mathcal{G}\rVert_{L_{\beta}^\infty}+\lVert \eta_{p,t}(t)\rVert_{L_{\beta}^\infty}\Big)\Big(\lVert \mathcal{G}\rVert_{L_{\beta}^\infty}+\lVert \psi_{p,t}(t)\rVert_{L_\beta^\infty}\Big)\\
&\leq C \Big(\lVert \mathcal{G}\rVert_{L_{\beta}^\infty}+a_*\Big)\Big(\lVert \mathcal{G}\rVert_{L_{\beta}^\infty}+a_*\Big).
\end{split}
\end{equation*}
To estimate the term (II), we first perform the perturbation decomposition for $\eta_t$
\begin{equation}\label{eq:functional_decomposition_2}
\textup{(II)}=\underbrace{\Bdl\mathcal{G},\mathcal{Q}_{\eta_t(s)}\big(\psi_t(s),\psi_t(s)\big)\Bdr}_{\textup{\blue{(II.1)}}}+\underbrace{\Bdl\eta_{t,p}(s),\mathcal{Q}_{\eta_t(s)}\big(\psi_t(s),\psi_t(s)\big)\Bdr}_{\textup{\blue{(II.2)}}}.
\end{equation}
To estimate the (II.1) term, we first notice that Lemma \ref{lem:multiplicative_nonlinearity} implies $\mathcal{Q}_{\eta_t(s)}(\psi_t(s),\psi_t(s))\in L_{\beta-1}^{\infty}$. Consequently the integrand in (II.1) is absolutely integrable. The fact that $(\psi_t,\eta_t)$ is the mild solution of the coupled Boltzmann equations \eqref{eq:transformed_evolution}
\begin{equation*}
\psi_t(t)=S_t\psi_t(0)-\int_0^tS_{t-s}\psi_t(s)\phi_t(s)ds+\int_0^tS_{t-s}\mathcal{Q}_{\eta_t(s)}\big(\psi_t(s),\psi_t(s)\big)ds.
\end{equation*}
Using the fact that $\mathcal{G}$ is a function independent of time and space, we have
\begin{equation*}
\begin{split}
\int_0^t\Big(\int \mathcal{G}\mathcal{Q}_{\eta_t(s)}\big(\psi_t(s),\psi_t(s)\big)dvdx\Big)ds&=\int_0^t\Big(\int \mathcal{G}S_{t-s}\mathcal{Q}_{\eta_t(s)}\big(\psi_t(s),\psi_t(s)\big)dvdx\Big)ds\\
&=\int \mathcal{G}\Big(\int_0^tS_{t-s}\mathcal{Q}_{\eta_t(s)}\big(\psi_t(s),\psi_t(s)\big)ds\Big)dvdx.
\end{split}
\end{equation*}
Since by assumption \eqref{assump:main_3} $\phi\equiv 0$, the equation above eventually yields the (II.1) term is equal to
\begin{equation}\label{eq:functional_decomposition_3}
\textup{(II.1)}=\int \mathcal{G}\Big(\psi_t(t)+\int_0^tS_{t-s}\psi_t(s)\phi_t(s)ds-S_t\psi_t(0)\Big)dvdx=\int \mathcal{G}\psi_t(t)dvdx-\int \mathcal{G}\psi_t(0)dvdx.
\end{equation}
For the (II.2) term, it has the upper bound as
\begin{equation}\label{eq:II.2_estimate}
\begin{split}
\textup{(II.2)}=&\int_0^t\Big[\int \eta_{t,p}(s)\mathcal{Q}_{\eta_t(s)}\big(\psi_t(s),\psi_t(s)\big)dvdx\Big]ds\\
\leq &C\int_0^t\lVert(1+|v|)\eta_{p,t}(s)\rVert_{L_{\beta-1}}\lVert(1+|v|)^{-1}\mathcal{Q}_{\eta_t(s)}\big(\psi_t(s),\psi_t(s)\big)\rVert_{L_{\beta}^{\infty}}ds\\
\leq &C\int_0^t \lVert\eta_{p,t}(s)\rVert_{L_{\beta}^{\infty}}\lVert\eta_{t}(s)\rVert_{L_{\beta}^{\infty}}\lVert\psi_{t}(s)\rVert_{L_{\beta}^{\infty}}\lVert\psi_{t}(s)\rVert_{L_{\beta}^{\infty}}ds,
\end{split}
\end{equation}
where in the last inequality we have used Lemma \ref{lem:multiplicative_nonlinearity}. This upper bound can be further written as
\begin{equation*}
\begin{split}
\textup{(II.2)}\leq &C\lVert\eta_{t,p}^{\reverse}\rVert_{P_{\beta}^{\sigma}}\int_0^t \big(1+(t-s)\big)^{-\sigma}\lVert\eta_{t}(s)\rVert_{L_{\beta}^{\infty}}\lVert\psi_{t}(s)\rVert_{L_{\beta}^{\infty}}\lVert\psi_{t}(s)\rVert_{L_{\beta}^{\infty}}ds\\
\leq& Ca_*\big(\lVert \mathcal{G}\rVert_{L_{\beta}^{\infty}}+a_*\big)^3.
\end{split}
\end{equation*}
Now it has been proved that there is a uniform bound for all terms in the decomposition of $\widehat{\mathcal{I}}$. This yields the uniform boundedness of $\widehat{\mathcal{I}}(t,g)$, and concludes the proof of the theorem.
\end{proof}

With some additional effort, we can prove the proposition below for the mild solution $\widehat{\mathcal{I}}(t,g)$ in Theorem \ref{th:global_HJ_solution_sigma1}.
\begin{proposition}\label{prop:time_limit}
\textup{\textbf{[Long-Time Behaviour]}} Fix $\widehat{g}$ as a function in $C_{x,v}^0$ which is time independent. For arbitrary function $g$ satisfying the assumptions \eqref{assump:main_2}-\eqref{assump:main_3} and such that $g(t)=\widehat{g}$, there is
\begin{equation}
\lim_{t\rightarrow +\infty}\widehat{\mathcal{I}}(t,g)=\widehat{\mathcal{I}}_{\infty}(\widehat{g}),
\end{equation}
with
\begin{equation}\label{eq:the_time_limit}
\widehat{\mathcal{I}}_{\infty}(\widehat{g}):=-1+\langle e^{\widehat{g}},M\rangle.
\end{equation}
\end{proposition}
In this framework the function $g$ is fixed at terminal time $g(t)=\widehat{g}$. According to assumption \eqref{assump:main_3}, there is $D_sg\equiv 0$. Consequently the function $g$ is determined on the whole time interval $[0,t]$ by $g(t)=\widehat{g}$.
\begin{proof}
According to equations \eqref{eq:functional_decomposition_1}, \eqref{eq:functional_decomposition_2}, and \eqref{eq:functional_decomposition_3} in the proof of Theorem \ref{th:global_HJ_solution_sigma1}, the functional $\widehat{\mathcal{I}}(t,g)$ can be rewritten as
\begin{equation*}
\widehat{\mathcal{I}}(t,g)=-1+\langle \eta_t(t),\psi_t(t) \rangle+\frac{1}{2}\langle \mathcal{G}, \psi_t(0)-\psi_t(t)\rangle+\frac{1}{2}\underbrace{\Bdl\eta_{t,p}(s),\mathcal{Q}_{\eta_t(s)}\big(\psi_t(s),\psi_t(s)\big)\Bdr}_{\textup{\blue{(II.2)}}}.
\end{equation*}
We want to prove the (II.2) term converge to $0$ as $t\rightarrow +\infty$. This can be achieved by the decomposition of the collision operator
\begin{equation*}
\begin{split}
\textup{(II.2)}=&2\Bdl\eta_{t,p}(s),\mathcal{Q}_{\eta_t(s)}\big(\psi_{t,p}(s),\mathcal{G}\big)\Bdr+\Bdl\eta_{t,p}(s),\mathcal{Q}_{\eta_t(s)}\big(\psi_{t,p}(s),\psi_{t,p}(s)\big)\Bdr\\
\leq& 2\int_0^t \lVert\eta_{p,t}(s)\rVert_{L_{\beta}^{\infty}}\lVert\eta_{t}(s)\rVert_{L_{\beta}^{\infty}}\lVert\psi_{t,p}(s)\rVert_{L_{\beta}^{\infty}}\lVert\psi_{t}(s)\rVert_{L_{\beta}^{\infty}}ds\\
&+\int_0^t \lVert\eta_{p,t}(s)\rVert_{L_{\beta}^{\infty}}\lVert\eta_{t}(s)\rVert_{L_{\beta}^{\infty}}\lVert\psi_{t,p}(s)\rVert_{L_{\beta}^{\infty}}\lVert\psi_{t,p}(s)\rVert_{L_{\beta}^{\infty}}ds\\
\leq &C\int_0^t (1+s)^{-\sigma}\big(1+(t-s)\big)^{-\sigma}ds\leq C(1+t)^{-\sigma}\rightarrow 0,\quad t\rightarrow +\infty,
\end{split}
\end{equation*}
where we have used Lemma \ref{lem:convolution_inequality} and the fact according to Theorem \ref{th:global_fixed_point_P}
\begin{equation*}
\lVert \psi_{t,p}\rVert_{P_{\beta}^{\sigma}}<a_*,\quad \lVert \eta_{t,p}^{\reverse}\rVert_{P_{\beta}^{\sigma}}<a_*.
\end{equation*}
In Theorem \ref{th:global_fixed_point_P}, the bound $\lVert \psi_{t,p}\rVert_{P_{\beta}^{\sigma}}<a_*$ is uniform for all $t>0$. According to the definition \eqref{eq:forward_P_norm} of $P_{\beta}^{\infty}$, there is
\begin{equation*}
(1+t)^{\sigma}\lVert \psi_{t,p}(t)\rVert_{L_{\beta}^{\infty}}<a_*.
\end{equation*}

Consequently as $t\rightarrow +\infty$, there is $\psi_t(t)\rightarrow \mathcal{G}$ in the $L_{\beta}^{\infty}$ norm. This implies
\begin{equation*}
\lim_{t\rightarrow +\infty}\langle \eta_t(t),\psi_t(t)\rangle=\lim_{t\rightarrow +\infty}\langle e^{g(t)}\mathcal{B},\mathcal{G}\rangle=\lim_{t\rightarrow +\infty}\langle e^{g(t)},M\rangle=\langle e^{\widehat{g}},M\rangle,
\end{equation*}
where we have used the fact that $\eta_t(t)=e^{g(t)}$ and $g(t)=\widehat{g}$.

These imply
\begin{equation*}
\lim_{t\rightarrow +\infty}\widehat{\mathcal{I}}(t,g)=-1+\langle e^{\widehat{g}},M\rangle+\frac{1}{2}\langle \mathcal{G},\mathcal{B}^{-1}f^0-\mathcal{G}\rangle.
\end{equation*}
The orthogonality assumption \eqref{assump:main_1} means $\langle \mathcal{G},\mathcal{B}^{-1}f^0-\mathcal{G} \rangle=0$, which concludes the proof of the lemma.
\end{proof}
The relation between the result above and the Schrödinger problem has been discussed in Subsection \ref{sec:implications}. 

In addition, we can further prove the long-time limit of $\widehat{\mathcal{I}}(t,g)$ given in Proposition \ref{prop:time_limit}, is also a non-trivial stationary solution of the Hamilton-Jacobi equation.
\begin{proposition}\label{prop:stationary_solution}
\textup{\textbf{[Stationary Solution]}} The functional $\widehat{\mathcal{I}}_{\infty}(\widehat{g})$ in \eqref{eq:the_time_limit} is a stationary solution of the Hamilton-Jacobi equation.
\end{proposition}
\begin{proof}
It is sufficient to verify 
\begin{equation*}
\mathcal{H}\Big(\frac{\p \widehat{\mathcal{I}}_{\infty}(\widehat{g})}{\p \widehat{g}},\widehat{g}\Big)=0.
\end{equation*}
By the definition \eqref{eq:the_time_limit} of $\widehat{\mathcal{I}}_{\infty}(\widehat{g})$, the derivative with respect to $\widehat{g}$ is
\begin{equation*}
\frac{\p \widehat{\mathcal{I}}_{\infty}(\widehat{g})}{\p \widehat{g}}=\frac{\p \langle e^{\widehat{g}},M\rangle}{\p \widehat{g}}=e^{\widehat{g}}M.
\end{equation*}
Then according to the definition \eqref{eq:functional_hamiltonian} of $\mathcal{H}$, we have
\begin{equation}
\begin{split}
&\mathcal{H}\big(e^{\widehat{g}}M,\widehat{g}\big)\\
=&\frac{1}{2}\int M(v)M(v_*)e^{\widehat{g}(v)+\widehat{g}(v_*)}\Big(e^{\widehat{g}(v')+\widehat{g}(v_*')-\widehat{g}(v)-\widehat{g}(v_*)}-1\Big)\big((v_*-v)\cdot\omega\big)_+d\omega dv_* dvdx\\
=&\frac{1}{2}\int M(v)M(v_*)\Big(e^{\widehat{g}(v')+\widehat{g}(v_*')}-e^{\widehat{g}(v)+\widehat{g}(v_*)}\Big)\big((v_*-v)\cdot\omega\big)_+d\omega dv_* dvdx=0,
\end{split}
\end{equation}
where the last inequality is by the change of variables from $(v,v_*,\omega)$ to $(v',v_*',-\omega)$. This concludes the proof.
\end{proof}

The two propositions above together states that, the solution $\widehat{\mathcal{I}}(t,g)$ converges to a stationary solution $\widehat{\mathcal{I}}_{\infty}(\widehat{g})$ as $t\rightarrow +\infty$. The stationary solution $\widehat{\mathcal{I}}_{\infty}(\widehat{g})=\langle e^{\widehat{g}},M \rangle-1$ is the cumulant generating functional of a random gas with Poisson-distributed total number, and i.i.d. distribution of the variables $(x,v)$.

\subsection{Functional Solution for Theorem \ref{th:main_theorem_2}}
\begin{theorem}\label{th:global_HJ_solution_sigma0}
Under the assumptions \eqref{assump:main_1_th2}-\eqref{assump:main_3_th2}, the functional solution $\widehat{\mathcal{I}}(t,g)$ constructed in Theorem \ref{th:justification_mild} of the Hamilton-Jacobi equation is uniformly bounded for arbitrary time $t\geq 0$ and function $g$ satisfying Assumptions \eqref{assump:main_2_th2}-\eqref{assump:main_3_th2}.
\end{theorem}
\begin{proof}
In the proof of Theorem \ref{th:global_HJ_solution_sigma1}, we have decomposed the functional $\widehat{\mathcal{I}}$ into several terms: first it is decomposed as \eqref{eq:functional_decomposition_1}, then the (II) term is decomposed into the summation of (II.1) term and (II.2) term in \eqref{eq:functional_decomposition_2}. The same decomposition applies here. Except for the (II.2) term, the estimate of the others is exactly the same as in the proof of theorem \ref{th:global_HJ_solution_sigma0}. Thus we only detail the estimate of the (II.2) term here.

In equation \eqref{eq:II.2_estimate}, it has been proved that
\begin{equation}
\begin{split}
\textup{(II.2)}=&\int_0^t\Big[\int \eta_{t,p}(\tau)\mathcal{Q}_{\eta_t(\tau)}\big(\psi_t(\tau),\psi_t(\tau)\big)dvdx\Big]d\tau\\
\leq &C\int_0^t \lVert\eta_{t,p}(\tau)\rVert_{L_{\beta}^{\infty}}\lVert\eta_{t}(\tau)\rVert_{L_{\beta}^{\infty}}\lVert\psi_{t}(\tau)\rVert_{L_{\beta}^{\infty}}\lVert\psi_{t}(\tau)\rVert_{L_{\beta}^{\infty}}d\tau
\end{split}
\end{equation}
According to the definition \eqref{eq:forward_E_norm} of the $E_{\beta}^0$ and the $E_{\beta}^{-\sigma}$ norms, we get
\begin{equation*}
\begin{split}
\textup{(II.2)}&\leq Ce^{\sigma t} \lVert\eta_{t,p}^{\reverse}\rVert_{E_{\beta}^{-\sigma}} \lVert\eta_{t}\rVert_{E_{\beta}^{0}}\lVert\psi_{t}\rVert_{E_{\beta}^{0}}\lVert\psi_{t}\rVert_{E_{\beta}^{0}}\leq Ce^{\sigma t}a_*e^{-\sigma t}\big(\lVert \mathcal{G}\rVert_{L_{\beta}^{\infty}}+a_*\big)^3,
\end{split}
\end{equation*}
where we have used the condition $\lVert \eta_{t,p}^{\reverse} \rVert_{E_{\beta}^{-\sigma}}<a_*e^{-\sigma t}$.

It shows the (II.2) term is also uniformly bounded. This concludes the proof of the uniform boundedness of $\widehat{\mathcal{I}}(t,g)$.
\end{proof}

\appendix
\section{Decomposition of Semigroup and Relevant Estimates}\label{app:decomposition_of_semigroup}
By Definition \ref{def:fixed_point}, the operator $B^+$ consists of a transport operator $-v\cdot\nabla_x$, and a linearized collision operator $2\mathcal{Q}_{\mathcal{G}}(\cdot,\mathcal{G})$. The operator $2\mathcal{Q}_{\mathcal{\mathcal{G}}}(\cdot,\mathcal{G})$ can be decomposed as the summation of a frequency multiplier $-\nu$ and a convolution operator $K$. This decomposition is initially due to \cite{Grad_1965} 
\begin{equation*}
2\mathcal{Q}_{\mathcal{\mathcal{G}}}(\cdot,\mathcal{G})=-\nu+K,
\end{equation*}
with
\begin{equation}\label{eq:multiplier_and_K}
\begin{split}
&\nu f(v)=\int_{\mathds{R}^d}\int_{\mathds{S}^{d-1}}\big((v_*-v)\cdot\omega\big)_+f(v)\mathcal{G}^2(v_*)d\omega dv_*,\\
&Kf(v)=\int_{\mathds{R}^d}\int_{\mathds{S}^{d-1}}\big((v_*-v)\cdot\omega\big)_+\Big(f(v')\mathcal{G}(v_*')+f(v_*')\mathcal{G}(v')-f(v_*)\mathcal{G}(v)\Big)\mathcal{G}(v_*)d\omega dv_*.
\end{split}
\end{equation}
The operator $K$ can also be written using the related transition kernel, which is given explicitly on Page 19 of \cite{ukai_note}. The following lemma about the convolution operator $K$ is classical \cite{cercignani_1994,Grad_1965,Ukai_1986}.
\begin{lemma}\label{lem:smoothing_effect}
The operator $K$ is a self-adjoint compact operator on $L^2$. For any $2\leq p\leq r\leq \infty$, it is also a bounded operator from $L^p$ to $L^r$. If $\beta\geq 0$, then it is a bounded operator from $L_{\beta}^{\infty}$ to $L_{\beta}^{\infty}$.
\end{lemma}
For a detailed proof of the following lemma, the reader may see Section 2.2. of \cite{ukai_note}.
\begin{lemma}\label{lem:L_2_decay}
\textbf{\textup{[$L^2$-Decay Estimate] }}
The operator $\rB$ (resp. $\lB$) in Definition \ref{def:fixed_point} generates a strongly continuous semigroup $e^{s\rB}$ (resp. $e^{s\lB}$) on $L^2(\mathds{T}_x^d\times \mathds{R}_v^d)$. Both semigroups decay exponentially in the $L_{x,v}^2$-norm if the initial data is orthogonal to the kernel $\mathcal{K}$: there exists constants $\nu_*>0$ and $C>0$ such that if $f\in\mathcal{K}^{\perp}$, then
\begin{equation*}
\lVert e^{s\rB}f\rVert_{L_{x,v}^2}\leq Ce^{-\nu_* s}\lVert f\rVert_{L_{x,v}^2},\quad \lVert e^{s\lB}f\rVert_{L_{x,v}^2}\leq Ce^{-\nu_* s}\lVert f\rVert_{L_{x,v}^2}.
\end{equation*}
If the function $f$ belongs to the kernel $\mathcal{K}$, then we have $e^{s\rB}f=e^{s\lB}f=f$.
\end{lemma}
Since $B^+$ is a bounded perturbation of $A^+:=-v\cdot\nabla_x-\nu$, according to Corollary 1.7 in Page 119 of \cite{Axler_Ribet_2005}, the semigroup $e^{sB^+}$ can be written as
\begin{equation}\label{eq:bootstrap}
e^{sB^+}=e^{sA^+}+\Big(e^{sA^+}K\Big)\ast e^{sB^+}.
\end{equation}
Here $\ast$ refers to the convolution over $[0,s]$,
\begin{equation*}
\Big(e^{sA^+}K\Big)\ast e^{sB^+} :=\int_0^s e^{(s-\tau)A^+}Ke^{\tau B^+}.  
\end{equation*}
Iterate this and we will have
\begin{equation*}
\begin{split}
e^{sB^+}&=e^{sA^+}+\Big(e^{sA^+}K\Big)\ast e^{sA^+}+\Big(e^{sA^+}K\Big)\ast \Big(e^{sA^+}K\Big)\ast e^{sA^+}+...+\Big(e^{sA^+}K\Big)^{\ast N}e^{sB^+}\\
&=e^{sA^+}+\sum_{j=1}^{N-1}\Big(e^{sA^+}K\Big)^{\ast j}\ast e^{sA^+}+\Big(e^{sA^+}K\Big)^{\ast N}\ast e^{sB^+}.
\end{split}
\end{equation*}
For the decomposition \eqref{eq:decomposition_semigroup} of $e^{sB^+}=\mathcal{D}_1^+(s)+\mathcal{D}_2^+(s)$, we define $\mathcal{D}_1^+(s)$ and $\mathcal{D}_2^+(s)$ as
\begin{equation}\label{eq:D_1 and D_2 +}
\mathcal{D}_1^+(s):=e^{sA^+},\quad D_2^+(s):=\sum_{j=1}^{N-1}\Big(e^{sA^+}K\Big)^{\ast j}\ast e^{sA^+}+\Big(e^{sA^+}K\Big)^{\ast N}\ast e^{sB^+}.
\end{equation}
The operator $\mathcal{D}_1^+(s)$ has the explicit expression
\begin{equation}
\mathcal{D}_1^{+}(s)f(x,v)=e^{-\nu(v)s}f(x-sv,v).
\end{equation}
The operator $\mathcal{D}_2^+$ has a decay estimate as a map from $L_{\beta-1}^{\infty}$ to $L_{\beta}^{\infty}$ (see Lemma \ref{lem:L_beta_decay}) due to the smoothing effect (Lemma \ref{lem:smoothing_effect}) of $K$.

The analysis above is also true for the backward component, where for the decomposition \eqref{eq:decomposition_semigroup} of $e^{sB^-}=\mathcal{D}_1^-(s)+\mathcal{D}_2^-(s)$, we define $\mathcal{D}_1^-(s)$ and $\mathcal{D}_2^-(s)$ as
\begin{equation}\label{eq:D_1 and D_2 -}
\mathcal{D}_1^-(s):=e^{sA^-},\quad \mathcal{D}_2^-(s):=\sum_{j=1}^{N-1}\Big(e^{sA^-}K\Big)^{\ast j}\ast e^{sA^-}+\Big(e^{sA^-}K\Big)^{\ast N}\ast e^{sB^-}.
\end{equation}
The operator $\mathcal{D}_1^{-}(s)$ is explicitly written as
\begin{equation*}
\mathcal{D}_1^{-}(s)f(x,v)=e^{-\nu(v)s}f(x+sv,v).
\end{equation*}
The decomposition has been well established in the literature \cite{Ukai_2006}, whose modification gives the proof of Lemma \ref{lem:L_beta_decay}.
\begin{proof}[Proof of Lemma \ref{lem:L_beta_decay}]
We only give the proof for the forward component $\mathcal{D}_2^+(s)$. The proof for $\mathcal{D}_2^-(s)$ defined in \eqref{eq:D_1 and D_2 -} is the same.

\underline{The case of $f\in\mathcal{K}^{\perp}$:} Using Lemma \ref{lem:smoothing_effect} with $p=2$ and $r=\infty$, we can prove $\big(e^{sA^+}K\big)\ast e^{sB^+}$ is a bounded operator from $L^2$ to $L_{0}^\infty=L^{\infty}$, with $Ce^{-\nu_*s}$ as the upper bound for the operator norm
\begin{equation}
\Big\lVert\int_0^{s} e^{(s-\tau)A^+}Ke^{\tau B^+}fd\tau\Big\rVert_{L_{0}^{\infty}}\leq C\int_0^{s}e^{-\nu_*(s-\tau)}e^{-\nu_*\tau}\lVert f\rVert_{L^2}\leq Ce^{-\nu_*s}\lVert f\rVert_{L^2}.
\end{equation}
Having an additional convolution with $e^{sA^+}K$, the operator $\big(e^{sA^+}K\big)^{\ast 2}\ast e^{sB^+}$ is a bounded operator from $L^2$ to $L_{1}^{\infty}$. Iterating this bootstrap argument and choosing $N=\lceil \beta\rceil$, we have $\big(e^{sA^+}K\big)^{\ast N}\ast e^{sB^+}$ is a bounded operator from $L^2$ to $L_{\beta}^{\infty}$ with $\beta>4$, also with $Ce^{-\nu_*s}$ as the upper bound for the operator norm. This implies
\begin{equation*}
\lVert \big(e^{sA}K\big)^{\ast N}\ast e^{sB}f\rVert_{L_{\beta}^{\infty}}\leq Ce^{-\nu_*s}\lVert f\rVert_{L^2}\leq Ce^{-\nu_*s}\lVert f\rVert_{L_{\beta-1}^{\infty}}=Ce^{-\nu_*s}\lVert (1+|v|)^{-1}f\rVert_{L_{\beta}^{\infty}}.
\end{equation*}
Using the explicit expression of $e^{sA^+}$ and the smoothing effect of $K$, we can prove for any $j\geq 1$ that $\Big(e^{sA^+}K\Big)^{\ast j}\ast e^{sA^+}$ has the decay estimate
\begin{equation}\label{eq:convolution_app}
\Big\lVert\Big(e^{sA^+}K\Big)^{\ast j}\ast e^{sA^+}f\Big\rVert_{L_{\beta}^{\infty}}\leq Ce^{-\nu_* s}\lVert (1+|v|)^{-1}f\rVert_{L_{\beta}^{\infty}}.
\end{equation}
\underline{The case of $f\in\mathcal{K}$:} Now the estimate \eqref{eq:convolution_app} is still true since it only depends on the explicit expression of $e^{sA^+}$ and the smoothing effect of $K$. For the other term in $\mathcal{D}_2^+$, it becomes
\begin{equation*}
\lVert \big(e^{sA}K\big)^{\ast N}\ast e^{sB}f\rVert_{L_{\beta}^{\infty}}=\lVert \big(e^{sA}K\big)^{\ast N}\ast \textup{Id}f\rVert_{L_{\beta}^{\infty}}\leq C\lVert (1+|v|)^{-1}f\rVert_{L_{\beta}^{\infty}}.
\end{equation*}
This concludes the proof of the lemma.
\end{proof}
Next we prove the continuity of $e^{\tau B^+}$ (resp. $e^{\tau B^-}$) with respect to the forward initial perturbation $f^0\mathcal{B}^{-1}-\mathcal{G}$ (resp. the backward terminal perturbation $e^{g(t)}\mathcal{B}-\mathcal{G}$). This lemma is crucial to the proof of the continuity of $(\psi_t(s),\eta_t(s))$ (see Lemma \ref{lem:order_of_variation_2}).
\begin{lemma}\label{lem:beta continuity}
\textup{\textbf{[Continuity Estimate]}} Under assumptions \eqref{assump:main_1}-\eqref{assump:main_3} or \eqref{assump:main_1_th2}-\eqref{assump:main_3_th2}, for any parameter $\beta>4$ we have
\begin{equation}\label{eq:beta continuity}
\begin{split}
&\lVert e^{\tau B^+}(f^0\mathcal{B}^{-1}-\mathcal{G})-(f^0\mathcal{B}^{-1}-\mathcal{G})\rVert_{L_{\beta}^{\infty}}\leq C(f_0)\tau,\\
&\lVert e^{\tau B^-}(e^{g(t)}\mathcal{B}-\mathcal{G})-(e^{g(t)}\mathcal{B}-\mathcal{G})\rVert_{L_{\beta}^{\infty}}\leq C(g)\tau,
\end{split}
\end{equation}
with $C(f_0)$ being a constant dependent on $f_0$, and $C(g)$ being a constant dependent on $g$.
\end{lemma}
\begin{proof}
We only detail the proof of the second equation in \eqref{eq:beta continuity}, since the other one is the same. For simplicity in notation and also in accordance with the choice of terminal data, we write $G(x,v):=(e^{g(t)}\mathcal{B}-\mathcal{G})(x,v)$. Using the bootstrap argument \eqref{eq:bootstrap}, we have
\begin{equation}\label{eq:beta+2 continuity}
\begin{split}
&\lVert e^{\tau B^-}G-G\rVert_{L_{\beta}^{\infty}}\\
\leq&\lVert e^{\tau A^-}G-G\rVert_{L_{\beta}^{\infty}}+\int_0^\tau\lVert e^{u A^-}Ke^{(\tau-u)B^-}G\rVert_{L_{\beta}^{\infty}}du\\
=&\lVert G(x+\tau v,v)e^{-\nu(v)\tau}-G(x,v)\rVert_{L_{\beta}^{\infty}}+ \int_0^\tau\lVert e^{u A^-}Ke^{(\tau-u)B^-}G\rVert_{L_{\beta}^{\infty}}du
\end{split}
\end{equation}
The first term in the third line of \eqref{eq:beta+2 continuity} is controlled as
\begin{equation*}
\begin{split}
&\lVert G(x+\tau v,v)e^{-\nu(v)\tau}-G(x,v)\rVert_{L_{\beta}^{\infty}}\\
\leq& \lVert G(x+\tau v,v)\big(1-e^{-\nu(v)\tau }\big)\rVert_{L_{\beta}^{\infty}}+\lVert G(x+\tau v,v)-G(x,v)\rVert_{L_{\beta}^{\infty}}\\
\leq& \Big\lVert (1+|v|)G(x+\tau v,v)\frac{1-e^{-\nu(v)\tau}}{(1+|v|)}\Big\rVert_{L_{\beta}^{\infty}}+\lVert e^{g(t,x+\tau v,v)}\mathcal{B}(v)-e^{g(t,x,v)}\mathcal{B}(v)\rVert_{L_{\beta}^{\infty}}\\
\leq& C\lVert G\rVert_{L_{\beta+1}^{\infty}}\tau+C\lVert G\rVert_{L_{\beta+1}^{\infty}}\lVert g\rVert_{C_{t,x}^1}\tau,
\end{split}
\end{equation*}
where the term $C\lVert G\rVert_{L_{\beta+1}^{\infty}}\tau$ is due to
\begin{equation*}
\frac{1-e^{-\nu(v)\tau}}{(1+|v|)}\leq C\tau  ,
\end{equation*}
and the term $C\lVert G\rVert_{L_{\beta+1}^{\infty}}\lVert g\rVert_{C_{t,x}^1}\tau$ is because
\begin{equation*}
e^{g(t,x+\tau v,v)}\mathcal{B}(v)-e^{g(t,x,v)}\mathcal{B}(v)
=\int_0^{\tau}\Big[\big(v\cdot \nabla_xg\big)e^{g}\mathcal{B}\Big](t,x+uv,v)du.
\end{equation*}
The norm $\lVert g\rVert_{C_{t,x}^1}$ is finite due to assumption \eqref{assump:main_3} or \eqref{assump:main_3_th2}, where we have assumed $g$ has uniformly bounded derivatives in $t$ and $x$. The control of the second term in the third line of \eqref{eq:beta+2 continuity} is straightforward, since all the involved operators are bounded operators from $L_{\beta}^{\infty}$ to $L_{\beta}^{\infty}$
\begin{equation*}
\int_0^\tau\lVert e^{u A^-}Ke^{(\tau-u)B^-}G(x,v)\rVert_{L_{\beta}^{\infty}}du\leq C\lVert G\rVert_{L_{\beta}^{\infty}}\tau
\end{equation*}
This concludes the proof of the lemma.
\end{proof}
In the end of this appendix, we give the proof of Lemma \ref{lem:convolution_inequality} for completeness. The proof is elementary.
\begin{lemma}\label{lem:convolution_inequality}
Given $\sigma_1>1$ and $\sigma_2>1$, we have the following inequality for a convolution,
\begin{equation*}
\int_{0\leq s\leq t}\Big(1+(t-s)\Big)^{-\sigma_2}(1+s)^{-\sigma_1}ds\leq C (1+t)^{-\min\{\sigma_1,\sigma_2\}}
\end{equation*}
\end{lemma}
\begin{proof}
We split the integral
\begin{equation*}
\begin{split}
&\int_{0\leq s\leq t}\Big(1+(t-s)\Big)^{-\sigma_2}(1+s)^{-\sigma_1}ds\\
=&\int_{0\leq s\leq \frac{t}{2}}\Big(1+(t-s)\Big)^{-\sigma_2}(1+s)^{-\sigma_1}ds+\int_{\frac{t}{2}\leq s\leq t}\Big(1+(t-s)\Big)^{-\sigma_2}(1+s)^{-\sigma_1}ds\\
\leq &   (1+\frac{t}{2})^{-\sigma_2}\int_{0\leq s\leq \frac{t}{2}}(1+s)^{-\sigma_1}ds+(1+\frac{t}{2})^{-\sigma_1}\int_{\frac{t}{2}\leq s\leq t}\Big(1+(t-s)\Big)^{-\sigma_2}ds
\end{split}
\end{equation*}
This is further less than
\begin{equation*}
\begin{split}
...&\leq C(1+\frac{t}{2})^{-\sigma_2}+C(1+\frac{t}{2})^{-\sigma_1}\\
&\leq C(1+t)^{-\min\{\sigma_1,\sigma_2\}}
\end{split}
\end{equation*}
This concludes the proof.
\end{proof}

\bibliography{main.bib}
\bibliographystyle{plain}
\end{document}